\documentclass[11pt,a4paper,leqno]{amsart}

\usepackage[a4paper]{geometry}
\usepackage{graphicx}
\usepackage{amsmath}
\usepackage{setspace}
\usepackage[usenames,dvipsnames]{xcolor}
\usepackage[english]{babel}
\usepackage{mathrsfs,amssymb,bbm,amstext,amsthm,amsfonts}
\usepackage{graphicx,enumitem}
\usepackage[arrow, matrix, curve]{xy}
\usepackage[ansinew]{inputenc}

\usepackage[colorlinks=true,linkcolor=Red,citecolor=Green]{hyperref}

\usepackage{tikz,xifthen}
\usetikzlibrary{arrows,calc,decorations.pathmorphing,backgrounds,positioning,fit,petri,matrix}

\numberwithin{equation}{section}          

\newtheorem{thm}{Theorem}
\newtheorem*{thm*}{Theorem}
\numberwithin{thm}{section}

\newcommand{\rubrik}{}
\newtheorem{prop}[thm]{Proposition}
\newtheorem{cor}[thm]{Corollary}
\newtheorem{lem}[thm]{Lemma}

\theoremstyle{definition}

\newtheorem{defn}[thm]{Definition}
\newtheorem{ex}[thm]{Example}

\theoremstyle{remark}

\newtheorem{rem}[thm]{Remark}     

\theoremstyle{plain}

\DeclareRobustCommand{\SkipTocEntry}[5]{}

\newcommand{\dd} {{\mathrm{d}}}

\DeclareMathOperator{\codim} {codim}
\DeclareMathOperator{\rk} {rk}
\newcommand{\Fu} {{\mathcal{F}}}

\newcommand{\red} {{\mathrm{red}}}
\newcommand{\pr}{{\mathrm{pr}}}
\newcommand{\jap}[1] {\langle{#1}\rangle}

\newcommand{\id}{\mathrm{id}}

\newcommand{\wt}[1]{\widetilde{#1}}

\newcommand{\cO}{{\mathcal{O}}}

\newcommand{\BB}{{\mathbb{B}}}
\newcommand{\BBd}{{\BB^d}}
\newcommand{\BBs}{{\BB^s}}

\newcommand{\NNz}{{\mathbb{N}}}
\newcommand{\NN}{{\NNz_0}}

\newcommand{\RR}{{\mathbb{R}}}
\newcommand{\RRd}{{\RR^d}}
\newcommand{\RRs}{{\RR^s}}
\newcommand{\RRdz}{{\RRd\setminus\{0\}}}

\newcommand{\SSS}{{\mathbb{S}}}
\newcommand{\SSSd}{{\SSS^{d-1}}}

\newcommand{\ap}{{\mathfrak{a}}} 
\newcommand{\wtap}{{\widetilde{\mathfrak{a}}}} 
\newcommand{\wpp}{{\mathfrak{w}_{\varphi}}} 
\newcommand{\wtpp}{{\mathfrak{w}_{\wtp}}} 
\newcommand{\wredp}{{\mathfrak{w}_{\pred}}} 
\newcommand{\bp}{{\mathfrak{b}}} 
\newcommand{\cl}{{\mathrm{cl}}}

\newcommand{\SG}{{\mathrm{SG}}}

\newcommand{\SGcl}{\mathrm{SG}_\mathrm{cl}}

\newcommand{\dSmz}{{\dot{\mathscr{C}}^\infty_0}}
\newcommand{\dSmdz}{({\dot{\mathscr{C}}^\infty_0})^\prime}

\newcommand{\Sm}{{\mathscr{C}^\infty}}

\newcommand{\Smc}{{\mathscr{C}^\infty_{c}}}

\newcommand{\cD}{\mathcal{D}} 
\newcommand{\cC}{\mathcal{C}} 
\newcommand{\wtp}{ {\widetilde{\varphi}} } 
\newcommand{\wtP}{ {\widetilde{\Phi}} } 
\newcommand{\wtl}{ {\widetilde{t}} } 
\newcommand{\wty}{ {\widetilde{y}} } 
\newcommand{\wtY}{ {\widetilde{Y}} } 
\newcommand{\wtU}{ {\widetilde{U}} } 

\newcommand{\supp}{\operatorname{supp}}

\newcommand{\Wt}{\mathcal{W}}

\newcommand{\We}{\Wt^e}
\newcommand{\Wp}{\Wt^\psi}
\newcommand{\Wpe}{\Wt^{\psi e}}

\newcommand{\lp}{{\lambda_\varphi}}
\newcommand{\ellp}{{\ell_\varphi}}

\newcommand{\Lp}{\Lambda_\varphi}

\newcommand{\Lpe}{{\Lambda_\varphi^e}}
\newcommand{\Lpp}{{\Lambda_\varphi^\psi}}
\newcommand{\Lppe}{{\Lambda_\varphi^{\psi e}}}

\newcommand{\Cp}{\mathcal{C}_\varphi}

\newcommand{\Cpe}{\mathcal{C}_\varphi^e}
\newcommand{\Cpp}{\mathcal{C}_\varphi^\psi}
\newcommand{\Cppe}{\mathcal{C}_\varphi^{\psi e}}

\newcommand{\B}{\mathcal{B}}

\newcommand{\Be}{\B^e}
\newcommand{\Bp}{\B^\psi}
\newcommand{\Bpe}{\B^{\psi e}}

\newcommand{\fred}{ {f_{\red}} } 
\newcommand{\pred}{ {\varphi_{\red}} } 
\newcommand{\Pred}{ {\Phi_{\red}} } 

\newcommand{\bV}{ {}^{b}\mathcal{V}}
\newcommand{\scat}{{\mathrm{sc}\hspace{-2pt}}}
\newcommand{\sct}{{\mathrm{sc}}}
\newcommand{\scV}{{}^\scat  \,\mathcal{V}}
\newcommand{\scVX}{{}^\scat \,\mathcal{V}^X}
\newcommand{\scVY}{{}^\scat \, \mathcal{V}^Y}
\newcommand{\scd}{{}^\scat \dd}
\newcommand{\scdY}{{}^\scat \dd_Y}
\newcommand{\scdX}{{}^\scat \dd_X}
\newcommand{\scHess}{{}^\scat H}
\newcommand{\scHX}{{}^\scat H^X}
\newcommand{\scHY}{{}^\scat H^Y}
\newcommand{\scOmega}{{}^\scat\, \Omega}
\newcommand{\scTheta}{{}^\scat\, \varTheta}
\newcommand{\scT}{{}^\scat \,T}
\newcommand{\scS}{{}^\scat \,S}
\newcommand{\scOverT}{{}^\scat \,\overline{T}}

\newcommand{\scOverN}{{}^\scat \,\overline{T}}
\newcommand{\WFsc}{\mathrm{WF}_\sct}
\newcommand{\bdf}[1]{{[#1]}}
\newcommand{\bx}{\mathbf{x}}
\newcommand{\by}{\mathbf{y}}
\newcommand{\wby}{ {\widetilde{\by}} } 
\newcommand{\byp}{ {\by^\prime} } 
\newcommand{\ypp}{ {y^{\prime\prime}} } 
\newcommand{\bzeta}{\boldsymbol{\zeta}}
\newcommand{\bxi}{\boldsymbol{\xi}}

\newcommand{\eps}{\varepsilon}

\author[S. Coriasco]{Sandro Coriasco}
\address{Dipartimento di Matematica ``G. Peano''\newline\indent Universit\`a degli Studi di Torino\newline\indent V. C. Alberto, n. 10, I-10126 Torino, Italy}
\email{sandro.coriasco@unito.it}

\author[M. Doll]{Moritz Doll}
\address{Leibniz Universit\"at Hannover\newline\indent Institut f\"ur Analysis\newline\indent Welfengarten 1, D-30167 Hannover, Germany}
\email{doll[AT]math.uni-hannover.de}

\author[R. Schulz]{Ren\'e Schulz}
\address{Leibniz Universit\"at Hannover\newline\indent Institut f\"ur Analysis\newline\indent Welfengarten 1, D-30167 Hannover, Germany}
\email{rschulz[AT]math.uni-hannover.de}

\title[]{Lagrangian distributions on asymptotically Euclidean manifolds}
\date{\today}

\keywords{Lagrangian distribution, Lagrange manifold, scattering calculus, principal symbol}
\subjclass[2010]{35S30, 46F05, 53D12}

\begin{document}

\begin{abstract}
We develop the notion of Lagrangian distribution on scattering manifolds, meaning on the compactified cotangent bundle, which is a manifold with corners equipped with a scattering symplectic structure. In particular, we study the notion of principal symbol of the arising class of distributions.
\end{abstract}

\maketitle

\setcounter{tocdepth}{1}
\tableofcontents


\section{Introduction}
\label{sec:intro}

In this article we develop a theory of Lagrangian distributions on asymptotically Euclidean manifolds.
Lagrangian distributions were defined by H\"ormander \cite{HormanderFIO} as a tool to obtain a global calculus of Fourier integral operators.
The latter are widely applied, e.g. in the study of partial differential equations \cite{DH}, spectral theory \cite{DG}, index theory \cite{BaeStro} and mathematical physics \cite{GuSt}.
Motivating examples for the necessity of studying Lagrangian distributions on asymptotically Euclidean spaces include fundamental solutions to the Klein-Gordon equation, which exhibit Lagrangian behavior ``at infinity'', see \cite{CoSc2}, as well as simple or multi-layers which arise when solving partial differential equations along infinite boundaries or Cauchy hypersurfaces, see \cite{Cordes}. 

In local coordinates, a classical Lagrangian distribution $u$ on a manifold $X$ is given by an oscillatory integral of the form
\begin{equation}
\label{eq:osciintproto}
I_\varphi(a)=\int_{\RRs} e^{i\varphi}a(x,\theta)\,\dd\theta,
\end{equation}
for some symbol $a\in S^m(\RRd\times\RRs)$ and a phase function $\varphi$ on a subset of $\RR^d\times\RR^s$ bounded in $x$. A class of oscillatory integrals on Euclidean spaces, the local model for our theory, was studied in \cite{CoSc}.

The key feature of the theory of Lagrangian distributions is that each such distribution is globally associated to a Lagrangian submanifold $\Lambda\subset T^*X$ and that its leading order behavior can be invariantly described by its principal symbol which is a section in a line bundle on $\Lambda$.

In this article, we prove that the situation on asymptotically Euclidean manifolds is similar, but with a more delicate structure ``at infinity''. To make this precise, we work within the framework of scattering geometry, developed in \cite{Melrose1,MZ}, see also \cite{HV,WZ}. In the article, we provide an extensive introduction to this theory and add to it a class of naturally arising morphisms, the \emph{scattering maps}. We note that the scattering manifolds may also be seen as Lie manifolds, and in this way our theory complements recent advances in the theory of Lagrangian distributions and Fourier integral operators on such singular spaces (via groupoid techniques), see \cite{Lescure}.

The prototype of a scattering geometry is the Euclidean space $\RRd$, identified with a ball under radial compactification. For this setting, a fitting theory of Lagrangian submanifolds on $\RRd$ was developed in \cite{CoSc2}. As a first step, we adapt this to general scattering manifolds with boundary $X=X^o\cup \partial X$, the boundary being viewed as infinity. On such manifolds, the environment for microlocalization is then the compactified scattering cotangent bundle $\scOverT^*X$, a manifold with corners of codimension $2$ and its boundary $\mathcal{W}=\partial\scOverT^*X$. 
This boundary may be seen as a stratified space, and the two boundary faces of $\scOverT^*X$, which intersect in the corner, inherit a type of contact structure.
The geometric objects of study in our theory are then Legendrian submanifolds of the faces $\mathcal{W}$ which intersect in the corner and are the boundary of some Lagrangian submanifold in the interior and smooth (distribution) densities thereon.

The link with Lagrangian distributions is now as follows. We prove that, despite the singular geometry, any Lagrangian submanifold $\Lambda \subset \mathcal{W}$ locally admits a parametrization through some phase function $\varphi$, via a generalization of the map
\[\lp:\Cp\rightarrow\Lp\quad (x,\theta)\mapsto \big(x,\dd_x\varphi(x,\theta)\big),\]
where $\Cp=(\dd_\theta\varphi)^{-1}\{0\}$.
For each such a phase function, a Lagrangian distribution can be expressed locally as an oscillatory integral as in \eqref{eq:osciintproto}.
Up to Maslov factors and some density identifications, the restriction of $a(x,\theta)$ to $C_\varphi$ yields the (principal) symbol $\sigma(u)$ of $u$ and is interpreted as a (density valued) function on $\Lambda$ by identification via $\lp$.
%
Indeed, the main theorem characterizing the principal symbol will be:
\begin{thm*}
    	Let $\Lambda$ be a $\sct$-Lagrangian on $X$. Then there exists a surjective principal symbol map
	\begin{equation*}
        j^\Lambda_{m_e,m_\psi}\colon I^{m_e,m_\psi}(X,\Lambda) \to \Sm(\Lambda, M_\Lambda\otimes\Omega^{1/2}),
	\end{equation*}
    where $M_\Lambda$ is the Maslov bundle and $\Omega^{1/2}$ denotes the half-density bundle over $\Lambda$.
    Moreover, its null space is $I^{m_e-1,m_\psi-1}(X,\Lambda)$ and
	we have the short exact sequence 
	\[
		0 \longrightarrow  
		I^{m_e-1,m_\psi-1}(X,\Lambda) \longrightarrow
		I^{m_e,m_\psi}(X,\Lambda) \xrightarrow{j^\Lambda_{m_e,m_\psi}}
        \Sm(\Lambda,M_\Lambda\otimes\Omega^{1/2})
		\longrightarrow 0.
	\]
    Equivalently,
    \[I^{m_e,m_\psi}(X,\Lambda) / I^{m_e-1,m_\psi-1}(X,\Lambda) \simeq \Sm(\Lambda, M_\Lambda\otimes\Omega^{1/2}).\]
	%
	%
	%
\end{thm*}
Summarizing, our results show that the theory of Lagrangian distributions, classically studied either locally or on compact manifolds, may be generalized to a theory of Lagrangian distributions on Euclidean spaces or manifolds with boundaries, hence a much wider class of geometries. It is formulated in a way that makes it easily transferable to other singular geometries as well as manifolds with corners, see \cite{Melrosemwc}.

The paper is organized as follows.
In Section \ref{sec:prelim} we give an introduction to scattering geometry.
In particular, we discuss the natural class of maps between scattering manifolds, compactification and scattering amplitudes.
In Section \ref{sec:phaseandlag} we define the Lagrangian submanifolds and phase functions that arise in our theory.
In Section \ref{sec:exchphase} we discuss the techniques of classifying phase functions which parametrize the same Lagrangian submanifold.
In Section \ref{sec:Lagdist} we define the Lagrangian distributions in this setting, starting from oscillatory integrals, and study their transformation properties.
Finally, in Section \ref{sec:symb}, we define the principal symbol of Lagrangian distributions and prove its invariance.

\subsection*{Acknowledgements}
The second author was supported by the DFG GRK-1463.
The third author has been partially supported by the University of Turin ``I@UniTO'' project ``Fourier integral operators, symplectic geometry and analysis on noncompact manifolds'' (Resp. S. Coriasco).
We wish to thank J. Wunsch for useful discussions.

\section{Preliminary definitions}\label{sec:prelim}
In the following, we will recall some elements of the geometric theory known as ``scattering geometry'', cf. \cite{Melrose1,Melrose2,MZ,WZ}. To start with, we need to recall some groundwork on the analysis on manifolds with corners, for which we adopt the definition of \cite{MelroseAPS,Melrosemwc}, cf. also \cite{MO} and \cite{Joyce} for a discussion on the different notions of manifolds with corners in the literature.

\subsection{Manifolds with corners and scattering geometry}
\label{sec:scat}

We recall the following extrinsic definition of a (smooth) manifold with (embedded) corners.

\subsubsection*{Manifolds with corners and $\Sm$-functions}
Let $X$ be a paracompact Hausdorff space. As in the case of manifolds without boundary, a manifold with corners is defined in terms of local charts. A $d$-dimensional chart with corners (of codimension $k$) on $X$ is a pair $(U,\phi)$, where $U$ is an open subset of $[0,\infty)^k\times \RR^{d-k}$ for some $0\leq k \leq d$, and $\phi : U \to \phi(U)\subset X$ is a homeomorphism. If $k = 1$ we call $(U,\phi)$ a chart with boundary.
As usual, we define compatibility between charts and an atlas of charts and therefore obtain a definition of manifolds with boundary and manifolds with corner (abbreviated mwb and mwc, respectively, in the following).
For every manifold with corners $X$ of dimension $d$ there exists a $d$-dimen\-sional $\Sm$-manifold $\wt{X}$ without boundary
with $X\subset\tilde{X}$, and the interior $X^o$ of $X$ is open in $\wt{X}$ and non-empty when $d>0$.
We denote by $\Sm(X)$ the space of the restrictions of the elements of $\Sm(\wt{X})$ to $X$. The tangent space $TX$ and differentials of maps $f:X\rightarrow Y$, $Tf:TX\rightarrow TY$, between manifolds with corners $X, Y$, are obtained as restrictions of the corresponding objects on $\wt{X}$ and $\wt{Y}$.

We always assume $X$ to be compact and assume that there is a finite collection of $\Sm$-functions on $\tilde{X}$, $\{\rho_i\}_{i\in I}$, called boundary defining functions (abbreviated bdf), such that $X=\bigcap_{i\in I}\{p\in \tilde{X}, \rho_i(p)\geq 0\}$, and at every point where $\rho_j=0$ for every $j\in J\subset I$, the differentials of these $\rho_j$ are supposed to be linearly independent. In particular, $\dd\rho_j\neq 0$ when $\rho_j=0$. We also always assume to be working in local coordinates of the form 
$\bx\colon p\mapsto (\rho_1,\dots,\rho_k,x_1,\dots,x_{d-k})(p)$, where $k$ is the number of boundary defining functions\footnote{Note that the $\rho_j$ cannot always be chosen as coordinates at interior points, since their differential may vanish in the interior. As it is customary, we disregard this minor technical inconvenience in order to allow for an easier consistent notation and think of the $\rho$ to be replaced by any other admissible coordinate function there.}. %
\begin{rem}\label{rem:joycedef}
    Joyce calls this notion a (compact) \emph{manifold with embedded corners} (cf. Remark 2.11 in \cite{Joyce}).
    By Proposition 2.15 in \cite{Joyce}, we see that, locally, a boundary defining function always exists, and the property that all corners are embedded ensures that a global boundary defining function exists.
    Most of the times the actual choice of boundary defining function is not relevant (cf. Proposition 2.15).
\end{rem}
Let $p\in X$. Then the depth of $p$, $\mathrm{depth}(p)$, is the number of independent boundary defining functions vanishing at $p$, which coincides with the co-dimension of the boundary stratum in which $p$ is contained. We recall that for $j\in\{0,\dots,d\}$ one sets
$\partial_jX=\{p\in X\,|\,\mathrm{depth}(p)= j\}.$
%
%
In particular, $X^o=\partial_0 X$ and $\partial X=\bigcup_{j>0} \partial_j X$. We note that as such, the boundary of a mwc is not a mwc itself, but rather a topological manifold. Nevertheless, it is possible to define smooth functions on $\partial X$ as the set of restrictions smooth functions on $X$ to $\partial X$.

Given a relatively open subset $U$ of a manifold with corner $X$, we say that $U$ is \emph{interior} if $\overline{U}\cap\partial X=\emptyset$. Otherwise, we always assume that $U$ contains all interior points of the boundary $\overline{U}\cap \partial X$ and call $U$ a \emph{boundary neighbourhood}.

We will write $f\in\Sm(U)$ if and only if there is an extension $\wt{f}\in\Sm(X)$ that coincides with $f$ on $U$. The space $\rho_1^{-m_1}\cdots\rho_k^{-m_k}\Sm(U)$ is the space of functions $h\in\Sm(U^o)$ such that $\rho_1^{m_1}\cdots\rho_k^{m_k}h$ extends to an element of $\Sm(U)$.

The class of mwc that interest us is that of (products of) fiber bundles where both the base as well as the fiber are allowed to be a compact manifold with boundary (abbreviated ``mwb''). The archetype of such a mwc is the product of two mwbs. Indeed, if $X$ and $Y$ are mwbs, $B=X\times Y$ is a mwc. We write $\B=\partial B$ and we have (adopting the notation of \cite{CoSc2,ES}) 
\begin{align*}
\B&=\underbrace{(\partial X\times Y^o) \cup (X^o\times \partial Y)}_{=\partial_1 B}\cup \underbrace{(\partial X\times \partial Y)}_{=\partial_2 B}=: \Be \cup \Bp\cup \Bpe.
\end{align*}
We now describe the basics of scattering geometry, cf. \cite{Melrose1,Melrose2,MZ,WZ}. We first recall the guiding example.
\begin{defn}[Radial compactification of $\RRd$]
Pick any diffeomorphism $\iota:\RRd\rightarrow (\BBd)^o$ that, for $|x|>3$, is given by
\[\iota:x\mapsto \frac{x}{|x|}\left(1-\frac{1}{|x|}\right).\]
Then its inverse is given, for $|y|\geq \frac{2}{3}$, by
\[\iota^{-1}:y\mapsto \frac{y}{|y|}(1-|y|)^{-1}.\]
The map $\iota$ is called the \emph{radial compactification map}. We may hence view $\RRd$ as the interior of the mwb $\BBd$ and call $\partial \BBd$ ``infinity''.

Denote by $\bdf{x}$ a smooth function $\RRd\rightarrow (0,\infty)$ that, for $|x|>3$, is given by $x\mapsto |x|$. Then $(\iota^{-1})^*\bdf{x}^{-1}$ is a boundary defining function on $\BBd$ (and we view $\bdf{x}^{-1}$ as a boundary defining function on $\RRd$). Indeed, for $|y|>2/3$ it is given by 
$y\mapsto 1-|y|=\rho_Y$.
\end{defn}
\begin{rem}
\label{rem:compequiv}
In scattering geometry, the explicit choice of compactification of $\RRd$ often differs from ours, see \cite{MZ}. Write $\jap{x}=\sqrt{1+|x|^2}$ for $x\in\RRd$ and define
\[x\mapsto \left(\frac{1}{\jap{x}},\frac{x}{\jap{x}}\right)=:\left(\wt{\rho_Y},\wt{y}\right).\]
This maps $\RRd$ into the interior of the half-sphere with positive first component, and $\wt{\rho_Y}$ and $d-1$ of the $\wt{y} = \wt{\rho_Y}\cdot x$ functions may be chosen as local coordinates. Because of the following computation, both compactifications are equivalent, meaning they yield diffeomorphic manifolds. In fact, for $|x|>3$, we may write
\[\jap{x}^{-1}=\bdf{x}^{-1}\frac{1}{1+\bdf{x}^{-2}}, \qquad \bdf{x}^{-1}=\jap{x}^{-1}\frac{1}{\sqrt{1-\jap{x}^{-2}}}.\]
Hence, $\jap{x}^{-1}$ and $\bdf{x}^{-1}$ yield equivalent boundary defining functions on $\RRd$.
\end{rem}
\begin{defn}[Scattering vector fields on mwbs]
Let $X$ be a mwb with bdf $\rho$. Consider the space $\bV(X)$ of vector fields tangential to $\partial X$.
Then $\scV(X)$ is the space $\rho\, \bV(X)$. 
Near any point with $\rho=0$, the vector fields
$\{\rho^2\partial_\rho,\ \rho\partial_{x_j}\}$
generate $\scV(X)$. In particular, $\scV(X)$ contains vector fields supported in $X^o$.

By the Serre-Swan theorem, there exists a $\Sm$-vector bundle $\scT X$ such that $\scV(X)$ are its $\Sm$-sections.
We have a natural inclusion map $\scT X\hookrightarrow TX$. Note that $\{\rho^2\partial_\rho,\ \rho\partial_{x_j}\}$ are, as elements of $\scT_pX$, non-vanishing at boundary points $p\in \partial X$ despite $\rho=0$.\\
The inclusion reverses for the dual bundles $ T^*X\hookrightarrow \scT^*X$. In coordinates, we denote the dual elements to $\{\rho^2\partial_\rho,\ \rho\partial_{x_j}\}$ by $\left\{\frac{\dd\rho}{\rho^2},\frac{\dd x_j}{\rho}\right\}$, and these span the sections of $\scT^*X$ near the boundary.

We now consider the the \textit{compactified scattering cotangent bundle} $\scOverT^*X$, which is the fiber-wise radial compactification of $\scT^*X$, a compact manifold with corners.
The new-formed fiber boundary may be identified with a rescaling of the cosphere bundle, called $\scS^*X$.
The boundary of the new-formed mwc $W=\scOverT^*X$, which we denote\footnote{This is a slight change of notation compared to \cite{Melrose1} where it is denoted $C_\sct$.} by $\Wt$, splits into three components: the boundary faces
$$\We:=\scT^*_{\partial X}X,\qquad \Wp:=\scS^*_{X^o}X,\qquad \Wpe:=\scS^*_{\partial X}X.$$
This geometric situation (with $X$ identified as the zero section) near the boundary is summarised in Figure \ref{fig:Wt} (cf. \cite{CoSc2,MZ}).
\begin{figure}[htb!]
\begin{center}
\begin{tikzpicture}[scale=1.2]
  \node (A) at (1.5,3.3) {$\Wp$};
  \node (B) at (3.4,1.5) {$\We$};
  \node (C) at (3.3,3.2) {$\Wpe$}; 
  \node (D) at (1.2,1.2) {$X^o$};
  \node (E) at (0.2,2) {$\scOverT^*X$};
  \draw (1,3) -- (3,3);
  \draw[dotted] (0,3) -- (1,3);
  \draw (3,0) -- (3,3);
  \draw[opacity=0.75] (1,1) -- (3,1);
  \draw[opacity=0.75,dotted] (3,0) -- (3,-1);
  \draw[opacity=0.75,dotted] (0,1) -- (1,1);
  \draw[->,thick] (3.5,0.6) -- (3.05,0.9);
  \node at (3.6,0.4) {$\partial X$};
\end{tikzpicture}
\caption{The boundary faces and corner of $\scOverT^*X$}
\label{fig:Wt}
\end{center}
\end{figure}

The exterior derivative $\dd$ lifts to a well-defined scattering differential $\scd$ on the scattering geometric structure.
In coordinates, with $\rho$ a local boundary defining function, we write
\begin{align}
\label{def:scddef}
\scd f=\rho^2\partial_\rho f\,\frac{\dd\rho}{\rho^2}+\sum_{j=1}^{d-1}\rho\partial_{x_j} f\,\frac{\dd x_j}{\rho}.
\end{align}
Note that for $f\in \Sm(X)$, this means that as a section of $\scT^*X$, $\scd f$ necessarily vanishes on the boundary. In fact, we may extend $\scd$ to the space $\rho^{-1}\Sm(X)$ and obtain a map
$$\scd:\rho^{-1}\Sm(X)\longrightarrow \scTheta(X)=\Gamma(\scT^*X).$$
That is, in local coordinates near the boundary,
$$\scd(\rho^{-1}f) = \rho^{-1}\, \scd f - f \frac{\dd\rho}{\rho^2}=(-f+\rho\partial_\rho f)\,\frac{\dd\rho}{\rho^2}+\sum_{j=1}^{d-1}\partial_{x_j} f\,\frac{\dd x_j}{\rho}.$$
%
\end{defn}
\begin{rem}
    We note that $\rho^{-1}\Sm(X)$ and similarly defined spaces are independent of the actual choice of boundary defining function $\rho$ (cf. Remark \ref{rem:joycedef}).
\end{rem}

\begin{ex}
Outside a compact neighbourhood of the origin, polar coordinates provide an isomorphism $\RRd\cong\RR_+\times \SSSd$. The vector fields $\partial_r$ and $\frac{1}{r}\partial_{x_j}$, $x_j$ being coordinates on $\SSSd$, correspond (up to a sign) under radial inversion $\rho=\frac{1}{r}$ to $\rho^2\partial_\rho$ and $\rho\partial_{x_j}$. Hence, scattering vector fields on $\BBd$ arise as the image of the vector fields of bounded length on $\RRd$ under radial compactification.
\end{ex}

\begin{defn}
A \emph{scattering manifold} (also called asymptotically Euclidean manifold) is a compact manifold with boundary $(X, \rho)$ whose interior is equipped with a Riemannian metric $g$ that is supposed to take the form, in a tubular neighbourhood of the boundary,
$$g=\frac{{(\dd \rho)}^{\otimes 2}}{\rho^4}+\frac{g_\partial}{\rho^2}$$
where $g_\partial\in\Sm(X,\mathrm{Sym}^2 T^*X)$ restricts to a metric on $\partial X$.
\end{defn}
Any mwb may be equipped with a scattering metric.
\begin{ex}
In polar coordinates, the metric on $\RRd\setminus\{0\}$ can be written as 
$$g=(\dd r)^{\otimes 2}+r^2 g_{\SSSd}.$$
Pulled back to $\BBd$ using $\iota$, that is $r=(1-|y|)^{-1}=\rho^{-1}$ near the boundary, this becomes
$$g_{\BBd}=\frac{(\dd \rho)^{\otimes 2}}{\rho^4}+\frac{g_{\SSSd}}{\rho^2}.$$
%
\end{ex}

\begin{defn}[Scattering vector fields on product type manifolds]
For a product $B=X\times Y$, with $(X,\rho_X)$ and $(Y,\rho_Y)$ mwbs, we may introduce $\scV(B)$ as $\rho_X\rho_Y  (\bV(B))$.
Near a corner point the resulting bundle $\scT^*B$ is hence generated, if 
$\bx=(\rho_X,x)$ and $\by=(\rho_Y,y)$ are local coordinates on $X$ and $Y$ respectively, by
$$\rho_X^2\rho_Y\partial_{\rho_X},\ \rho_X\rho_Y\partial_{x_j},\ \rho_X\rho_Y^2\partial_{\rho_Y},\ \rho_X\rho_Y\partial_{y_k}.$$
The space $\scV(B)$ splits into horizontal and vertical vector fields\footnote{Consider the projection $\pr_X:B\rightarrow X$. Then $v\in\scV(B)$ satisfies $v\in\scVX(B)$ if $v(\pr_X^*f)=0$ for all $f\in\Sm(X)$. The set $\scVY(B)$ is defined in analogy.},
$\scVX(B)$ and $\scVY(B)$, respectively, and we define
$\scTheta^X(B)$ as the set of (scattering) 1-forms $w \in \scTheta^1(B)$ such that
$w(v) = 0$ for all $v \in \scVY(B)$.

Given complete set of coordinates $\bx=(\rho_X, x)$, 
$\by=(\rho_Y, y)$ on $X$ and $Y$, respectively,
we see that $\scTheta^X(B)$ is the set of sections generated by
\[\frac{\dd\rho_X}{\rho_X^2\rho_Y}, \frac{\dd x_j}{\rho_X\rho_Y}.\]
The underlying vector bundle will be denoted by $\scHX B$.
Similarly, we define $\scTheta^Y(B)$ and $\scHY B$.
It is important to note that we have the following ``rescaling identifications'':
\begin{equation}
\label{eq:rescal}
\begin{aligned}
\scTheta^X(B)\ni \frac{\dd \rho_X}{\rho_X^2 \rho_Y}&\,\longleftrightarrow\, \rho_Y^{-1}\frac{\dd\rho_X}{\rho_X^2}\in \rho_Y^{-1}\Sm(Y,\scTheta(X)),
\\
\scTheta^X(B)\ni \frac{\dd x_j}{\rho_X\rho_Y}&\,\longleftrightarrow\, \rho_Y^{-1}\frac{\dd x_j}{\rho_X}\in \rho_Y^{-1}\Sm(Y,\scTheta(X)).
\end{aligned}
\end{equation}

Again, we may define the scattering exterior differential
$\scd$, induced by the usual exterior differential $\dd$, and extend it to a map 
\[\scd : \rho_X^{-1}\rho_Y^{-1}\Sm(B)\longrightarrow \scTheta(B).\]
In terms of the scattering differentials on $X$ and $Y$ we may decompose $\scd$ as $\scd=\scdX+\scdY$, where
\begin{align*}
    \scdX : \rho_X^{-1}\rho_Y^{-1} \Sm(B) \to \scTheta^X(B),\\
    \scdY : \rho_X^{-1}\rho_Y^{-1} \Sm(B) \to \scTheta^Y(B).
\end{align*}
\end{defn}

\subsection{Amplitudes}

\begin{defn}[Amplitudes of product-type]
Let $B$ be a mwc, $\{\rho_j\}_{j=1\dots k}$ a complete set of bdfs. Then $a$ is called an amplitude of order $m\in\RR^{k}$ if 
$$a\in \rho_1^{-m_1} \cdots \rho_k^{-m_k}\Sm(B).$$ 
For an open subset $U$ of $X$, a \emph{locally defined} amplitude of product type is an element of $\rho_1^{-m_1}\cdots\rho_k^{-m_k}\Sm(\overline{U})$. 
For $p\in\partial X$ we call $a$ \emph{elliptic at $p$} if $\rho^{m_1}_1\cdots\rho_k^{m_k}a(p)\neq 0$.
We write 
$$\dSmz(X):=\bigcap_{m\in\RR^k}\rho_1^{-m_1}\cdots\rho_k^{-m_k}\Sm(B)$$
for the smooth functions vanishing at the boundary of infinite order.

For $p\in\partial B$ we call $a$ \emph{rapidly decaying at $p$} if there exists a neighbourhood $U$ of $p$ such that $a$ vanishes of infinite order on $U\cap \partial B$, that is $a \in \dSmz(\overline U)$.
\end{defn}

We now study the leading boundary behavior of these amplitudes. For simplicity, we only consider $B=X\times Y$ for mwbs $X$ and $Y$.
\begin{defn}
\label{def:princsymbol}
	Let $a \in \rho_X^{-m_e}\rho_Y^{-m_\psi} \Sm(B)$ and write $a = \rho_X^{-m_e}\rho_Y^{-m_\psi} f$ for some $f \in \Sm(B)$.
	Given a coordinate neighbourhood $U$ of a point $p\in\B^\bullet$, we define symbols $\sigma^\bullet(a)$ of $a$ on $U$ by
	\begin{align*}
	\begin{cases}
    	\sigma^e(a)(\bx,\by)=\rho_X^{-m_e}\rho_Y^{-m_\psi}f(0,x,\by), &p\in\B^e\cup\B^{\psi e} \\
    	\sigma^\psi(a)(\bx,\by)=\rho_X^{-m_e}\rho_Y^{-m_\psi}f(\bx,0,y), &p\in\B^\psi\cup\B^{\psi e}\\
    	\sigma^{\psi e}(a)(\bx,\by)=\rho_X^{-m_e}\rho_Y^{-m_\psi}f(0,x,0,y) &p\in\B^{\psi e}.
	\end{cases}
	\end{align*}
	The tuple $(\sigma^\psi(a),\sigma^e(a), \sigma^{\psi e}(a))$ is denoted by $\sigma(a)$ and called the \emph{principal symbol}.
\end{defn}
Fix $\epsilon>0$ so small that $\rho_X$ and $\rho_Y$ can be chosen as coordinates on $B$ respectively whenever $\rho_X<\epsilon$ and $\rho_Y<\epsilon$.
We choose a cut-off function $\chi \in \Sm(\RR)$ such that $\chi(t) = 0$ for $t > \epsilon/2$ and $\chi(t) = 1$ for $t < \epsilon/4$.

\begin{defn}\label{def:princpart}
    For any $a \in \rho_X^{-m_e}\rho_Y^{-m_\psi} \Sm(B)$ the amplitude
    \begin{align*}
        a_p(\bx,\by) = \chi(\rho_X) \sigma^e(a)(\bx,\by) + \chi(\rho_Y) \sigma^\psi(a)(\bx,\by) - \chi(\rho_X)\chi(\rho_Y) \sigma^{\psi e}(a)(\bx,\by)
    \end{align*}
    is called the \emph{principal part} of $a$.
\end{defn}
While $a_p$ does depend on the choice of $\chi$, its leading boundary asymptotic do not. By Taylor expansion of $f$, we obtain:
\begin{lem}
	\label{lem:princpart}
    The principal part $a_p$ of $a$ satisfies $ a - a_p \in \rho_X^{-m_e+1}\rho_Y^{-m_\psi+1} \Sm(B).$
    
\end{lem}
%

\begin{ex}[Classical $\SG$-symbols]
Let $B=\BBd\times\BBs$, where $\BBd$ and $\BBs$ are the radial compactifications of $\RRd$ and $\RRs$.
The space of so-called classical $\SG$-symbols, $\SGcl^{m_e,m_\psi}(\RRd\times\RRs)$, is that of $a\in\Sm(\RRd\times\RRs)$ such that
$(\iota^{-1}\times\iota^{-1})^*a\in \rho_X^{-m_e} \rho_Y^{-m_\psi}\Sm(B)$. These symbols are then precisely those that satisfy the estimates
\begin{equation}
\label{eq:SGest}
\left|\partial_x^{\alpha}\partial_\theta^{\beta} a(x,\theta)\right|\lesssim \jap{x}^{m_e-|\alpha|}\jap{\theta}^{m_\psi-|\beta|}
\end{equation}
and admit a polyhomogeneous expansion, see \cite{ES,Melrose1,WZ} and the principal symbol of $a$ corresponds to its homogeneous coefficients, see \cite[Chap. 8.2]{ES}.
\end{ex}
%
We will need to consider density-valued amplitudes and integrate amplitudes on mwbs. For this, we introduce the space of scattering $\sigma$-density bundles, cf. \cite{Melrose1}, where $\scOmega^{\sigma}(X)=\rho^{-\sigma(d+1)}\Omega^\sigma(X)$ in terms of the usual $\sigma$-density bundle. Note that $\scOmega^{\sigma}$ does not depend on the choice of boundary defining function.

\begin{ex}\label{ex:mugdensity}
    Under the radial compactification, the canonical Lebesgue integration density on $\RR^d$, $\dd x \in \Omega^1(\RR^d)$, is mapped to $\iota_*\dd x \in \scOmega^1(\BB^d)$.
    In particular, we obtain $\iota_*\dd x = \rho^{-(d+1)}\dd\rho\,\dd\SSSd$. More generally, if $(X,g)$ is a scattering manifold, then the metric induces a canonical volume scattering 1-density $\mu_g$.
\end{ex}

Since the density bundle is a line bundle, any choice of scattering density provides a section of it and allows for an identification of scattering densities on $X$ and $\Sm$-functions.

We denote the set of all smooth sections of the bundle $\scOmega^\sigma(X)$ by $\Sm(X,\scOmega^\sigma(X))$, and
the tempered distribution densities $(\dSmz)'(X, \scOmega^\sigma(X))$ are the continuous linear functionals on $\dSmz(X, \scOmega^{1-\sigma}(X))$.
\begin{lem}\label{lem:intdensity}
    Let $X$ be a mwb and $Y$ a manifold without boundary. Then, integration over $Y$ induces a map
    \begin{align*}
        \int_Y : \Smc(X\times Y, \scOmega^1(X\times Y)) \longrightarrow \rho_X^{-\dim Y} \Smc(X,\scOmega^1(X)).
    \end{align*}
\end{lem}
\begin{rem}\label{rem:pushforward}
    More generally, let $X,Y$ be mwbs and $Z$ a manifold without boundary.
    Consider a differentiable fibration $f : X \to Y$ with typical fiber $Z$.
    For every scattering density $\mu \in \Sm(X,\scOmega^1(X))$ the pushforward 
    \[f_* \mu \in \rho_Y^{-\dim Z} \Smc(Y, \scOmega^1(Y))\]
    is defined locally by integration along the fiber.

    Let $(U, \psi)$ be a trivializing neighborhood of the fiber bundle,
    that is $U \subset Y$ open, $\psi : X \to U \times Z$ smooth and $f|_{f^{-1}(U)} = \pr_M \circ \psi$.
    Assume without loss of generality that $\mu$ is supported on $f^{-1}(U)$.
    Then set
    \[f_* \mu = \int_Z \mu\circ\psi_j.\]
\end{rem}
\subsection{Scattering maps}
We now introduce and characterize the class of maps whose pull-backs preserve amplitudes of product type. They are a special case of interior $b$-maps in the sense of \cite{MelroseAPS}, and humbly mimicking Melrose's naming conventions we call them $\sct$-maps. We first introduce them on manifolds with boundary and then generalize to manifolds with higher corner degeneracy, such as products of mwcs.

\begin{defn}[$\sct$-maps on mwb]
Let $Y$ and $Z$ be mwbs. Suppose $\Psi:Y\rightarrow Z$. Then $\Psi$ is called an $\sct$-map if for any $m\in\RR$ and $a\in \rho_Z^{-m}\Sm(Z)$ it holds that:
\begin{enumerate}
\item $\Psi^*a\in \rho_Y^{-m}\Sm(Y)$;
\item if $p\in \Psi(Y)$ with $p=\Psi(q)$ and $(\rho_Z^{m} a)(p)> 0$, then $(\rho_Y^{m} \Psi^*a)(q)> 0$. 
\end{enumerate}
\end{defn}

\begin{rem}\label{rem:inward}
In particular, $\Psi$ maps the boundary of $Y$ into that of $Z$.
It also follows that $T\Psi$ maps inward pointing vectors at the boundary (meaning vectors with strictly positive $\partial_\rho$-component) to inward pointing vectors at the corresponding points. Indeed, we see that, at the boundary, $\Psi_*\partial_{\rho_Z}=h^{-1}\partial_{\rho_Y}$.
\end{rem}

\begin{rem}\label{rem:SGmapcomp}
It is obvious that the composition of two $\sct$-maps is again a $\sct$-map.
\end{rem}

It is straightforward to adapt this definition to that of a local $\sct$-map by replacing $Y$ and $Z$ with open subsets.

\begin{lem}[$\sct$-maps in coordinates]
\label{lem:SGmapcoord}
Let $Y$ and $Z$ be mwbs, $U\subset Y$ and $V\subset Z$ open subsets. 
A smooth map $\Psi:U\rightarrow V$ is a local $\sct$-map if and only if for the boundary defining functions on $Y$ and $Z$, $\rho_Y$ and $\rho_Z$, respectively, we have 
\begin{equation}
\label{eq:locscmap}
\Psi^*\rho_Z=\rho_Yh\text{ for some }h\in\Sm(Y)\text{ with }h> 0.
\end{equation}
\end{lem}

Hence, any local diffeomorphism of mwbs is a local scattering map. Moreover:

\begin{lem}\label{lem:SGmapproj}
    Let $X, Z$ be mwbs.
    Given any open, bounded set $U \subset \RRd$, define the projection $\pr_Z : Z \times U \to Z, (z,y) \mapsto z$.
    Then $\id_X \times \pr_Z$ is a $\sct$-map.
\end{lem}

We now investigate the action of pull-backs by $\sct$-maps on the 
objects introduced above. The following assertions can be verified in local coordinates.

\begin{lem}\label{lem:scdpullback}
	Let $Y$ and $Z$ be mwbs, $U\subset Y$ and $V\subset Z$ open subsets. 
	Let $\Psi:U\rightarrow V$ be a local $\sct$-map. Then, 
	the following properties hold true.
	
	\begin{itemize}
		\item $\Psi^*$ yields a map $\rho_Z^m\,\scTheta^k(V)\rightarrow \rho_Y^m\,\scTheta^k(U)$ for any $m\in \RR$ and $k\in \NNz$. Moreover, for $\theta\in\rho_{Z}^{m} \,\scTheta^k(V)$, we have 
				$\scd (\Psi^*\theta) = \Psi^*(\scd \theta)$.
		\item $\Psi^*$ yields a map $\scOmega^\sigma(V)\rightarrow \scOmega^\sigma(U)$ for any $\sigma\in[0,1]$.
		\item The map $T^*\Psi:T^*V\rightarrow T^*U$ lifts to a map $\scOverT^*\Psi:\scOverT^*V\rightarrow \scOverT^*U$. In local coordinates, away from fiber-infinity, $\scOverT^*\Psi$ is given by
		$$(\Psi(\by),\bzeta)\mapsto \big(\by,\iota(^t(J \Psi)(\iota^{-1}\bzeta))\big),$$ 
		wherein $J\Psi$ is the Jacobian of $\Psi$ at $\by.$
		The extension to fiber-infinity is obtained by taking interior limits $|\zeta|^{-1}\rightarrow 0$.
	\end{itemize}

\end{lem}

We observe that $\sct$-maps provide a natural class of maps between scattering manifolds.

\begin{cor}
Suppose $Y$ is a mwb, $(Z,\rho_Z,g)$ a scattering manifold, $\Psi$ 
a $\sct$-map $Y\rightarrow Z$ which is an immersion. Then $(Y,\Psi^*\rho_Z,\Psi^*g)$ is a scattering manifold.
\end{cor}

\begin{proof}
We first observe that $\Psi^*\rho_Z$ is a boundary defining function on $Y$. Indeed,
\begin{equation}
\label{eq:scTident}
\dd\Psi^*\rho_Z=h\,\dd\rho_Y+\rho_Y \dd h.
\end{equation}
This implies, at the boundary, $h\,\dd\rho_Y\neq 0$. The scattering metric on $Z$ pulls back to   
$$\Psi^*g=\Psi^*\frac{(\dd \rho_Z)^{\otimes 2}}{\rho_Z^4}+\Psi^*\frac{g_\partial}{\rho_Z^2}=\frac{(\dd \Psi^*\rho_Z)^{\otimes 2}}{(\Psi^*\rho_Z)^4}+\frac{\Psi^*g_\partial}{(\Psi^*\rho_Z)^2},$$
which is again a scattering metric.
\end{proof}

%
\begin{cor}
\label{cor:realizeonball}
Any scattering manifold $Y$ of dimension $s$ is locally diffeomorphic to $\BBs$. Moreover, any scattering density on $Y$ can locally be written as the pull-back by one on $\BBs$.
\end{cor}

We now extend the notion of $\sct$-map to manifolds with corners.

\begin{defn}[$\sct$-maps on mwc]
Let $Y$ and $Z$ be mwcs. Then, a smooth map $\Psi:Y\rightarrow Z$ is a local $\sct$-map for some complete sets of local bdfs $\{\rho_{Y_i}\}_{i\in I}$ and $\{\rho_{Z_i}\}_{i\in I}$ if:
$$\text{For all }i\in I\text{ we have }\Psi^*\rho_{Z_i}=\rho_{Y_i} h_i\text{ for some }h_{i}\in\Sm(Y)\text{ with }h_{i}> 0.$$
\end{defn}

\begin{rem}
In particular, $\Psi$ maps the boundary of $Y$ into that of $Z$.

As mentioned before, $\sct$-maps are special cases of $b$-maps. In fact, they are those \emph{interior} $b$-maps that are smooth maps in the sense of \cite{Joyce}. The only difference with the smooth maps in \cite{Joyce} is that, therein, $\Psi^*\rho_{Z_i}\equiv 0$ is allowed.
\end{rem}

\begin{ex}
In particular, if $\Psi_1:Y_1\rightarrow Z_1$ and $\Psi_2:Y_2\rightarrow Z_2$ are $\sct$-maps on mwb, then $\Psi_1\times \Psi_2:Y_1\times Y_2\rightarrow Z_1\times Z_2$ is a $\sct$-map on the resulting product mwc.
\end{ex}

\begin{rem}
Note that we fix the ordering of the boundary defining functions. This is important, in particular, when considering $\sct$-maps between products $X\times Y\rightarrow X\times Z$ or of the form $X\times Y\rightarrow\scOverT^*X$. Most of the times, the choice of bdfs will be clear from the context.
\end{rem}

Note that, on a mwb, it is possible to extend any map $\partial X\mapsto \partial X$ with $x\mapsto x'$ to a scattering map, by setting $(\rho_X,x)\mapsto (\rho_X,x')$ in a collar neighbourhood of $\partial X$ given by $X\cong [0,\epsilon)\times \partial X$. The following proposition grants us the ability to continue scattering maps from a corner into the interior.
\begin{prop}
\label{prop:cornerdiffeo}
	Let $B_1=X_1\times Y_1$ and $B_2=X_2\times Y_2$ be products of mwbs. Let $\Psi^e$, $\Psi^\psi$ be two (local) scattering maps near a point $p\in\Bpe_1$,
	\begin{align*}
		\Psi^e:\Be_1\longrightarrow \Be_2 \quad\text{ and }\quad
		\Psi^\psi: \Bp_1\longrightarrow \Bp_2
	\end{align*}
such that $\Psi^e=\Psi^\psi$ when restricted to $\Bpe_1$. Then there exists a (local) scattering map $\Psi$ on a neighbourhood $U\subset B_1$ of $p$ with $\Psi^\bullet=\Psi|_{\B^\bullet}$ such that
\begin{equation}
	\label{eq:strictscproperty}
	\partial_{\rho_{X_1}}\Psi^*\rho_{Y_2}=\partial_{\rho_{Y_1}}\Psi^*\rho_{X_2}=0\quad \text{on }\B_1.
\end{equation}
If $\Psi^e$ and $\Psi^\psi$ are local diffeomorphisms near $p$ (in their respective boundary faces), then $\Psi$ is a local diffeomorphism near $p$.
\end{prop}
\begin{proof}
This is Whitney's extension theorem for smooth functions, applied to the system of functions (and their derivatives)
\begin{align*}
(\Psi^e)^*x,(\Psi^e)^*y,(\Psi^e)^*\rho_Y &\qquad\textrm{on }\Be_1,\\
(\Psi^\psi)^*\rho_X,(\Psi^\psi)^*x,(\Psi^\psi)^*y &\qquad\textrm{on }\Bp_1,
\end{align*}
together with the conditions \eqref{eq:strictscproperty} and
\begin{align*}
	\label{eq:strictscproperty}
	D_{x,y}\Psi^*\rho_{Y_2}=0\quad \text{on }\B_1^\psi,\\
	D_{x,y}\Psi^*\rho_{X_2}=0\quad \text{on }\B_1^e.
\end{align*}
Note that, if $\Psi^e$ and $\Psi^\psi$ are local diffeomorphisms at $p$, the differential of $\Psi$ is an invertible block matrix, and hence $\Psi$ is a local diffeomorphism.
\end{proof}

\begin{lem}\label{lem:trafoprinc}
	Consider a $\sct$-map $\Psi : X \times Y \to X \times Y$ of product form $\Psi = \Psi_X \times \Psi_Y$, with $\sct$-maps on $X,Y$, $\Psi_X$ and $\Psi_Y$, respectively.
	Assume $a \in \rho_Y^{-m_\psi} \rho_X^{-m_e} \Sm(X\times Y)$. With the notation of Definition \ref{def:princsymbol} and \ref{def:princpart}, we have:
    \begin{align*}
        \sigma^\psi(\Psi^*a) - \Psi^*(\sigma^\psi a) &\in \rho_Y^{-m_\psi + 1} \rho_X^{-m_e}\Sm,\\
        \sigma^e(\Psi^*a) - \Psi^*(\sigma^e a) &\in \rho_Y^{-m_\psi}\rho_X^{-m_e + 1} \Sm,\\
        (\Psi^*a)_p - \Psi^*(a_p) &\in \rho_Y^{-m_\psi + 1} \rho_X^{-m_e + 1}\Sm.
    \end{align*}
\end{lem}
\begin{proof}
    We will only prove the first identity, the others follows by similar arguments. 
    Write $(\Psi^*\rho_X)(\bx)=\rho_X h_X(\bx)$ and $(\Psi^*\rho_Y)(\by)=\rho_Y h_Y(\by)$. If $a = \rho_X^{-m_e} \rho_Y^{-m_\psi} f$ then
    \begin{align*}
        (\Psi^*a)(\bx,\by) = \rho_X^{-m_e} \rho_Y^{-m_\psi} h_X^{-m_e}(\bx) h_Y^{-m_\psi}(\by) (\Psi^*f)(\bx, \by).
    \end{align*}
    This implies
    \begin{align*}
        \sigma^{\psi}(\Psi^*a)(\bx,\by) &= \rho_X^{-m_e} \rho_Y^{-m_\psi} h_X^{-m_e}(\bx) h_Y^{-m_\psi}(0,y) (\Psi^*f)(\bx, 0, y),\\
        \Psi^*(\sigma^{\psi}a)(\bx,\by) &= \rho_X^{-m_e} \rho_Y^{-m_\psi} h_X^{-m_e}(\bx) h_Y^{-m_\psi}(\by) (\Psi^*f)(\bx, 0, y).
    \end{align*}
    Using Taylor's theorem, we obtain that $h_Y^{-m_\psi}(\by) - h_Y^{-m_\psi}(0,y) \in \rho_Y \Sm(X\times Y)$, and therefore
    $\sigma^\psi(\Psi^*a) - \Psi^*(\sigma^\psi a) \in \rho_Y^{-m_\psi + 1} \rho_X^{-m_e}\Sm(X\times Y)$, as claimed.
\end{proof}
\begin{cor}
The principal part of $a\in \rho_Y^{-m_\psi} \rho_X^{-m_e} \Sm(X\times Y)$ is well-defined as an element of
\[
	\rho_X^{-m_e} \rho_Y^{-m_\psi} \Sm(X\times Y) / \rho_X^{-m_e+1} \rho_Y^{-m_\psi+1} \Sm(X\times Y),
\]
and does not depend on the choice of boundary-defining functions $\rho_X,\rho_Y$ on $X,Y$.
\end{cor}
\begin{rem}
Note that the space 
\[\rho_X^{-m_e} \rho_Y^{-m_\psi} \Sm(X\times Y) / \rho_X^{-m_e+1} \rho_Y^{-m_\psi+1} \Sm(X\times Y)\]
can be identified with $\Sm(\partial(X\times Y))$, which identifies our notion of principal symbol with that of \cite[Section 6.4]{Melrose2}.
\end{rem}
%

The following lemma is one of the main technical tools in this article.
We have already observed that the local model of a scattering manifold near the boundary is the radial compactification of $\RRd$. We now show that scattering maps arise naturally as the composition of vector-valued amplitudes and radial compactification. Furthermore, we clarify the relation between total derivative and the scattering differential under compactification.
\begin{lem}
\label{lem:horror}
Let $Y$ be a mwb. 
Let $f\in \rho_Y^{-1}\Sm(Y,\RRd)$ with $\rho_Y|f|\neq 0$ on $\partial Y$.\footnote{This means $\rho_Y f$ is the restriction to $Y^o$ of an element of $g\in\Sm(Y,\RRd)$ with $g\neq 0$ on $\partial Y$.} Then,
$\Psi=\iota\circ f$ extends to a local $\sct$-map $Y\rightarrow \BBd$. Moreover, the matrix of coefficients of
\[\scd f = \begin{pmatrix}
	\scd f_1\\
	\vdots\\
	\scd f_d
\end{pmatrix}\]
has the same rank as the differential $T\Psi$ of $\Psi$.
\end{lem}
\begin{proof}
Since $\iota$ is a diffeomorphism, $\iota \circ f$ is a smooth map while $\rho_Y>\eps$ and we may thus restrict our attention to a neighbourhood of $\partial Y$ where $\rho_Y |f|$ is everywhere non-vanishing. 
As usual, we pick a suitable collar neighbourhood of product type such that locally $Y=[0,\eps)\times \partial Y$, and we write $\dim(Y)=s$ and 
$\by=(\rho_Y,y)$ for the coordinates. There we need to compute $\Psi^*\rho_{Z}$.
Write
   $ f(\rho_Y, y) = \rho^{-1}_Y h(\rho_Y,y)$
for $h\in \Sm(Y,\RRd)$ with $h(0,y)\neq 0$ for all $(0,y)\in\partial Y$. Since $\rho_Y$ is assumed sufficiently small, 
$|f(\by)|=\rho_Y^{-1}|h(\by)|$ may be assumed sufficiently large and hence 
$$\Psi(\by)=(\iota\circ f)(\by)=\frac{f(\by)}{|f(\by)|}\left(1-\frac{1}{|f(\by)|}\right)=\frac{h(\by)}{|h(\by)|}\left(1-\frac{\rho_Y}{|h(\by)|}\right).$$ 
In this form, $\Psi$ clearly extends up to the boundary.
The boundary defining function on $\BBd$ is, in this coordinate patch,
$\rho_Z=1-|x|$. Thus, 
$$\Psi^*\rho_Z=\frac{1}{|f(\by)|}=\rho_Y \frac{1}{\rho_Y|f(\by)|}.$$
By assumption, $\rho_Y|f(\by)|=|h(\by)|$ is smooth and non-vanishing, which proves that $\Psi$ is an $\sct$-map.

For the second half of the statement we first observe that, since $\iota$ is a diffeomorphism $\RRd\rightarrow (\BBd)^o$ and $\scd$ coincides, up to a rescaling by a non-vanishing factor, with the usual differential in the interior, we may restrict our attention to the boundary $\partial Y$. Then we compute
\begin{align*}
    \scd f(\by) &= 
    \rho_Y^2\partial_{\rho_Y} f(\by)\,\frac{\dd\rho_Y}{\rho_Y^2}+\sum_{j=1}^{s-1}\rho_Y\partial_{y_j} f(\by)\,\frac{\dd y_j}{\rho_Y}\\
    &=(-h(\by)+ \rho_Y \partial_{\rho_Y}h(\by))\,\frac{\dd\rho_Y}{\rho_Y^2} + \sum_{j=1}^{s-1}\partial_{y_j} h(\by) \,\frac{\dd y_j}{\rho_Y}.
\end{align*}
We identify $\scd f$ with its coefficients ($s\times d$)-dimensional block matrix 
$$\begin{pmatrix}
    -h(\by)+ \rho_Y \partial_{\rho_Y}h(\by) & (\partial_{y_j} h(\by))_{j=1,\dots,s-1}
\end{pmatrix}.$$
At the boundary $\rho_Y=0$ we obtain
\begin{align}
\label{eq:scdhorror}
\begin{pmatrix}
-h & (\partial_{y_j} h)_{j=1,\dots,s-1}
\end{pmatrix}\!(0,y).
\end{align}
We want to compare the rank of \eqref{eq:scdhorror} with that of the differential of $\Psi$ at the point $(0,y)\in\partial Y$.
As shown above, the function $\Psi$ is given, at an arbitrary point $\by=(\rho_Y,y)$ close enough to $\partial Y$, by
\begin{align*}
    \frac{h(\by)}{|h(\by)|}\left(1-\frac{\rho_Y}{|h(\by)|}\right),
\end{align*}
whose differential at $(0,y)$ is the block matrix
\begin{align}\label{eq:JPsi}
    T \Psi(0, y) =
    \begin{pmatrix}
        -\frac{h}{|h|^2}+\partial_{\rho_Y}\frac{h}{|h|} & \left(\partial_{y_j} \frac{h}{|h|}\right)_{j=1,\dots,s-1}    
    \end{pmatrix}\!(0,y).
\end{align}
Now observe that, since they are derivatives of unit vectors,
$\partial_{y_j} \frac{h}{|h|}$ and $\partial_{\rho_Y} \frac{h}{|h|}$ are orthogonal to $h$, which is itself non-zero.\footnote{Recall that,
in fact, $|v(t)|=1 \Leftrightarrow v(t)\cdot v(t) = 1\Rightarrow 2v(t)\cdot v^\prime(t)=0\Leftrightarrow v(t) \perp v^\prime(t)$.}
Therefore, the rank of $T\Psi(0,y)$ equals that of the block matrix 
\begin{align}\label{eq:JPsimod}
    \begin{pmatrix}
        -h &
        \left(\partial_{y_j} \frac{h}{|h|}\right)_{j=1,\dots,s-1}
    \end{pmatrix}\!(0,y).
\end{align}
%
Finally, we have that
$$\partial_{y_j} h= \partial_{y_j} \left(|h|\frac{h}{|h|}\right)=\underbrace{|h| \partial_{y_j} \frac{h}{|h|}}_{\text{collinear to }{\partial_{y_j} \frac{h}{|h|}}} + \underbrace{\frac{(h\cdot\partial_{y_j} h)}{|h|^2} h}_{\text{collinear to }h} .$$
This means that the null space (and hence the ranks) of \eqref{eq:scdhorror} and \eqref{eq:JPsimod} coincide.
\end{proof}

\begin{ex}
The simplest example for a map where Lemma \ref{lem:horror} applies is given by the map $f=\iota^{-1}:\BBd\rightarrow \RRd$.
\end{ex}

\begin{rem}
Recall (cf. \cite[App. C.3]{Hormander3}) that the intersection of two $\Sm$-sub\-mani\-folds $Y$ and $Z$ of a $\Sm$-manifold $X$ is \emph{clean} with excess $e\in\NN$ if $Y\cap Z$ is a $\Sm$-submanifold of $X$ satisfying 
\begin{align*}
T_x(Y\cap Z)&=T_xY\cap T_x Z,\qquad \forall x\in Y\cap Z,\\
\codim(Y)+\codim(Z)&=\codim(Y\cap Z)+e.
\end{align*}
%
\end{rem}

\begin{ex}\label{ex:embdball}
    Let $X$ be a mwb and $a \in \rho_X^{-m_e} \rho_{\BB^s}^{-m_\psi} \Sm(X \times \BB^s)$. In this example, we extend $a$ to a local symbol on a suitable subset of $X \times \BB^{s+1}$.

    We view $\BB^{s+1}$ as embedded in $\RR^{s+1}$ with coordinates $(y_1,\dots,y_s,\tilde{y})$.
Define 
\[\jmath : \BB^{s+1} \to \BB^s \times (-1,1), \qquad (y,\tilde{y}) \mapsto \left(\frac{y}{\sqrt{1 - \tilde{y}^2}}, \tilde{y}\right),\]
where $y = (y_1, \dotsc, y_s)$.
For every $\eps \in (0,1)$, we obtain coordinates on
$$U = \jmath^{-1}\left\{\BB^s \times (-\eps, \eps)\right\} = \BB^{s+1} \cap \{|\tilde{y}| < \eps\},$$
cf. Figure \ref{fig:fiberball}.
We note that $U$ is a fibration of base $\BB^s$ and fiber $(-\eps, \eps)$.

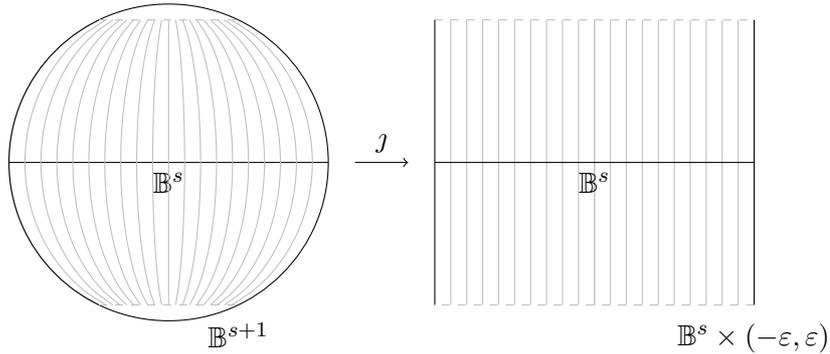
\begin{figure}[!ht]
\begin{center}
\begin{tikzpicture}[
		scale=0.7,
        MyPoints/.style={draw=black,fill=white,thick},
]

	\draw (-3,0) -- (3,0) node[below, midway]{$\mathbb{B}^s$};
	
	\foreach \sss in {-3,-2.7,...,3}{
		\draw[lightgray,domain=-0.9:0.9,smooth,variable=\y] plot ({\sss*sqrt{(1-\y*\y)},3*\y});
		};	

\draw (0,0) circle (3);
\draw[lightgray,dashed] (-{3*sqrt(0.19)},2.7) -- ({3*sqrt(0.19)},2.7);
\draw[lightgray,dashed] (-{3*sqrt(0.19)},-2.7) -- ({3*sqrt(0.19)},-2.7) node[black,below=3pt]{$\mathbb{B}^{s+1}$};

\draw[->] (3.5,0) -- (4.5,0) node[above, midway]{$\jmath$}; 

\begin{scope}[shift={(8,0)}]	
\foreach \sss in {0,0.3,0.6,...,3}{
\draw[-,lightgray] ({\sss},-2.7)--({\sss},2.7);
\draw[-,lightgray] ({-\sss},-2.7)--(-{\sss},2.7);} ;%

\draw (-3,0) -- (3,0) node[below, midway]{$\mathbb{B}^s$};
\draw[dashed,lightgray] (-3,2.7) -- (3,2.7);
\draw (3,2.7) -- (3,-2.7) node[black,below=3pt]{$\mathbb{B}^s \times (-\eps, \eps)$};
\draw[dashed,lightgray] (3,-2.7) -- (-3,-2.7);
\draw (-3,-2.7) -- (-3,2.7);
\end{scope}

\end{tikzpicture}
\caption{The action of $\jmath$ visualized}
\label{fig:fiberball}
\end{center}
\end{figure} 

We verify that $\jmath$ is a $\sct$-map. For this we now view $\BB^s\times(-\eps,\eps)$ as a (non-compact) manifold with 
boundary\footnote{This means we view $\BB^s\times(-\eps,\eps)$ as embedded in the manifold with boundary $\BB^s\times\SSS^1$, which can be embedded in $\SSS^s\times\SSS^1$.
For higher dimension, we embed $(-\eps, \eps)^r \hookrightarrow \mathbb{T}^r$.}
with boundary defining function $\rho_Z=1-[y]$. Observe that near the boundary we have
\begin{align*}
\jmath^*\rho_Z &= 1-\frac{[y]}{\sqrt{1-\tilde{y}^2}}\\
&=(1-\sqrt{[y]^2+\tilde{y}^2})\cdot \frac{1}{ \sqrt{1-\tilde{y}^2} } \cdot \frac{\sqrt{1 - \tilde{y}^2} - [y]}{1 - \sqrt{\tilde{y}^2 + [y]^2}}\\
&=\rho_{\BB^{s+1}}h.
\end{align*}
Since $|\tilde{y}| <  \epsilon$, $h$ is positive and in $\Sm(U)$. Hence $\jmath$ is an $\sct$-map.

As usual, we may perform the same construction fiber-wise on a fiber bundle by considering local product decompositions to obtain a local $\sct$-map. Namely, let $X$ be an arbitrary mwb. Then $\Psi = \id_X \times \jmath$ is again a $\sct$-map on the product $X\times \big(\BBs\times(-\eps,\eps)\big)$. Using Lemma \ref{lem:SGmapproj} and Remark \ref{rem:SGmapcomp}, wee see that $\tilde{\Psi} = \Psi \circ (\id_X \times \pr_{\BB^s}) : X \times U \to X \times \BB^s$
is a $\sct$-map. Hence, $\tilde{\Psi}^* a \in \rho_X^{-m_e} \rho_{\BB^{s+1}}^{-m_\psi} C^\infty(X \times U)$.
\end{ex}

\section{Phase functions and Lagrangian submanifolds}
\label{sec:phaseandlag}

\subsection{Clean phase functions}

\begin{defn}[Phase functions]
Let $X$ and $Y$ be mwbs, $B=X\times Y$. Let $U$ be an open subset in $B$. Then, a real valued $\varphi\in \rho_X^{-1}\rho_Y^{-1}\Sm(U)$ is a \emph{local} ($\sct$-\-)\allowbreak phase function if it is the restriction of some $\widetilde{\varphi}\in \rho_X^{-1}\rho_Y^{-1}\Sm(B)$ to $U$ such that $\scd\tilde{\varphi}(p)\neq 0$ for all $p\in\overline{\Bp}\cap\overline{\partial U}$.

If $U=B$, that is $\varphi\in \rho_X^{-1}\rho_Y^{-1}\Sm(B)$ with $\scd\varphi(p)|_{\overline{\B^\psi}}\neq 0$, we call $\varphi$ a \emph{global} $\sct$-phase function.
\end{defn}

\begin{rem}
Phrased differently, if $U$ is an interior open set, $\varphi$ is just a smooth function.
In the non-trivial case of $U$ being a boundary neighbourhood, the above definition means that, for every $p\in\partial B$ in the $\psi$- or $\psi e$-component of the boundary of $U$, there exists an element $\zeta \in\scV(B)$ such that $\zeta (\varphi)$ is elliptic at $p$, meaning $\zeta (\varphi)\in\Sm(X\times Y)$ satisfies $\big(\zeta\varphi\big)(p)\neq 0$. It is, by compactness, bounded away from zero at the possible limit points in $\overline{\partial U}$. In the following, we usually do not write $\widetilde{\varphi}$ but simply identify $\widetilde{\varphi}$ and $\varphi$ at these limit points.
\end{rem}

\begin{ex}[$\SG$-phase functions]
If $B=\BBd\times\BBs$, such $\varphi$ correspond to so-called (classical) $\SG$-phase functions on $\RRd\times\RRs$, cf. \cite{CoSc,CoSc2}, but with a relaxed condition as $\|x\|\rightarrow \infty$. Indeed, in light of the $\SG$-estimates \eqref{eq:SGest}, the previous definition translates to
\begin{equation}
\label{eq:SGphaseineq}
|\jap{x}^{-1} \nabla_\theta\varphi|^2+|\jap{\theta}^{-1}\nabla_x\varphi|^2\geq C\quad \text{for} \quad |\theta|\gg 0.
\end{equation} 
The relationship between these and ``standard'' phase functions which are homogeneous in $\theta$ is discussed in \cite{CoSc2}. Examples of $\SG$-phase functions are the standard Fourier phase $x\cdot \theta$ on $\RR^d_x\times\RR^d_\theta$ and $x_0\jap{\theta}-x\cdot \theta$ on 
$\RR_{x_0,x}^{d+1}\times\RR^d_\theta$.
\end{ex}

\begin{defn}[The set of critical points]
Let $B=X\times Y$, $\varphi\in \rho_X^{-1}\rho_Y^{-1}\Sm(B)$ a (local) phase function. A point $p\in B$ (in the domain of $\varphi$) is called a \emph{critical point} of $\varphi$ if $\scdY\varphi(p)=0$, that is, if
$\zeta(\varphi)(p)=0$ for every $\zeta \in\scVY(B).$
We define
\begin{equation}
C_\varphi=\{p\in B\,|\, \scdY\varphi(p) = 0 \}.
\end{equation}
We set $\Cp=C_\varphi\cap \B$ and specify
$$\Cp^\bullet = \Cp\cap \B^\bullet \quad \text{for} \quad \bullet \in\{e,\psi,\psi e\}. $$
\end{defn}

We now adapt the usual definition of a \emph{clean} phase function from the classical setting to the case with boundary.

\begin{defn}[Clean phase functions]
\label{def:cleanphase}
A phase function $\varphi$
is called \emph{clean} if the following conditions hold:

\begin{itemize}
\item[1.)] there exists a neighbourhood $U\subset B$ of $\partial B$ such that $C_\varphi\cap U$ is a manifold with corners with $\partial C_\varphi\subset \partial B$;
\item[2.)] the tangent space of $T_pC_\varphi$ is at every point $p$ given by those vectors in $v\in T_p B$ such that $v(\zeta(\varphi))=0$ for all $\zeta\in \scVY$, that is, $T(\scdY\varphi)v=0$;
\item[3.)] the intersections $\Cp^\bullet=C_\varphi\cap\B^\bullet$ are clean.
\end{itemize}
\end{defn}

The last condition is equivalent to the existence of $w \in T_{\Cp^\bullet}\Cp^\bullet$ such that
\begin{align}\label{eq:cleanbdry}
    (T\scd_Y\varphi)(w + \partial_{\rho_\bullet}) = 0.
\end{align}
This means that, for some $w$ tangent to $\B^\bullet$, we have $w + \partial_{\rho_\bullet} \in T_{\Cp^\bullet} \Cp$. Here, $\rho_\bullet$ is a bdf of $\B^\bullet$. We now discuss the implications of these conditions.

\begin{lem}
\label{lem:Cpprops}
Let $\varphi$ be a clean phase function. Then either we are in the ``non-corner crossing case'' $1a.)$ or in the ``corner crossing case'' $1b.)$, namely,
\begin{enumerate}[label=1\alph*.)]
\item both $\Cpe$ and $\Cpp$ are closed manifolds (without boundary) and $\Cppe=\emptyset$; 
\item $\Cp$ consists of two components, $\overline{\Cpe}$ and $\overline{\Cpp}$, which are both submanifolds (with boundary), of the same dimension $\dim(C_\varphi)-1$, with joint boundary $\Cppe=\partial \overline{\Cpe}=\partial \overline{\Cpp}$ of $\B$. The intersection of $\overline{\Cpe}$ and $\overline{\Cpp}$ in $\Cppe$ is again clean.
\end{enumerate}
In both cases, the differential of $\scdY\varphi:B\rightarrow \scT^*B$, viewed as a map $T(\scdY\varphi):TB\rightarrow T(\scT^*B)$, characterizes $T\Cp^\bullet$:
\begin{enumerate}[label=\arabic*.)]
\setcounter{enumi}{1}
\item \label{it:Cpprops2} The tangent space of $\overline{\Cpe}$ and $\overline{\Cpp}$ at each point $p$ is given by those vectors $v\in T\B^\bullet$ such that $v(\zeta(\varphi))=0$ for all $\zeta\in \scVY$, that is $T(\scdY\varphi)v=0$.
\end{enumerate}
%
\end{lem}

By condition $3.)$ of Definition \ref{def:cleanphase}, we have 
$\dim(\ker(T(\scdY\varphi)))=\dim C_\varphi$. Hence,
the restrictions of $T(\scdY\varphi)$ to the individual boundary components of $B$ on $\Cp$ are of constant rank. Namely,
\[
	\rk(T(\scdY\varphi))=
	\begin{cases}
		s-e		& \text{on $C_\varphi^o$},
		\\
		s-e-1 & \text{on $\Cpp$ and $\Cpe$},
		\\
		s-e-2 & \text{on $\Cppe$},
	\end{cases}
\]
for some fixed number $e$, called the excess of $\varphi$, which is given by
$$e=\dim C_\varphi - d.$$
%

\begin{rem}
Conversely, if the rank of $T(\scdY\varphi)$ is constant \emph{in a neighborhood} of each critical point of $\scdY\varphi$, then $\varphi$ is clean by the constant rank theorem. In case $e=0$, $\varphi$ is called \emph{non-degenerate}, and the two characterizations coincide. The corresponding case of $\SG$-phase functions (on $\RRd$) was studied in \cite{CoSc2}.
\end{rem}

\subsection{The associated Lagrangian}
In the classical local theory without boundary on subsets of $\RRd\times(\RR^s\setminus\{0\})$,  see \cite[Chapter XXI.2]{Hormander3},
the set of critical points $\Cp$ is realized as an immersed Lagrangian in $T^*\RRd$ by the map 
$(x,\theta)\rightarrow (x,\varphi_x^\prime(x,\theta))$.
In the present setting, the situation is more complicated.
Following \cite{CoSc2}, we define an analogous map $\lp$ on the mwc $B=X\times Y$ into $\scOverT^*X$.

For that, we consider the following sequence of maps: Using the ``rescaling identifications'' \eqref{eq:rescal}, we may view $(\bx,\by)\rightarrow \scdX\varphi(\bx,\by)$ as a map in 
$\rho_Y^{-1}\Sm(Y,\scTheta(X))$. Since $\scTheta(X)$ are the sections of $\scOverT^*X$, composing with the radial compactification yields, in view
of Lemma \ref{lem:horror}, a map into the compactified fibers of $\scOverT^*X$.
\begin{defn}
The map $\lambda_\varphi: B\rightarrow \scOverT^*X$ is defined by 
$$(\bx,\by)\mapsto \big(\bx,\iota(\scdX\varphi(\bx,\by))\big).$$
\end{defn}

\begin{lem}
\label{lem:lpsct}
There is a neighbourhood $U\subset B$ of $\Cp$ such that $\lp: U\rightarrow \scOverT^*X$ is a local $\sct$-map.
\end{lem}
\begin{proof}
We write, $\bx=(\rho_X,x)$, $\by=(\rho_Y,y)$ for coordinates in $B$, $\bx$ and $\bxi=(\rho_\Xi,\xi)$ for coordinates in $\scOverT^*X$. Since $\lambda_\varphi$ is the identity in the first set of variables, we have
$\lambda_\varphi^*\bx=\bx.$
In the second set of variables, $\lambda_\varphi$ acts as $\iota\,\circ\,\scdX\varphi$, with $\scdX\varphi\in\rho_Y^{-1}\Sm(Y,\scTheta(X))$. Notice that on $\Cpp\cup\Cppe$, we have $\scdX\varphi(\bx,\by)\neq 0$, since $\scd \varphi\neq 0$ on $\Bp\cup\Bpe$ and $\scdY\varphi=0$ on $\Cp$. Hence, due to compactness, we may find a neighbourhood of $\Cpp\cup\Cppe$ on which $\scdX\varphi(\bx,\by)\neq 0$.
Writing $\varphi=\rho_X^{-1}\rho_Y^{-1}f$ for $f\in\Sm(X\times Y)$, this means
$$(-f+\rho_X\partial_{\rho_X}f) \frac{\dd \rho_X}{\rho_X^2\rho_Y} + \sum_{j=1}^{d-1}\partial_{x_j}f\frac{\dd x_j}{\rho_X\rho_Y} \neq 0.$$
Rescaling and viewing $\scdX \varphi$ as a map in $\rho_Y^{-1}\Sm(Y,\scTheta(X))$, we express $\scdX\varphi$ as
\begin{equation}
\label{eq:scdxexpl}
\scdX\varphi=\rho_Y^{-1} \left((-f+\rho_X\partial_{\rho_X}f) \frac{\dd \rho_X}{\rho_X^2} + \sum_{j=1}^{d-1}\partial_{x_j}f\frac{\dd x_j}{\rho_X} \right).
\end{equation}
Composing with $\iota$, we are therefore in the situation of Lemma \ref{lem:horror}, up to additional smooth dependence on the $X$-variables, and conclude that $\lambda_\varphi$ is a local $\sct$-map.

On $\Cpe$, away from $\Cppe$, we have that $\rho_Y\neq 0$ and correspondingly $\scdX\varphi(\bx,\by)$ stays bounded. Since $\iota$ maps bounded arguments into the interior, 
we find $\lp^*\rho_\Xi\neq 0$. Since $\lp$ is smooth, $\lp$ is an $\sct$-map.
\end{proof}
In particular, $\iota(\scdX\varphi(\bx,\by))$ maps boundary points with $\rho_Y=0$ to boundary points of the fiber, that is to $\Wp\cup \Wpe$.
\begin{defn}
We define $L_\varphi=\lp(C_\varphi)$ and $\Lp:=\lp(\Cp)$. We further write $\Lp^\bullet$ for $\lp(\Cp^\bullet)\subset \Wt^\bullet$ for $\bullet\in\{e,\psi,\psi e\}$. We say that $\varphi$ parametrizes $L_\varphi$ and $\Lp$. 
\end{defn}
\begin{thm}
\label{thm:lpsubm}
The map $\lp: \Cp \rightarrow \scOverT^*X$ is of constant rank $d$. Its image $L_\varphi$ as well as the boundary and corner faces $\Lp^\bullet=\lp(\Cp^\bullet)$ are immersed manifolds of dimension $\dim\Lp^\bullet=\dim\Cp^\bullet-e$. Furthermore, $\lp:\Cp\rightarrow \Lp$ is a submersion.
\end{thm}
The proof is inspired by that of Lemma 2.3.2 in \cite{Duistermaat} (adapted to clean phase functions), but much more involved, due to the presence of the compactification. We treat this new phenomenon by carefully applying  Lemma~\ref{lem:horror}.

%
\begin{proof}
We obtain the rank of $T\lp$ for $\lp: \Cp \rightarrow \scOverT^*X$ by computing the dimension of its null space.
Let $v = \delta\rho_X \cdot \partial_{\rho_X} + \delta x\cdot \partial_x + \delta\rho_Y\cdot \partial_{\rho_Y} + \delta y\cdot \partial_y$ be a vector at a point
$p = (\rho_X,x,\rho_Y,y) \in \Cp$. For the moment, we assume $\rho_Y>0$. We write $\lp = (\id \times \iota) \circ \ellp$ with
\begin{align*}
    \ellp : X\times Y^o \rightarrow \scT^*X\qquad
    (x,y) \mapsto (x,\scd_X \varphi(x,y)).
\end{align*}
Assume that $T\ellp(p)v = 0$ and $v \in T_p \Cp$.
The condition $T\ellp(p)v = 0$ implies that $\delta\rho_X = 0$ and $\delta x = 0$.
Let $\tilde{v} = \delta\rho_Y\cdot \partial_{\rho_Y} + \delta y\cdot \partial_y$. Hence the assumptions are reduced to
\begin{equation}\label{eq:VscdYX}
    \begin{aligned}
        \tilde{v} \scd_X \varphi(p) &= 0,\\
        \tilde{v} \scd_Y \varphi(p) &= 0,
    \end{aligned}
\end{equation}
where $\tilde{v}$ is interpreted as acting on the coefficient functions of the differentials.

In coordinates, these coefficient functions are given by
\begin{align*}
    \scd_X \varphi(p) = \rho_Y^{-1}(-f + \rho_X\partial_{\rho_X} f, \partial_x f)(p), \qquad
    \scd_Y \varphi(p) = (-f + \rho_Y \partial_{\rho_Y} f, \partial_y f)(p).
\end{align*}

On $\Cp$, where $-f + \rho_Y \partial_{\rho_Y} f = 0$ and $\partial_y f = 0$ hold true, it is easily seen that \eqref{eq:VscdYX} is
equivalent to
\begin{align}\label{mat:scdX}
    \begin{pmatrix}
        \rho_X\rho_Y^{-2} (\rho_Y \partial_{\rho_Y} - 1) \partial_{\rho_X} f & \rho_X\rho_Y^{-1} \partial_{\rho_X}\partial_y f\\
        \rho_Y^{-2}(\rho_Y \partial_{\rho_Y} - 1) \partial_x f & \rho_Y^{-1}\partial_x\partial_y f\\
        \rho_Y \partial_{\rho_Y}\partial_{\rho_Y} f & \rho_Y\partial_{\rho_Y}\partial_y f\\
        \partial_{\rho_Y}\partial_y f & \partial_y \partial_y f
    \end{pmatrix}
    \begin{pmatrix}
        \delta \rho_Y\\
        \delta y
    \end{pmatrix}
    = 0.
\end{align}

The cleanness condition translates to the dimension of the nullspace of $T\scd_X \varphi$ being constantly $e$. We identify $T\scd_Y\varphi$ with the matrix
\begin{align}\label{mat:clean}
    J = 
    \begin{pmatrix}
        (\rho_Y \partial_{\rho_Y}-1) \partial_{\rho_X}f & \partial_y \partial_{\rho_X} f\\
        (\rho_Y \partial_{\rho_Y}-1) \partial_x f & \partial_y\partial_x f\\
        \rho_Y\partial_{\rho_Y}\partial_{\rho_Y} f & \partial_y\partial_{\rho_Y} f\\
        \rho_Y \partial_{\rho_Y}\partial_y f & \partial_y\partial_y f
    \end{pmatrix}.
\end{align}
The matrices appearing in \eqref{mat:scdX} and \eqref{mat:clean} are related by
\begin{align*}
    J
    =
    \begin{pmatrix}
        \rho_Y\rho_X^{-1} & 0 & 0 & 0\\
        0 & \rho_Y & 0 & 0\\
        0 & 0 & \rho_Y^{-1} & 0\\
        0 & 0 & 0 & 1
    \end{pmatrix}
    \begin{pmatrix}
        \rho_X\rho_Y^{-2} (\rho_Y \partial_{\rho_Y} - 1) \partial_{\rho_X} f & \rho_X\rho_Y^{-1} \partial_{\rho_X}\partial_y f\\
        \rho_Y^{-2}(\rho_Y \partial_{\rho_Y} - 1) \partial_x f & \rho_Y^{-1}\partial_x\partial_y f\\
        \rho_Y \partial_{\rho_Y}\partial_{\rho_Y} f & \rho_Y\partial_{\rho_Y}\partial_y f\\
        \partial_{\rho_Y}\partial_y f & \partial_y \partial_y f
    \end{pmatrix}
    \begin{pmatrix}
        \rho_Y & 0\\
        0 & 1
    \end{pmatrix}.
\end{align*}

This proves that \eqref{eq:VscdYX} is equivalent to $v \in \ker T\scd_Y \varphi$ under our assumptions $\rho_Y > 0$ and $\rho_X > 0$, and the rank of
$\ellp$ is given by
\begin{align*}
    \rk \ellp &= \dim T_p \Cp - \dim \ker T\scd_Y \varphi
    = (d + e) - e=d.
\end{align*}

Now assume that $\rho_X = 0$. We see that the first row of \eqref{mat:scdX} vanishes identically,
but we have the additional condition \eqref{eq:cleanbdry}, implying that, at $\rho_X=0$, the first row of \eqref{mat:clean} depends linearly on the other rows.
Therefore, the rank of $\ellp$ is still $d$ at points with $\rho_X=0$. The composition with $\id \times \iota$ changes nothing for $\rho_Y > 0$, since $\iota$ is a diffeomorphism there.

To perform the limit $\rho_Y \rightarrow 0$, we have to examine carefully the effect of the presence of the 
compactification $\iota$, in the spirit of the proof of Lemma \ref{lem:horror}.
For $v \in T_p \Cp$ such that $T\lp(p)v = 0$, that is, as above, of the form
\[v = \delta\rho_Y\cdot \partial_{\rho_Y} + \delta y\cdot \partial_y,\]
we now obtain the set of equations
\begin{equation}\label{eq:iotaVscdYX}
    \begin{aligned}
        v \big(\iota\,\scd_X \varphi\big)(p) &= 0,\\
        v \scd_Y \varphi(p) &= 0,
    \end{aligned}
\end{equation}
which are equivalent to the set of equations
\begin{align}\label{mat:iotaVscd}
    \begin{pmatrix}
        \partial_{\rho_Y} \iota \scd_X \varphi & \partial_y \iota \scd_X \varphi\\
        \partial_{\rho_Y} \partial_y f & \partial_y \partial_y f
    \end{pmatrix}
    \begin{pmatrix}
        \delta \rho_Y\\
        \delta y
    \end{pmatrix}
    = 0.
\end{align}
We need to compare the rank of the coefficient matrix in \eqref{mat:iotaVscd} with that of $T\scd_Y \varphi$ at points of the form $(\rho_X,x,0,y)$. 
For this purpose, we go through a series of ``reductions'', along the lines of the proof of Lemma \ref{lem:horror}, to simplify the comparison.
First, we can identify $\scd_X\varphi$ with
\[
	\rho_Y^{-1}\begin{pmatrix}- f+\rho_X\partial_{\rho_X}f \\ \partial_x f \end{pmatrix}=:\rho_Y^{-1} h.
\]
Note that $h\neq 0$ near $\overline{\Cpp}$, since $\varphi$ is a phase function.
As in the proof of Lemma \ref{lem:horror}, the evaluation at $(\rho_X,x,0,y)$ then gives
\begin{align}
\label{eq:matma}
    \begin{pmatrix}
        \partial_{\rho_Y} \iota \scd_X \varphi & \partial_y \iota \scd_X \varphi\\
        \partial_{\rho_Y} \partial_y f & \partial_y \partial_y f
    \end{pmatrix}=    \begin{pmatrix}
        -\frac{h}{|h|^2}+\partial_{\rho_Y}\frac{h}{|h|} & \partial_{y} \frac{h}{|h|} \\
        \partial_{\rho_Y} \partial_y f & \partial_y \partial_y f
    \end{pmatrix}.
\end{align}
Since all derivatives of $\frac{h}{|h|}$ are orthogonal to $\frac{h}{|h|}$ and $h\neq 0$, the rank of the matrix \eqref{eq:matma} equals the one of 
\begin{align}
\label{eq:matmat}
\begin{pmatrix}
        -\frac{h}{|h|^2} & \partial_y \frac{h}{|h|}\\
        0 & \partial_y \partial_y f
    \end{pmatrix}.
\end{align}
In fact, in \eqref{eq:matma}, as well as in \eqref{eq:matmat}, the first column is linearly independent of the others. 
Now we write
$$\partial_{y_j} \frac{h}{|h|}=\frac{1}{|h|}\partial_{y_j} h- \underbrace{\frac{(h\cdot\partial_{y_j} h)}{|h|^3} h}_{\text{collinear to }h},$$
and remove the collinear summands, which again does not change the rank of the matrix \eqref{eq:matmat}. Therefore, the rank of \eqref{eq:matma} is the same
as the one of 
\begin{align}
\label{eq:matmat2}
\begin{pmatrix}
        -\frac{h}{|h|^2} & \frac{1}{|h|}\partial_y h\\
        0 & \partial_y \partial_y f
    \end{pmatrix}.
\end{align}
Multiplying the first $d$ rows and the first column of \eqref{eq:matmat2} by the non-vanishing factor $|h|$, again the rank does not change, and we can look at

\begin{align}
\label{eq:matmat3}
\begin{pmatrix}
        -h & \partial_y h\\
        0 & \partial_y \partial_y f
    \end{pmatrix}=\begin{pmatrix}
        f-\rho_X\partial_{\rho_X}f & -\partial_yf + \rho_X\partial_y\partial_{\rho_X} f\\
        -\partial_x f & \partial_y\partial_x f\\
        0 & \partial_y \partial_y f
    \end{pmatrix}.
\end{align}
On $\Cp$ at $\rho_Y=0$ this equals
\begin{align}
\label{eq:matmat4}
\begin{pmatrix}
        -\rho_X\partial_{\rho_X}f & \rho_X\partial_y\partial_{\rho_X} f\\
        -\partial_x f & \partial_y\partial_x f\\
        0 & \partial_y \partial_y f
    \end{pmatrix}.
\end{align}
Finally, we observe that the dimension of the null space of \eqref{eq:matmat4} is, 
by cleanness of $\varphi$ (in particular by \eqref{eq:cleanbdry} applied to $\Cpp$ or $\Cppe$), the same as the one of
\begin{align}
\label{eq:cleandiff}
    \begin{pmatrix}
        - \partial_{\rho_X}f & \partial_y \partial_{\rho_X} f\\
        - \partial_x f & \partial_y\partial_x f\\
        0 & \partial_y\partial_{\rho_Y} f\\
        0 & \partial_y\partial_y f
    \end{pmatrix} = T\scd_Y \varphi|_{\Cpp},
\end{align}
namely $e$.
Therefore, the rank of $\lp$ equals $d=(d+e)-e$ near $\Cp$, which concludes the proof.
\end{proof}
\begin{lem}
\label{lem:lpfibration}
The map $\lp: C_\varphi\rightarrow L_\varphi$ is a local fibration and the fiber is everywhere a smooth manifold without boundary.
\end{lem}
\begin{proof}
Since $\lp$ is locally an $\sct$-map, $T\lp$ maps the set of vectors at the boundary that are inwards pointing into itself, see Remark \ref{rem:inward}. Therefore $\lp$ is a so-called ``tame'' submersion in the sense of \cite[Lemma 1.3]{Nistor}. As such, it is a local fibration and the fiber is a manifold without boundary.
\end{proof}
\subsection{Symplectic properties of the associated Lagrangian}
\label{sec:symp}
As in the classical theory, $L_\varphi$ is an immersed Lagrangian submanifold, and its boundary faces $\Lambda^\bullet$ are immersed Legendrian submanifolds. Let us briefly recall these concepts. For more information, the reader is referred to \cite{CoSc2,MZ,HV}.

As a cotangent space, $T^*X^o$ carries a natural symplectic $2$-form $\omega$ induced by the canonical $1$-form $\alpha\in\Sm(T^*X^o,T^*(T^*X^o))$ as $\omega=\dd\alpha$. This $1$-form can be recovered from $\omega$ by setting $\alpha=\varrho^\psi\lrcorner\,\,\omega$ for the radial vector field $\varrho^\psi$ on $\Sm(T^*X^o)$, which is given by $\varrho^\psi=\xi\cdot\partial_{\xi}$ in canonical coordinates. \\
We now write $(\bx,\bxi)=(\rho_X,x,\rho_\Xi,\xi)$ for the coordinates in the mwc $\scOverT^*X$ which are obtained from the rescaled canonical coordinates under radial compactification in the fiber, cf. \cite{MZ}. Then $\varrho^\psi$ corresponds to $\rho_\Xi\partial_{\rho_\Xi}$ on $\Sm(\overline{T}^*X^o)$. For the purpose of scattering geometry, it is natural to rescale further and define, on $T^*(\scOverT^*X)^o$,
$$\alpha^\psi:=\rho_\Xi^2\partial_{\rho_\Xi}\lrcorner\,\omega.$$
There exists another form of interest, namely
\begin{align*}
\alpha^e:=\rho_X^2\partial_{\rho_X}\lrcorner\,\omega.
\end{align*} 
We now extend these forms to $T^*(\scOverT^*X)$ and define the boundary restrictions of $\alpha^\bullet$.
Observe that, while their explicit form depends on the choice of bdfs, the induced contact structure at the boundary does not, see next Lemma \ref{lem:alphaext}
\begin{lem}\label{lem:alphaext}
The forms $\alpha^\bullet$ extend to $1$-forms on $\Wt^\bullet$, denoted by the same letter. The induced contact structures do not depend on the choice of bdfs. 
\end{lem}
\begin{ex}
On $T^*\RRd\cong \RRd\times\RRd$, with canonical coordinates $(x,\xi)$, the vector fields $\varrho^\psi$ and $\varrho^e$ correspond to $\varrho^\psi=\xi\cdot\partial_\xi$ and $\varrho^e=x\cdot\partial_x$. The symplectic $2$-form is $\sum_j\dd\xi_j\wedge\dd x_j$ and hence
$$\varrho^\psi\lrcorner\,\omega=\xi\cdot \dd x\quad \text{and}\quad\varrho^e\lrcorner\,\omega=-x\cdot \dd \xi.$$
Obviously, the coefficients of these forms diverge as $[\xi]\rightarrow \infty$ and $[x]\rightarrow \infty$. The rescaled forms ``at the boundary at infinity'' then correspond to 
$$\alpha^\psi=\frac{\xi}{[\xi]}\cdot \dd x\quad \text{and}\quad \alpha^e=-\frac{x}{[x]}\cdot \dd \xi.$$
After a choice of coordinates near the respective boundaries, this is the general local geometric situation.
\end{ex}
We are now in the position to formulate the symplectic properties of $\Lp$, cf. \cite{CoSc}. Recall that a submanifold $N$ of a symplectic manifold $(M,\omega)$ is Lagrangian if $\omega|_{TN}=0$ and a submanifold $N$ of a contact manifold $(M,\alpha)$ is Legendrian if $\alpha|_{TN}=0$.
\begin{prop}
The immersed manifolds defined in Theorem \ref{thm:lpsubm} satisfy:
\begin{itemize}
\item[1.)] $L_\varphi^o$ is an immersed Lagrangian submanifold with respect to the $2$-form $\omega$ on $(\scOverT^* X)^o\cong T^*X$;
\item[2.)] $\Lpp$ is Legendrian with respect to the canonical $1$-form $\alpha^\psi$ on $\Wp\cong S^*(X^o)$;
\item[3.)] $\Lpe$ is Legendrian with respect to the $1$-form $\alpha^e$ on $\We\cong T^*_{\partial X}X$.
\end{itemize}
\end{prop}
We take this as the definition of an $\sct$-Lagrangian, cf. \cite{CoSc2}.
\begin{defn}[$\sct$-Lagrangians]\label{def:scLagr}
Let $\Lambda:=\overline{\Lambda^\psi}\cup\overline{\Lambda^e}\subset \Wt$. $\Lambda$ is called an $\sct$-Lagrangian if:
\begin{itemize}
    \item[1.)] $\Lambda^\psi=\Lambda\cap\Wp$ is Legendrian with respect to the canonical $1$-form $\alpha^\psi$ on $\Wp = \scS^*_{X^o}X$;
\item[2.)] $\Lambda^e=\Lambda\cap\We$ is Legendrian with respect to the $1$-form $\alpha^e$ on $\We = \scT^*_{\partial X}X$;
\item[3.)] $\overline{\Lambda^\psi}$ has a boundary if and only if $\overline{\Lambda^e}$ has a boundary, and, in this case,
$$\Lambda^{\psi e}:=\partial \overline{\Lambda^\psi}=\partial \overline{\Lambda^e}=\overline{\Lambda^\psi}\cap\partial \overline{\Lambda^e},$$
with clean intersection.
\end{itemize}
\end{defn}
Figure \ref{fig:Lpintersect}, which is taken from \cite{CoSc2}, summarizes, schematically, the relative positions of $\Lpe$ and $\Lpp$ near the corner in $W$.
\begin{figure}[ht!]
\begin{center}
\begin{tikzpicture}[scale=1.2]
  \node (A) at (2.5,3.5) {$\Wp$};
  \node (B) at (6,0.5) {$\We$};
  \node (C) at (6.7,2.7) {$\Wpe$}; 
  \node (D) at (3.2,1) {$\Lambda^{\psi e}$}; 
  \draw[opacity=0.5,->] (A) -- (3.2,2.8);
  \draw[opacity=0.5,->] (B) -- (5.5,1.5);
  \draw[opacity=0.5,->] (C) -- (5.6,2.8);
  \draw[opacity=0.5,->] (D) -- (3.95,1.95);  
  \draw[opacity=0.75] (3,1.5) -- (6,3);
  \draw[dotted,opacity=0.5] (1.5,2.25) -- (2.25,1.875);
  \draw[dotted,opacity=0.75] (2.25,1.875) -- (3,1.5) -- (3,0);
  \draw[dotted,opacity=0.75] (4.5,3.75) -- (5.25,3.375);
  \draw[dotted,opacity=0.75] (5.25,3.375) -- (6,3) -- (6,1.5);
  \fill[opacity=0.04] (3,0) -- (3,1.5) -- (6,3) -- (6,1.5);
  \fill[opacity=0.01] (1.5,2.25) -- (3,1.5) -- (6,3) -- (4.5,3.75);
  \draw[->,ultra thick] (4,2) -- (5,2.5) node [below] {$\quad x,\xi$};
  \draw[->,ultra thick] (4,2) -- (3,2.5) node [left] {$\rho_X$};
  \draw[->,ultra thick] (4,2) -- (4,1) node [left, below] {$\rho_\Xi$};
  \draw[thick] (4,2) .. controls (4,1.1) and (4.8,1.3) .. (4.7,0.85) node [right] {$\Lambda^e$};%
  \draw[thick] (4,2) .. controls (3,2.5) and (4,3) .. (3.5,3.25) node [right] {$\ \Lambda^\psi$};%
\end{tikzpicture}
\caption{Intersection of $\Lambda^\psi\subset\Wp$ and $\Lambda^e\subset\We$ at the corner $\Wpe$}
\label{fig:Lpintersect}
\end{center}
\end{figure}
We may take the analysis one step further in order to stress the Legendrian character of the boundary components near the corner and to reveal the symplectic properties of $\Lambda^{\psi e}$ by blow-up. For the sake of brevity here, we move this analysis to the appendix, Section \ref{sec:blowup}.

We may sum up our previous analysis by stating the next Theorem \ref{thm:imphlagr}.
\begin{thm}\label{thm:imphlagr}
For a clean phase function $\varphi$, the image $\Lp$ under $\lp$ of $\Cp$ is an immersed $\sct$-Lagrangian.
\end{thm}
\begin{defn}
We say that an $\sct$-Lagrangian $\Lambda$ is locally parametrized by a phase function $\varphi$ if, over the domain of definition of $\varphi$, we have $\Lambda=\Lp$.
\end{defn}
In particular, if $\Lambda$ is locally parametrized by a phase function, then it is admissible. Conversely, we have the following result, cf. \cite{CoSc2}.
\begin{prop}
\label{prop:locpar}
If $\Lambda$ is an $\sct$-Lagrangian, then it is locally parametrizable by a clean phase function $\varphi$, that is $\Lambda^\bullet\cap U^\bullet=\Lp^\bullet\cap U^\bullet$ for some open $U\subset \Wt^\bullet$. In particular, $\Lambda$ arises as the boundary of some Lagrangian submanifold $L_\varphi$ of $\scOverT^* X$.
\end{prop}
\begin{rem}
The proof of Proposition \ref{prop:locpar} in \cite{CoSc2} is based on concrete parametrizations in $\RRd\times\RRd$. It applies here nonetheless, since any $d$-dimensional manifold with boundary $X$ can be locally modelled by $\BBd$. Hence, $\scOverT^*X$ can be locally modelled by $\BBd\times\BBd$ and thus, under inverse radial compactification (applied to both factors), by $\RRd\times\RRd$. 
Note that in \cite{CoSc2} we imposed additional conditions, namely 
\begin{equation}
\label{eq:nonzerosec}
\Lambda^e\cap (\partial X\times\iota(\{0\}))=\emptyset,
\end{equation}
and that $x\cdot \xi=0$ in local canonical coordinates on $\Lambda^{\psi e}$, since this is always true for a parametrized Lagrangian (see \eqref{eq:conormbi} below). However, condition \eqref{eq:nonzerosec} is equivalent to the stronger assumption that $\scd\varphi\neq 0$ also on $\Be$, which we do not impose here. The assumption $x\cdot \xi=0$, in turn, is superfluous, since it already follows from the symplectic assumptions on $\Lambda^{\psi e}$, as we now show.

Assume that both $\xi\cdot \dd x\equiv 0$ and $-x\cdot \dd \xi\equiv 0$ on a bi-conic submanifold $L$ of $\RR^d\times\RR^d$. Then we must have $\dd(x\cdot\xi)=0$. However, when $|x|$ and $|\xi|$ tend to $\infty$, this blows up unless $x\cdot \xi=0$. This shows that $x\cdot \xi=0$ is indeed automatically fulfilled.

This corresponds to the fact that, for the bi-homogenous principal symbol of a phase function $\varphi^{\psi e}$, we have, when 
$\nabla_\theta\varphi(x,\theta)=0$, that (cf. \cite{CoSc2})
\begin{equation}
\label{eq:conormbi}
\langle x,\nabla_x\varphi(x,\theta)\rangle=\varphi(x,\theta)=\langle \theta,\nabla_\theta\varphi(x,\theta)\rangle=0,
\end{equation}
where we have used Euler's identity for homogeneous functions twice.
\end{rem}
\subsection{Scattering conormal bundles}
\label{sec:conorm}
In this section, we consider the simple example of a scattering conormal bundle. Consider a $k$-dimensional submanifold $X'\subset X$ which intersects the boundary of $X$ cleanly or not at all (called $p$-submanifold in \cite{Melrosemwc}). In the following, we assume an intersection with the boundary.
Then there exist local coordinates $(\rho_X,x^\prime,x'')$ such that $X'$ is locally given by
$$X'=\{(\rho_X,x^\prime,x'')\mid \rho_X\geq 0, x^\prime=0\in\RR^{d-1-k}, x''\in\RR^{k-1}\}.$$
We can now consider the compactified scattering conormal $\scOverN^*X'\subset\scOverT^*_{X'}X$. The boundary faces of $\scOverN^*X'$ constitute a Lagrangian.

In fact, write $X=\iota(\RRd)$, so that $X'$ corresponds to a subspace of $\RRd$ of the form
$$X'=\{(x^\prime,x'')\mid x'=0\in\RR^{d-k}, x''\in\RR^{k}\}.$$
We can then introduce $Y=\iota(\RR^{d-k})$ and $\phi(x,y)=x'\cdot y$ on $\RRd\times\RR^{d-k}$, which is an $\SG$-phase function, taking into account \eqref{eq:SGphaseineq}. The true phase function on $X\times Y$ is then $(\iota^{-1}\times\iota^{-1})^*\phi$. We can then compute $C_\varphi=X'\times Y$ and $\Lp=\scOverN^*X'$.

Indeed, in the Euclidean setting, $\Lp$ corresponds to the the three conic manifolds
\begin{align*}
\Lpe&=\{(0,x'',\xi^\prime,0)\}\subset (\RRdz)\times\RRd\\
\Lppe&=\{(0,x'',\xi^\prime,0)\}\subset (\RRdz)\times(\RRdz)\\
\Lpp&=\{(0,x'',\xi^\prime,0)\}\subset \RRd\times(\RRdz)
\end{align*}
which have the claimed symplectic properties. Compactification of the $\RRd$-components and projection of the conic $(\RRdz)$-component to the corresponding sphere then yields the compactified notions in $\scOverT^*X$.
\section{Phase functions which parametrize the same Lagrangian}
\label{sec:exchphase}
In this section, we adapt the classical techniques for exchanging the phase function locally parametrizing a given Lagrangian, see \cite[Chapter 8.1]{Treves}, to the setting with boundary. Since $\Lambda_\varphi$, not $L_\varphi$, is our true object of interest, we say that two phase functions $\varphi_i$, $i=1,2$, locally parametrize the same Lagrangian at $p_0\in\Wt$ if $\Lambda_{\varphi_1}=\Lambda_{\varphi_2}$ in a small (relatively) open neighbourhood of $p_0$ in the respective boundary faces.

Our first observation is the following:
\begin{lem}
\label{lem:phaseplussmooth}
If $\varphi\in\rho_X^{-1}\rho_{\BB^s}^{-1}\Sm(X \times \BB^s)$ is a local phase function and $r\in \Sm(X \times \BB^s)$, then $\varphi+r$ is still a local phase function and it parametrizes the same Lagrangian as $\varphi$.
\end{lem}
\begin{proof}
Since $r\in\Sm(X \times \BB^s)$, $\scd r=0$ when restricted to the boundary. Therefore, $\varphi+r$ is still a local phase function. By the same reason, $\mathcal{C}_{\varphi}=\mathcal{C}_{\varphi+r}$. Finally, we have
$$\lambda_{\varphi+r}(\bx,by)=(\bx,\iota(\scd_X(\varphi + r))).$$
Computing $\scd_X(\varphi + r)$ in coordinates, see \eqref{eq:scdxexpl},
$$
\scdX\varphi=\rho_Y^{-1} \left((-f+\rho_X\partial_{\rho_X}f+\rho_Y\rho_X^2\partial_{\rho_X}r) \frac{\dd \rho_X}{\rho_X^2} +
 \sum_{j=1}^{d-1}(\partial_{x_j}f+\rho_Y\rho_X\partial_{x_j}r)\frac{\dd x_j}{\rho_X} \right), 
$$
we observe that at $\rho_X=0,$ the contribution from $r$ vanishes. The same is true in the limit of $\rho_Y\rightarrow 0$ under application of $\iota$, see also Lemma \ref{lem:horror}.
\end{proof}

\subsection{Increasing fiber variables}
Given a clean phase function $\varphi \in \rho_X^{-1}\rho_{\BB^s}^{-1}C^\infty(X \times \BB^s)$ with excess $e$, define $\widetilde{\psi} \in \rho_X^{-1} \rho_{\BB^s}^{-1} C^\infty(X \times \BB^s \times (-\eps, \eps))$ as follows:
\[\widetilde{\psi}(\bx,\by, \tilde y) = \varphi(\bx, \by) + \frac{\tilde y^2}{\rho_X \rho_{\BB^s}}.\]
We see that $\scd\widetilde{\psi} \neq 0$ when $\scd\varphi\neq 0$ and $\scd_{\BB^s\times(-\epsilon,\epsilon)} \tilde{\psi} = 0$ if and only if $\tilde y = 0$ and $\scd_{\BB^s} \varphi = 0$.
Thus, 
$$C_{\widetilde{\psi}} = \left\{(\bx,\by, 0)\mid (\bx,\by) \in C_\varphi\right\},$$ 
which implies that the excess is not changed, and $\Lambda_{\widetilde{\psi}} = \Lambda_{\varphi}$. Summing up, $\psi$ is a local clean phase function in $s+1$ fiber variables with the same excess $e$ as $\varphi$ and (locally) parametrizing the same Lagrangian as $\varphi$.

This construction may once again be moved to balls, by using Example \ref{ex:embdball} and setting $\psi = \Psi^*\widetilde{\psi}$. Then $\psi \in \rho_X^{-1}\rho_{\BB^{s+1}}^{-1}C^\infty(X \times U)$.
Using the fact that $\scd \psi = \Psi^*\widetilde{\psi}$, we see that $\psi$ is a clean phase function parametrizing $\Lambda_{\varphi}$ with excess $e$. 
Again, $X\times\BB^s$ can be exchanged by any relatively open subset, hence starting with local phase functions.

\subsection{Reduction of the fiber variables}\label{subs:fbred}

Starting again from a clean phase function $\varphi \in \rho_X^{-1}\rho_{\BB^s}^{-1}C^\infty(X \times \BB^s)$ with excess $e$, we now construct a (local) phase function $\psi$ in the smallest possible number of phase variables (without changing the excess) which (locally) parametrizes the same Lagrangian.  
The argument is similar to the classical one, but extra attention needs to be paid at to what happens near points with $\rho_Y=0$, namely, we never seek to get rid of $\rho_Y$ as a parameter.
\begin{rem}
In the classical theory, meaning for homogeneous phase functions, it is possible to reduce the number of fiber variables under the assumption that the matrix $\partial^2_{\theta\theta}\varphi(x,\theta)$ has rank $r>0$ on $C_\varphi$.  
However, since a classical phase function $\varphi$ is homogeneous in $\theta$, it holds that $\theta\cdot \nabla_\theta\varphi=\varphi$ and hence the second radial derivative is automatically zero on $C_\varphi$. Furthermore, the radial variable can always be chosen to parametrize $\Lambda_\varphi$.
\end{rem}
We proceed as in the proof of Theorem \ref{thm:lpsubm}. We first recall that, for $p_0\in C_\varphi$, writing
$\varphi=\rho_Y^{-1}\rho_X^{-1}f$ with $f\in C^\infty(X \times \BB^s)$,
we have there
\begin{equation}\label{eq:scDyphi}
    0=\scdY\varphi=\left(-f+\rho_Y\partial_{\rho_Y}f,\partial_{y_k}f \right).
\end{equation}
We then identify $T_Y\scdY\varphi$ in coordinates with the matrix
\begin{equation}
\label{eq:DyDyphi}
J_Y\varphi = \begin{pmatrix}
\rho_Y\partial_{\rho_Y}^2f & -\partial_{y_j}f+\rho_Y \partial_{y_j}\partial_{\rho_Y} f \\
\partial_{\rho_Y}\partial_{y_k} f & \partial_{y_j}\partial_{y_k} f
\end{pmatrix}. 
\end{equation}
We see, using \eqref{eq:scDyphi}, that on $\Cpp \subset \{\rho_Y = 0\}$ this becomes
\begin{equation}
\label{eq:Tyscdphipsi}
J_Y\varphi\big|_{\Cp^\psi}= \begin{pmatrix}
0 & 0 \\
\partial_{\rho_Y}\partial_{y_k} f  & \partial_{y_j}\partial_{y_k} f
\end{pmatrix}.
\end{equation}
Therefore, the rank of this matrix is at most $s-1$. Indeed, we observe that, by \eqref{eq:cleanbdry}, at $\rho_Y=0$ we have $\dd\rho_Y\neq 0$ on $TC_\varphi^\psi$ and hence we can always choose $\rho_Y$ as a parameter to locally describe $C^\psi_\varphi$. 
\begin{rem}
	By the same argument, $\rho_X$ can be chosen
	as a parameter close to $\Be$, while, close to $\Bpe$,
	both $\rho_X$ and $\rho_Y$ can be chosen as parameters to represent 
	$C_\varphi$.
\end{rem}
We now seek to reduce the remaining set of variables under the assumption that
\begin{equation}
\label{eq:DyDyass0}
\text{The matrix }\big(\partial_{y_j}\partial_{y_k} \rho_X\rho_Y\varphi\big)_{jk} \text{ has rank }r>0\text{ at }p_0\in\Cpp\cup\Cppe.
\end{equation}
Since at points where $\rho_Y\neq 0$ the variable $\rho_Y$ behaves like all other variables, the same restriction does not hold near a point $p\in \Cpe$. Here, we simply assume that
\begin{equation}
\label{eq:DyDyass1}
\text{The matrix } T_Y\scdY\varphi \text{ has rank }r>0\text{ at }p_0\in\Cpe.
\end{equation}
Since up to multiplication by $\rho_Y>0$ in one row, \eqref{eq:DyDyphi} is the Hessian of $h$ (with respect to $\by$), this is equivalent to
$\mathrm{rk}(H_Y f)=r>0$.
The two conditions may be summarized into one. Namely, consider the scattering Hessian (with respect to the $\by$-variables) of $\varphi$
\begin{equation}
\begin{aligned}
\scHess_Y\varphi&=\begin{pmatrix}
\rho_Y^2\rho_X\partial_{\rho_Y}\rho_Y^2\rho_X\partial_{\rho_Y}\varphi & \rho_Y\rho_X\partial_{y_j}\rho_Y^2\rho_X\partial_{\rho_Y}\varphi \\
\rho_Y^2\rho_X\partial_{\rho_Y}\rho_Y\rho_X\partial_{y_k}\varphi & \rho_Y\rho_X\partial_{y_j}\rho_Y\rho_X\partial_{y_k}\varphi
\end{pmatrix} 
\\
&= \rho_Y\rho_X\begin{pmatrix}
\rho_Y^2\partial_{\rho_Y}^2f & -\partial_{y_j}f+\rho_Y \partial_{y_j}\partial_{\rho_Y} f \\
\rho_Y\partial_{\rho_Y}\partial_{y_k} f & \partial_{y_j}\partial_{y_k} f
\end{pmatrix}.
\end{aligned} 
\end{equation}
Then $\rho_Y^{-1}\rho_X^{-1}\,\scHess_Y\varphi$ becomes, at a point in $\Cp$:
\begin{align*}
\rho_Y^{-1}\rho_X^{-1}\scHess_Y\varphi&=\begin{pmatrix}
0 & 0 \\
0 & \partial_{y_j}\partial_{y_k} f
\end{pmatrix},\quad\text{ if }p_0\in\Cpp\cup\Cppe; \\
\rho_Y^{-1}\rho_X^{-1}\,^\scat H_Y\varphi&=\begin{pmatrix}
\rho_Y^2\partial_{\rho_Y}^2f & \rho_Y \partial_{y_j}\partial_{\rho_Y} f \\
\rho_Y\partial_{\rho_Y}\partial_{y_k} f & \partial_{y_j}\partial_{y_k} f
\end{pmatrix},\quad\text{ if }p_0\in\Cpe.
\end{align*} 
Notice that we can factorize these matrices as
\begin{equation}
\label{eq:scatHessfactor}
\begin{pmatrix} \rho_Y & 0 \\
0 & \mathbbm{1} \end{pmatrix} 
\begin{pmatrix}
\partial_{\rho_Y}^2f & \partial_{y_j}\partial_{\rho_Y} f \\
\partial_{\rho_Y}\partial_{y_k} f & \partial_{y_j}\partial_{y_k} f
\end{pmatrix}
\begin{pmatrix} \rho_Y & 0 \\
0 & \mathbbm{1} \end{pmatrix},
\end{equation}
the rank of which therefore is, for $\rho_Y\neq 0$, that of the standard Hessian of $f$, $H_Y f$. Therefore, our assumption may be expressed as:
\begin{equation}
\label{eq:DyDyass2}
\text{The matrix } \rho_Y^{-1}\rho_X^{-1}\,^\scat H_Y\varphi \text{ has rank }r>0\text{ at }p_0\in\Cp.
\end{equation}
We may now proceed as in the standard theory and introduce a splitting of variables $\by=(\by^\prime,\by^{\prime \prime})$
such that $(\partial_{\by^{\prime\prime}}\partial_{\by^{\prime\prime}}f)_{jk}$ is an invertible $r\times r$ matrix. We can then apply the implicit function theorem to
$$0=\scdY\varphi=\left(-f+\rho_Y\partial_{\rho_Y}f,\partial_{y_k}f \right)$$
at $p_0$. We obtain a map from an open neighbourhood of $p_0$,
$$k:(\bx,\by^{\prime})\mapsto \big(\bx,\by^\prime,\by^{\prime\prime}(\bx,\by^\prime)\big),$$
such that $C_\varphi$ and the range of $k$ locally coincide. Note that $k$ is a scattering map, since $\rho_Y$ is always one of the $\by^\prime$ near the $\psi$-face.

Then $\varphi_{\red}=\varphi\circ k$ is a clean local phase function in $d\times (s-r)$ variables with excess $e$, and $k$ provides a local isomorphism $C_{\varphi_{\red}}\rightarrow C_\varphi$. Furthermore, at stationary points $p_0$ and $k(p_0)$, we have that $\iota(\scdX \varphi_{\red})=\iota(\scd_X \varphi)$, since $\scdY\varphi=0$ there. 
Hence, $\varphi_{\red}$ locally parametrizes the same Lagrangian as $\varphi$.
\begin{rem}
Note that, after applying a change of coordinates in the $\by$ variables,
$\varphi_{\red}$ may be assumed to be defined on $\BBd\times\BB^{s-r}$, see also Lemma \ref{lem:CLFstarff} below.
\end{rem}
Summing up, we can formulate the next Proposition \ref{prop:fiberred}.
\begin{prop}
\label{prop:fiberred}
Let $\varphi\in\rho_Y^{-1}\rho_X^{-1}\Sm(X\times\BB^s)$ be a local clean phase function of excess $e$. Assume 
$$\rho_Y^{-1}\rho_X^{-1}\,^\scat H_Y\varphi\text{ has rank }r>0\text{ at a stationary boundary point }p_0\in\Cp.$$
We may then define a local phase function $\varphi\in\rho_Y^{-1}\rho_X^{-1}\Sm(X\times\BB^{s-r})$ of excess $e$ parametrizing the same Lagrangian.
\end{prop}

We mention that, locally, the minimal number of fiber variables $y$ that a clean phase function of excess $e$ locally parametrizing $L_\varphi$ has to possess is $$s_{\mathrm{min}}=d+e-n,$$ where $n$ is the (local) number of independent $x$ variables on $L_\varphi$. This follows from a simple dimension argument: the dimension of $L_\varphi$ is $d$, that of $C_\varphi$ is $d+e$, and the one of the projection to $x$ of $C_\varphi$ coincides with that of $L_\varphi$. Note that, by cleanness of the intersection $\Cp\cap\Bp$, near $\Lambda^\psi$ we have $s_{\mathrm{min}}>0.$

\subsection{Increasing the excess}
Given a (local) clean phase function $\varphi \in \rho_X^{-1}\rho_{\BB^s}^{-1}C^\infty(X \times \BB^s)$ with excess $e$, define $\psi:=\pr_{X\times\BB^s}^*\varphi$ on $X\times(\BB^s\times(-\eps,\eps))$, viewing $\BB^s\times(-\eps,\eps)$ as an open subset of $\BB^s\times\SSS^1$, which is a manifold with boundary whose boundary defining function may be chosen as $\pr_{\BB^s}^*\rho_{\BB^s}$. In particular we have, with the obvious identifications,
$$\scd_{\BB^s\times(-\eps,\eps)}\psi=\pr_{X\times\BB^s}^*\left(\scd_{\BB^s}\varphi\right).$$
Then $C_\psi=C_\varphi\times(-\eps,\eps)$ and hence $\dim(C_\psi^\bullet)=\dim(C_\varphi^\bullet)+1.$ Furthermore, $\lambda_\psi=\pr_{X\times\BB^s}^*\lp$ and $\Lp=\Lambda_\psi$. Summing up, $\psi$ is a local clean phase function in $s+1$ fiber variables with excess $e+1$, defined and (locally) parametrizing the same Lagrangian as $\varphi$.

As before, we may choose to keep working on balls by invoking the construction from Example \ref{ex:embdball} and replacing $\psi$ with 
$$\Psi^*\psi=\widetilde{\Psi}^*\varphi \in \rho_X^{-1}\rho_{\BB^{s+1}}^{-1}\Sm(X \times U).$$
In this way, since $\Psi$ is a diffeomorphism, $\psi$ becomes a clean phase function with excess $e+1$ defined on a relatively open subset of $X\times\BB^{s+1}$ and similarly we may raise the excess by any natural number.

\begin{ex}
The standard Fourier phase on $\RR\times\RR$, $\varphi(x,\xi)=x\cdot \xi$, cannot be seen as an $\SG$-phase on all of $\RR\times\RR^2$ by setting $\psi(x,\xi,\eta)=x\cdot\xi$. Indeed,
\begin{align}
\label{eq:SGphasextend}
\jap{x}^2|\nabla_x \varphi(x)|^2+\jap{(\xi,\eta)}|\nabla_{\xi,\eta}\varphi|^2&=(1+x^2)\xi^2+(1+\xi^2+\eta^2)x^2\\
\notag &=\jap{x}\jap{\xi}+x^2\eta^2-1
\end{align}
For $\xi=0$ and $x=0$ and $\eta\rightarrow \infty$, this vanishes but should be bounded from below by $c(1+|\eta|)^2$ if $\psi$ were an $\SG$-phase function, given \eqref{eq:SGphaseineq}.

Reviewing Example \ref{ex:embdball}, the ray $\xi=0$, $x=0$ and $\eta\neq 0$ corresponds precisely to the poles in Figure \ref{fig:fiberball} which were cut off. Indeed, \eqref{eq:SGphasextend} is bounded from below by $\jap{x}^2\jap{(\xi,\eta)}^2$ in any neighbourhood where $\frac{|\xi|}{|\eta|}>c $ and hence a local phase function in such sets.
\end{ex}

\subsection{Elimination of excess}
\label{sec:phaseexelim}
Assume now that $\varphi$ is a phase function on $X \times \BB^s$ with excess $e$ and that at some point $p_0=(\rho_{X,0},x_0,\rho_{Y,0},y_0)\in \Cp$ we have $\lp(p_0)=(\rho_{X,0},x_0,\rho_{\Xi,0},\xi_0)$. Then, by Lemma \ref{lem:lpfibration}, the preimage of $(\rho_{X,0},x_0,\rho_{\Xi,0},\xi_0)$ under $\lp$, meaning the fiber in $\Cp$ through $p_0$, is an $e$-dimensional smooth submanifold. Locally, since $\lp$ is a submersion we may, by \cite[Prop. 5.1]{Joyce}, reduce to the case of a projection, that is, we may find a splitting $y=(y^\prime,y^{\prime\prime})$ near $p_0$ such that $\lp$ does not depend on $y^{\prime\prime}$. Then,
$$\tilde{\varphi}(\rho_{X},x,\rho_{Y},y^\prime):=\varphi(\rho_{X},x,\rho_{Y},y^\prime,y^{\prime\prime}_0)$$
defines a phase function without excess (i.e., a non-degenerate phase function) that parametrizes the same Lagrangian as $\varphi$. As usual, we may again reduce to the case of a ball and hence replace $\varphi$ by a phase function on an open subset of $X\times\BB^{s-e}$.

\subsection{Equivalence of phase functions}
We will now discuss the changes of phase function under a change of coordinates and which phase functions can be considered equivalent. We first check how the stationary points of a phase function transform under changes by local diffeomorphisms. 
\begin{lem}\label{lem:CLFstarff}
Let $X_1$, $Y_1$, $X_2$, $Y_2$ be mwbs, set $B_i=X_i\times Y_i$, $i\in\{1,2\}$, and let $\varphi \in \rho_{X_2}^{-1}\rho_{Y_2}^{-1}C^\infty(B_2)$ be a (local) phase function. Assume $g:X_1\rightarrow X_2$, $h:Y_1\rightarrow Y_2$ to be diffeomorphisms, and set $F=g\times h$. Then, $F^*\varphi\in \rho_{X_1}^{-1}\rho_{Y_1}^{-1}C^\infty(B_1)$ is a (local) phase function with the same excess of $\varphi$, and we have 
\begin{align*}
C_{F^*\varphi}&=\big\{(\bx_1,\by_1)\in B_1\,|\,F(\bx_1,\by_1)
\in C_\varphi \big\},\quad
L_{F^*\varphi}=(\scOverT^*g)(L_\varphi).
\end{align*}
\end{lem}
\begin{rem}
	This means that, while the boundary defining function $\rho_{\Xi_1}$ of $\scOverT^* X_1$ does not vanish, 
	$L_{F^*\varphi}$ can then be computed as
	\[
	L_{F^*\varphi}=
	\big\{(\bx_1, \,\iota(^t (Jg) \iota^{-1}(\boldsymbol{\xi}_1))\in
	\scOverT^*X_1\mid (g(\bx_1), \boldsymbol{\xi}_1)\in L_\varphi\big\}.
	\]
	As $\rho_\Xi\rightarrow 0$, $\Lambda^\psi_{F^*\varphi}$ is obtained by taking interior limits, see also Lemma \ref{lem:horror}.
\end{rem}

\begin{proof}[Proof of Lemma \ref{lem:CLFstarff}]
	%
	%
	%
	The result for $C_\varphi$ follows immediately from 
	the first assertion in Lemma \ref{lem:scdpullback}.
	The statement for $L_\varphi$ then follows by writing
	\begin{equation}\label{eq:Lfpullback}
		\lambda_{F^*\varphi}(\bx_1,\by_1)=
		(\scOverT^*g)(\lambda_\varphi(\bx_2,\by_2))
	\end{equation}
	near a point 
	$(\bx_1,\by_1) \in (C_{F^*\varphi})^o$ such that
	$(\bx_2,\by_2)=(g(\bx_1),h(\bx_1,\by_1))$. Indeed, at
	these stationary points, 
	$\scd_X F^*\varphi=F^*(\scd_X\varphi$),
	since there $\scd_Y\varphi=0$.
	Since equality \eqref{eq:Lfpullback} holds in the interior, 
	the result at the boundary faces can be
	obtained as interior limits (see also Lemma \ref{lem:lpsct}).
	\end{proof}
\begin{rem}
The diffeomorphism $g\times h$ may be replaced by a single diffeomorphism $F:X_1\times Y_1\rightarrow X_2\times Y_2$ locally of product type near the boundary faces of $X_2\times Y_2$, i.e., a (local) diffeomorphism that is
a fibered-map at the boundary.
\end{rem}
We now define in which sense two phase functions may be considered equivalent.
\begin{defn}\label{def:phequiv}
	Let $X$, $Y_1$, $Y_2$ be mwbs, $B_i=X\times Y_i$. Let 
	$\varphi_i\in\rho_{X}^{-1}\rho_{Y_i}^{-1}C^\infty(B_i)$. 
	We say that $\varphi_1$ and $\varphi_2$ are equivalent at a pair of boundary points 
	$(\bx^0,\by_1^0)\in\mathcal{B}_1$ and $(\bx^0,\by_2^0)\in\mathcal{B}_2$ if there exists a local diffeomorphism
	$F:X\times Y_2\rightarrow X\times Y_1$ of the form $F=\mathrm{id}\times g$ with $g(\bx^0,\by_2^0)=\by_1^0$ 			such that the following two conditions are met:
	\begin{equation}
	\label{eq:eqv1}
	F^*\varphi_1-\varphi_2\text{ is smooth in a neighbourhood $U$ of }(\bx^0,\by_2^0),
	\end{equation}
    \begin{equation}
	\label{eq:eqv2}
	\rho_X\rho_{Y_2} \left(F^*\varphi_1-\varphi_2\right) \text{ restricted to } \mathcal{C}_{\varphi_2}\cap \partial U \text{ vanishes to second order.}
	\end{equation}
\end{defn}
\begin{lem}
\label{lem:equivlag}
Equivalent phase functions parametrize the same Lagrangian, meaning $\Lambda_{F^*\varphi}=\Lambda_{\varphi}$ and we have $\mathcal{C}_{F^*\varphi_1}=\mathcal{C}_{\varphi_2}$.
\end{lem}
\begin{proof}
		This follows from Lemmas \ref{lem:phaseplussmooth} and \ref{lem:CLFstarff}.
\end{proof}

We now associate to any local phase function its \emph{principal phase part}, which corresponds in the $\SG$-case to the leading homogeneous components of $\varphi$. From the fact that the principal part of Definition \ref{def:princpart} is obtained from the boundary restrictions of $\varphi$, we observe, using $F=\id\times\id$ and Lemma \ref{lem:princpart}:
\begin{lem}\label{lem:phpequiv}
    A local phase function $\varphi$ and its principal part $\varphi_p$ are equivalent.
\end{lem}
\begin{rem}
In particular, each phase function is locally equivalent at the $e$- and $\psi$-face, respectively, to a homogeneous (w.r.t. $\rho_X$ or $\rho_Y$) phase function, after a choice of collar decomposition. In general, this is not true  near the corner $\Bpe$.
\end{rem}
Since the difference in condition \eqref{eq:eqv2} is restricted to the boundary, it does not restrict the behavior of $F^*\varphi_1-\varphi_2$ into the direction transversal to the boundary, e.g. $\partial_{\rho_X}\rho_X\rho_{Y_2}(F^*\varphi_1-\varphi_2)$ at $\mathcal{C}_{\varphi_2}^e$. The following lemma states the transformation behavior of this directional derivative.
\begin{lem}
Let $X,Y_1,Y_2$ be mwbs and let $F : X\times Y_2 \to X\times Y_1$ be a $\sct$-map of the form $F = \id \times \Psi$. Set $h = \rho_{Y_2}^{-1} F^*\rho_{Y_1}$.
    Consider a clean phase function $\varphi$ on $X \times Y_1$. Write $f = \rho_X \rho_{Y_2} \varphi$.
    Then we have the following transformation laws:
    \begin{align*}
        hF^*\partial_{\rho_{Y_1}} \rho_X^{-1}f &= \partial_{\rho_{Y_2}} F^*\rho_X^{-1}f, \quad\text{ on } F^*\mathcal{C}_{\varphi}^\psi,\\
        F^*\rho_{Y_1}^{-1}\partial_{\rho_X} f &= \partial_{\rho_X} F^*\rho_{Y_1}^{-1}f, \quad\text{ on } F^*\mathcal{C}_{\varphi}^e.
    \end{align*}
\end{lem}
\begin{proof}
    On $F^*\Cpp$, we have that
    \begin{align*}
        \partial_{\rho_{Y_2}} F^*f &= hF^*\partial_{\rho_{Y_1}}f + F^*(\partial_{y_1} f)\partial_{\rho_{Y_2}}y_1= hF^*\partial_{\rho_{Y_1}}f,
    \end{align*}
    where we have used $\partial_{y_1} f = 0$ on $F^*\Cpp$.
    This proves the first equality.

    On $F^*\Cpe$, we compute
    \begin{align*}
        \partial_{\rho_X} F^*\rho_{Y_1}^{-1}f_1 &= F^*\rho_{Y_1}^{-1}\partial_{\rho_X} f_1 + F^*(\partial_{\rho_{Y_1}} \rho_{Y_1}^{-1}f_1)\, \partial_{\rho_X} F^*\rho_{Y_1} + F^*(\rho_{Y_1}^{-1}\partial_{y_1} f_1)\, \partial_{\rho_X} F^*y_1\\
        &= \rho_{Y_2}^{-1}h^{-1}F^*\partial_{\rho_X} f_1.
    \end{align*}
    Therein, we used $\partial_{y_1}f_1 = 0$ and $\partial_{Y_1} \rho_{Y_1}^{-1} f_1 = 0$ on $\mathcal{C}_{\varphi_1}$.
\end{proof}
\begin{rem}\label{rem:strictness}
    The previous lemma, combined with Lemma \ref{lem:phpequiv}, will imply that, away from the corner, any phase function can be replaced by an equivalent phase function without radial derivative (at $\Cp$) and the vanishing of this derivative at $\Cp$ is preserved under application of scattering maps.

    This corresponds to the fact that, in the classical theory, one can always choose a homogeneous phase functions. The (non-homogeneous) terms of lower order which arise in transformations can be absorbed into the amplitude.
\end{rem}
%
%
The rest of this section will be dedicated to establishing a necessary and sufficient criterion for the local equivalence of phase functions.
\begin{lem}\label{lem:arrange}
	Let $X$, $Y_1$, $Y_2$ be mwbs such that $\dim(Y_1)=\dim(Y_2)$, and set $B_i=X\times Y_i$,  $i\in\{1,2\}$. Let $\varphi_i \in \rho_{X}^{-1}\rho_{Y_i}^{-1}C^\infty(B_i)$ be phase functions which have the same excess, and assume that there exist $p^0_i=(\bx^0,\by^0_i)\in\mathcal{C}_{\varphi_i}$, $i\in\{1,2\}$, such that 
\begin{align*}
\lambda_{\varphi_1}(\bx^0,\by^0_1)&=\lambda_{\varphi_2}(\bx^0,\by^0_2),
\end{align*}
and, close to $(\bx^0,\by^0_i)$,  $i\in\{1,2\}$, both phases parametrize the same Lagrangian $\Lambda$, i.e., locally $\Lambda=\Lambda_{\varphi_i}$,
$i\in\{1,2\}$. 
Then, there exists a local diffeomorphism $F\colon B_2\to B_1$ of the
form $F=\id\times g$ with $F(\bx^0,\by^0_2)=(\bx^0,\by^0_1)$, such that
$F^*\varphi_1=\rho_X\rho_{Y_2}\widetilde{f}_1$
with $\mathcal{C}_{F^*\varphi_1}=\mathcal{C}_{\varphi_2}$,
locally. Moreover, locally near $(\bx^0,\by^0_2)$,
\begin{equation}\label{eq:prsymbeq}
	(f_2-\widetilde{f}_1)|_{\B_2}
	\text{ vanishes of second order at any point of 
	$\mathcal{C}_{\varphi_2}$.
	}
\end{equation}
\end{lem}
\begin{rem}
	Notice that \eqref{eq:prsymbeq} means that the principal
	part of $F^*\varphi_1$ and $\varphi_2$ in Lemma \ref{lem:arrange}
	coincide on $\mathcal{C}_{\varphi_2}$.	%
\end{rem}
\begin{proof}[Proof of Lemma \ref{lem:arrange}]
	Since $\lambda_{\varphi_i}$ are local fibrations from 
	$\mathcal{C}_{\varphi_i}$ to $\Lambda_{\varphi_i}$, $i\in\{1,2\}$,
	and $\Lambda_{\varphi_1}=\Lambda_{\varphi_2}=\Lambda$, 
	there is a local fibered diffeomorphism $F\colon B_2\to B_1$
	of the form $F=\id\times g$, locally
	locally near $(\bx^0,\by^0_1)=F(\bx^0,\by^0_2)$,
	such that the following diagram is
	commutative.
	\vspace{-0.5cm}
\begin{center}
\begin{tikzpicture}
  \matrix (m) [matrix of math nodes,row sep=2.5em,column sep=1.5em,minimum width=2em,minimum height=7mm,
        text depth=0.5ex,
        text height=2ex,
        inner xsep=1pt,
        outer sep=1pt]
  {
     \ & \Lambda & \ \\
     \mathcal{C}_{\varphi_2} & & \mathcal{C}_{\varphi_1} \\
};
  \path[->]
    (m-2-1) edge node [left] {$\lambda_{\varphi_2}$}(m-1-2);
  \path[->]
    (m-2-3) edge node [right] {$\lambda_{\varphi_1}$}(m-1-2);
  \path[->]
    (m-2-1) edge node [above] {$\exists F$}(m-2-3);
\end{tikzpicture}
\end{center}
\vspace{-0.5cm}
	Note that $F$ is not uniquely determined, not even on $\mathcal{C}_{\varphi_2}$ when the phases are merely clean and not non-degenerate.

	After application of $F$, we may assume that $Y_1=Y_2=:Y$, $\by_1^0=\by_2^0=:\by^0$ and, 
	locally, $\mathcal{C}_{\varphi_1}=\mathcal{C}_{\varphi_2}=:\Cp$.
	We now show that the restriction of $f_1$ and $f_2$ to a relative
	neighbourhood of $(\bx^0,\by^0)$ in $\Cp$ vanishes of 
	second order. Recall that, since $\scd_Y\varphi_1=\scd_Y\varphi_2=0$,
	for any $p=(\bx,\by)\in\Cp$ we have
	\begin{equation}\label{eq:scdmapphfncs}
		\begin{pmatrix}\rho_Y\partial_{\rho_Y}f_1-f_1
			&
		 \partial_{y_k}f_1	
		\end{pmatrix}
		=
		\begin{pmatrix}\rho_Y\partial_{\rho_Y}f_2-f_2
			&	
		 	\partial_{y_k}f_2
		 \end{pmatrix}=0
	\end{equation}
	Furthermore, since $\varphi_1$ and $\varphi_2$ parametrize the same
	Lagrangian, we also have 
	$\lambda_{\varphi_1}(p)=\lambda_{\varphi_2}(p)$, that is,
	$\iota(\scd_X\varphi_1(p))=\iota(\scd_X\varphi_2(p))$.
	We treat separately the cases $p\in\Cpe$ and $p\in\Cpp\cup\Cppe$.
	
	If $p\in\Cpe$, we then find
	\begin{equation}\label{eq:scdxequal}
		\iota((\rho_Y^{-1}\rho_X\partial_{\rho_X}f_1(p)-f_1(p),
		\rho_Y^{-1}\partial_{x_k}f_1(p)))
		=
		\iota((\rho_Y^{-1}\rho_X\partial_{\rho_X}f_2(p)-f_2(p),
		\rho_Y^{-1}\partial_{x_k}f_2(p))).
	\end{equation}
	Since $\rho_Y\not=0$ on $\Cpe$, 
	and $\iota$ is a diffeomorphism on the interior, this implies
	\[
		f_1(p)=f_2(p), \quad \partial_{x_k}f_1(p)=\partial_{x_k}f_2(p),
		\; k = 1, \dots, d-1.
	\]
	Combining this with \eqref{eq:scdmapphfncs}, this further implies
	\[
		\partial_{\rho_Y}f_1(p)=\partial_{\rho_Y}f_2(p),
		\quad
		\partial_{y_k}f_1(p)=\partial_{y_k}f_2(p), \;
		k=1, \dots, s-1.
	\]
	Since $(x,\by)$ are a complete set of variables on $\Be$, we
	can indeed
	conclude that $f_1-f_2$ vanishes of second order along $\Cpe$.
	
	If $p\in\Cpp$ or $p\in\Cppe$, \eqref{eq:scdmapphfncs} implies that
	\[
		f_1(p)=f_2(p)=0, \quad
		\partial_{y_k}f_1(p)=\partial_{y_k}f_2(p), \; k=1,\dots, s-1.
	\]  
	We have to evaluate \eqref{eq:scdxequal} as a limit $\rho_Y\to0^+$,
	using, as in Lemma \ref{lem:horror}, 
	$\iota(z)=\frac{z}{|z|}(1-\frac{1}{|z|})$. We obtain that, with
	\[
		v_1=(\rho_X\partial_{\rho_X}f_1, \partial_{x_k}f_1),
		\quad
		v_2=(\rho_X\partial_{\rho_X}f_2, \partial_{x_k}f_2),
	\]
	$\frac{v_1}{\|v_1\|}=\frac{v_2}{\|v_2\|}$, but not necessarily
	$v_1=v_2$, in which case the proof would be complete. We now modify
	$F$ in order to achieve $v_1=v_2$. Notice that, since $\varphi_1$ and $\varphi_2$ are phase functions,
	we have $v_1\not=0$ at $\Cp$. We can therefore scale $\varphi_1$
	by means of the local diffeomorphism (near $\Cp$)
	\[
		\widetilde{F}\colon(\rho_Y, y)\to
		(\rho_Y \, r(\rho_X,x,\rho_Y,y), y),
	\]
	where $r(\rho_X,x,\rho_Y,y)=\frac{\|v_2\|}{\|v_1\|}$. Notice that,
	by our previous computations, $r|_{\Cpe\cup\Cppe}=1$, 
	and $\widetilde{F}$
	is the identity for $\rho_Y=0$. Therefore, by Lemma \ref{lem:CLFstarff},
	\[
		\mathcal{C}_{\widetilde{F}^*\varphi_1}=\mathcal{C}_{\varphi_1},
		\;\text{ and }\;
		\Lambda_{\widetilde{F}^*\varphi_1}=\Lambda_{\varphi_1}.
	\]
	By definition, for $\widetilde{F}^*\varphi_1$ we have
	\[
		\widetilde{f}_1:=\rho_X\rho_Y\widetilde{F}^*\varphi_1=
		\frac{\|v_2\|}{\|v_1\|}(F^*f_1).
	\]
	Therefore,
	\[
	(\rho_X\partial_{\rho_X}\widetilde{f}_1, \partial_{x_k}\widetilde{f}_1)
	=
	\frac{\|v_2\|}{\|v_1\|}\cdot
	(\rho_XF^*(\partial_{\rho_X}{f}_1), F^*(\partial_{x_k}\widetilde{f}_1))
	=:\widetilde{v}_1,
	\]
	since the derivatives acting on $r$ produce a $\rho_Y$ factor, and 
	then vanish along $\Cpp$. Hence, $\widetilde{v}_1=v_2$, which completes
	the proof.
\end{proof}
\begin{rem}
	The additional computations in the proof of the previous lemma	 near the face $\Cpp$ correspond to the fact that, classically,	$x\cdot\theta$ and $x\cdot(2\theta)$ both parametrize
	\[
		\Lambda=\left\{(0,\xi)\mid \xi\in\RRd\setminus\{0\}\right\}.
	\]
	In fact, we observe from the same proof that we may choose
	the norm of $(\rho_X\partial_{\rho_X}f_1,\partial_{x_k}f_1)$
	at any point of $\Lpp$ without changing $\Lp$.  
\end{rem}

\begin{thm}[Equivalence of phase functions]
\label{thm:equivphase}
Let $X$, $Y_1$, $Y_2$ be mwbs such that $\dim(Y_1)=\dim(Y_2)$, and set $B_i=X\times Y_i$,  $i\in\{1,2\}$. Let $\varphi_i \in \rho_{X}^{-1}\rho_{Y_i}^{-1}C^\infty(B_i)$, $i\in\{1,2\}$,
be phase functions which have the same excess, assume that there exist $(\bx^0,\by^0_i)\in\mathcal{C}_{\varphi_i}$, $i\in\{1,2\}$, such that 
\begin{align*}
\lambda_{\varphi_1}(\bx^0,\by^0_1)&=\lambda_{\varphi_2}(\bx^0,\by^0_2),
\end{align*}
and, close to $(\bx^0,\by^0_i)$,  $i\in\{1,2\}$, both phase functions parametrize the same Lagrangian $\Lambda$, i.e., locally $\Lambda=\Lambda_{\varphi_i}$,
$i\in\{1,2\}$. Then, it is necessary and sufficient 
for $\varphi_1$ and $\varphi_2$ to be equivalent at $(\bx^0,\by^0_1)$ 
and $(\bx^0,\by^0_2)$ that there it holds that
\begin{equation}\label{eq:sgncond}
\mathrm{sgn}\left(\,\rho_{Y_1}^{-1}\rho_X^{-1}\,\scHess_{Y_1}\varphi_1\right)
=\mathrm{sgn}\left(\rho_{Y_2}^{-1}\rho_X^{-1}\, \scHess_{Y_2}\varphi_2\right).
\end{equation}
\end{thm}
\begin{rem}
\label{rem:scHess}
Before we go into the details of the proof, we recall the expression for 
the differential in condition \eqref{eq:sgncond} in coordinates. By \eqref{eq:scatHessfactor} we have, writing 
$\varphi=\rho_X^{-1}\rho_Y^{-1}f$,
$$
\rho_{Y}^{-1}\rho_X^{-1} \, \scHess_Y\varphi=
\begin{pmatrix} \rho_Y & 0 \\
0 & \mathbbm{1} \end{pmatrix} 
\begin{pmatrix}
\partial_{\rho_Y}^2f & \partial_{y_j}\partial_{\rho_Y} f \\
\partial_{\rho_Y}\partial_{y_k} f & \partial_{y_j}\partial_{y_k} f
\end{pmatrix}
\begin{pmatrix} \rho_Y & 0 \\
0 & \mathbbm{1} \end{pmatrix}.
$$
Hence, for $\rho_Y\neq 0$, the signature of this matrix is that of $H_Y f$, whereas for $\rho_Y=0$ it is that of the Hessian of $f$ \emph{restricted to $\rho_Y=0$}, that is,
only with respect to the boundary variables, 
$\left(\partial_{y_j}\partial_{y_k} f(0,y)\right)_{jk}$.
\end{rem}
\begin{proof}[Proof of Theorem \ref{thm:equivphase}]
	%
	%

	We first prove that condition \eqref{eq:sgncond} is necessary. 
		In view of Lemma \ref{lem:equivlag}, we only need to compare $\scHess_{Y_1}\varphi_1$ and $\scHess_{Y_2}\varphi_2$ by writing 
		\begin{equation}
		\scHess_{Y_2}\varphi_2=\scHess_{Y_2}F^*\varphi_1+\scHess_{Y_2}(\varphi_2-F^*\varphi_1).
		\end{equation}
		We write $r=(\varphi_2-F^*\varphi_1)$, which, by assumption,
		satisfies $r\in\Sm(X\times Y_2)$.
		Therefore, $\rho_{Y_2}^{-1}\rho_X^{-1}\,\scHess_{Y_2}r$ 
		vanishes at the boundary.		
		Indeed, in local coordinates we have
		$$
\rho_Y^{-1}\rho_{X}^{-1}\,\scHess_{Y_2} r=\begin{pmatrix}
\rho_Y\rho_X \partial_{\rho_Y}\rho_Y^2\partial_{\rho_Y}r & \rho_Y^2\rho_X\partial_{y_j}\partial_{\rho_Y}r \\
\rho_Y\rho_X\partial_{\rho_Y}\rho_Y\partial_{y_k}r & \rho_Y\rho_X\partial_{y_j}\partial_{y_k}r
\end{pmatrix}.
$$
		Thus, we have, at the boundary,
\begin{equation}\label{eq:sgneq}
\mathrm{sgn}\left(\,\rho_{Y_2}^{-1}\rho_X^{-1} \, \scHess_{Y_2}F^*\varphi_1\right)
=\mathrm{sgn}\left(\rho_{Y_2}^{-1}\rho_X^{-1} \,\scHess_{Y_2}\varphi_2\right).
\end{equation}
		By computing these differentials in coordinates at corresponding stationary points, using \eqref{eq:scatHessfactor}, this implies \eqref{eq:sgncond}.

		For the sufficiency of \eqref{eq:sgncond}, we assume familiarity of the reader with the equivalence of phase function theorem in the usual homogeneous setting, see \cite[Prop. 4.1.3]{Treves}, \cite[Prop. 4.1.3]{Treves} and sketch briefly that the argument goes through with little modification.

		By Lemma \ref{lem:arrange} we may assume $Y_1=Y_2$. Note that equivalence is achieved for $\varphi_i=\rho_X\rho_Y f_i$ if the $f_i$ agree on the boundary. The condition on $\scHess_Y\varphi_i$ 
		means precisely that the signatures of the Hessians of the $f_i$ in the tangential derivatives agree in the interior and the signatures of the Hessians of the restriction of the $f_i$ to $\rho_Y=0$ as well, see Remark \ref{rem:scHess}. As such, we may use the same techniques as in the classical situation to construct a diffeomorphism on the boundary which transforms the restriction of $f_1$ into that of $f_2$, cf. also \cite{CoSc2}. This diffeomorphism is then extended by means of Proposition \ref{prop:cornerdiffeo} into the interior. For sake of brevity, we omit the details here.
		 %
		 %
		 
 
\end{proof}
\begin{rem}
Note that near $(\bx^0,\by^0)\in\Cpp$, we can also invoke the classical equivalence theorem directly. We need to find a transformation 
		$$F:(\bx,0,y)\mapsto (\bx,0,\tilde{y}(\bx,y))$$
		such that $F^*\varphi_1=\varphi_2$. For $\lambda>0$ we set $\phi_j(\bx,\lambda,y)=\lambda f_j(\bx,0,y)$, $j\in\{1,2\}$. Then $\phi_j$ are  equivalent \emph{phase functions in the usual homogeneous sense} on $X\times (\RR_+\times Y)$. Indeed, evaluating $\dd\phi_j$ and $\scd\varphi_j$ in coordinates, we see that $\dd\phi_j\neq 0$ and $\phi_j$ is manifestly homogeneous. Furthermore, the signatures of $H_Y\phi_j$ are the same as those of $\scHess_Y\varphi_j$. Since the $f_j$ are equal up to second order, the $\phi_j$ are equivalent in the usual sense and there exists a $\lambda$-homogeneous $G: (\bx,\lambda,y)\mapsto (\bx,\lambda,\tilde{y}(\lambda,\bx,y))$ which is homogeneous such that $G^*\phi_1=\phi_2$. Setting $F=G|_{\lambda=1}$ and possibly applying a scaling, as in the proof of Lemma \ref{lem:arrange}, concludes the proof for $(\bx^0,\by^0)\in\Cpp$.
\end{rem}

\section{Lagrangian distributions}
\label{sec:Lagdist}
In this section, we will address the class of Lagrangian distributions on scattering manifolds.
First, we introduce oscillatory integrals associated with a phase function and show that they are well-defined in the usual sense.
Then, we define Lagrangian distributions as a locally finite sum of oscillatory integrals, where the phase function parametrizes a
Lagrangian submanifold.
Using the results from the previous section, we are able to reduce the number of fiber-variables to a minimum
and see that the order of the Lagrangian distribution is well-defined independently of the dimension of the fiber.

\subsection{Oscillatory integrals associated with a phase function}
\begin{defn}
    Let $Y$ be a mwb. For the remainder of this section, $m_\eps$ with $\eps\in(0,1]$,
    denotes a family of functions $m_\eps \in\dSmz(Y)$ such that for all $k\in\NN$, $\alpha \in \NNz_0^{d-1}$ and $\epsilon>0$,
\begin{equation}
\label{eq:approxone}
\left|(\rho_Y^2\partial_{\rho_Y})^k (\rho_Y\partial_y)^\alpha m_\eps(\by)\right| \leq C_{k,\alpha}\,\rho_Y^{k + |\alpha|},
\end{equation}
such that, for all $\by \in Y^o$, we have $m_\eps(\by) \to 1$ as $\eps \to 0.$
\end{defn}

\begin{rem}
We make the observation that \eqref{eq:approxone} does not depend on the choice of bdf and is preserved under pullbacks by $\sct$-maps.
It is possible to find such a family on any manifold with boundary. In fact, any choice of tubular neighbourhood $U$ of $\partial Y$ such that $U\cong [0,\delta)\times \partial Y$ with coordinates $(\rho_Y,y)$ introduces a dilation in the first variable.
Take a function $\chi \in \Smc[0,\infty)$ such that $\chi(x) = 1$ on $[0,\delta]$.
Then set $m_\eps=1$ on $Y\setminus U$ and
\[m_\eps(\rho_Y,y)=\begin{cases}
    \chi(\eps\rho_Y^{-1}) & \text{if }\eps\rho_Y^{-1}> \delta/2,\\
    1 & \text{otherwise}.
\end{cases}\]
\end{rem}

\begin{defn}
Consider $X$, $Y$ mwbs, $U\subset X\times Y$ an open subset, $\varphi\in\rho_X^{-1}\rho_Y^{-1}\Sm(U)$ a phase function
and $a\in\rho_X^{-m_e}\rho_Y^{-m_\psi}\Sm(X\times Y, \scOmega^{1/2}(X) \times \scOmega^1(Y))$ an amplitude supported in $U$.
Then $I_\varphi(a)\in\dSmdz(X,\scOmega^{1/2}(X))$ is defined as the distributional $1/2$-density acting on $f\in\dSmz(X,\scOmega^{1/2}(X))$ by
\begin{equation}
\label{eq:oscidef}
\langle I_\varphi(a),f\rangle:= \lim_{\eps\searrow 0} \iint_{X\times Y} \left(e^{i\varphi} a \cdot (f \otimes m_\eps)\right).
\end{equation}
\end{defn}

\begin{rem}
If $X$ and $Y$ are equipped with a scattering metric, we have a canonical identification of functions and $1$-densities provided by the volume form.
Therefore, we can freely choose whether to view functions and distributions as matching (distributional) $1$-, $0$- or $\frac{1}{2}$-densities.
\end{rem}
\begin{rem}
    When $X=\BBd$ and $Y=\BBs$, these oscillatory integrals correspond, under (inverse) radial compactification, to the tempered oscillatory integrals analyzed in \cite{CoSc2,Schulz}.
\end{rem}

\begin{lem}\label{lem:osciwelldef}
    The expression \eqref{eq:oscidef} yields a well-defined tempered distribution (density) on $X$.
    In particular, it is independent of the choice of $m_\eps$.
\end{lem}
\begin{proof}
    Assume, without loss of generality, that we have a fixed scattering metric and we can identify scattering densities and functions.
    Let $U \subset X \times Y=: B$ be an open neighborhood of the boundary $\B^\psi$ such that $\scd \varphi \not = 0$ on $U$.
    
    On $X\times Y \setminus U$, the dominated convergence theorem implies that \eqref{eq:oscidef} is well-defined. The 	integrand $u_\eps = e^{i\varphi} a (f \otimes m_\eps)$ converges pointwise and is dominated by $|a\cdot f|$, which is 	bounded for $\rho_Y>c$.

    On $U$, as in the classical theory, we can define a first order scattering differential $L \in \mathrm{Diff}^1_{\sct}(U)$ which has the property that $Le^{i\varphi} = e^{i\varphi}$. By Proposition 1 from \cite{Melrose1}, we see that $L^t \in \mathrm{Diff}^1_{\sct}(U)$.
    Using repeated integration by parts and \eqref{eq:approxone}, we are able to increase the order in $\rho_X$ and $\rho_Y$ to arbitrary powers, and an application of the dominated convergence theorem then finishes the proof.
\end{proof}
After an arbitrary choice of scattering metrics, we may locally identify $(X,g_X)$ and $(Y,g_Y)$ with subsets of $\BB^d$ and $\BB^s$, respectively. Then, using some explicit local isomorphism $\Psi=\Psi_X\times\Psi_Y$, we can identify densities with functions using the induced measures $\mu_X$ and $\mu_Y$. After use of a partition of unity, we may locally express \eqref{eq:oscidef} as
\begin{align}\label{eq:oscilocdef}
    \langle I_\varphi(a),f\rangle:= \lim_{\eps\searrow 0} \iint_{\BBd\times \BBs} \Psi^*\left(e^{i\varphi(\rho_X,x,\rho_Y,y)} a(\rho_X,x,\rho_Y,y) m_\eps(\rho_Y,y) f(\rho_X,x)\right)\\
    \label{eq:oscilocdefbis}
    = \lim_{\eps\searrow 0} \iint_{\BBd\times \BBs} e^{i\Psi^*\varphi(\rho_X,x,\rho_Y,y)} \wt m_\eps(\rho_Y,y) \wt a(\rho_X,x,\rho_Y,y) \wt f(\rho_X,x) \dd \mu_{\BBd} \dd \mu_{\BBs}
\end{align}
where $\wt f=\Psi^*f |\dd \mu_\BBd|^{-1/2}$ and $\wt a\in \rho_\BBd^{-m_e}\rho_\BBs^{-m_\psi}\Sm(\BBd\times \BBs)$ satisfies $\wt a \wt f \dd \mu_{\BBd} \dd \mu_{\BBs}=a f$.
%
%
%
%
Summing up, we may always transform to locally work on $\BBd\times\BBs$ and in local coordinates we work with usual oscillatory integrals.

Since \eqref{eq:oscidef} does not depend on the choice of $m_\eps$, as it is usual we drop it from the notation and write, \emph{in the sense of oscillatory integrals},
\begin{equation}\label{eq:osciformdef}
    I_\varphi(a):= \int_{Y} e^{i\varphi}a.
\end{equation}

\subsubsection{Singularities of oscillatory integrals}

Recall that there is a notion of wavefront-set adapted to the pseudo-differential scattering calculus, called the scattering wavefront-set, cf. \cite{Cordes,Melrose1,CoMa}. 

\begin{defn}
Let $u\in\dSmdz(X,\scOmega^{1/2})$.
A point $z_0 \in \Wt=\partial\big(\scOverT^*X\big)$ is not in the scattering wavefront-set, and we write $z_0 \notin \WFsc(u)$, 
if there exists a scattering pseudo-differential operator $A$ whose symbol is elliptic at $z_0$ such that $Au\in\dSmz(X,\scOmega^{1/2}).$
\end{defn}

\begin{prop}
\label{prop:WFosci}
For the oscillatory integral in \eqref{eq:oscidef}, we have
$$\WFsc(I_\varphi(a))\subseteq \Lambda_\varphi.$$
Furthermore, if $z\in \Lambda_\varphi$ and $a$ is rapidly decaying near $\lambda_\varphi^{-1}(z)$, then $z\notin \WFsc(I_\varphi(a))$.
\end{prop}

\begin{rem}
\label{rem:css}
The ($\sct$-)singular support of $u$ is defined as follows:
a point $p_0\in X$ is contained in $\mathrm{singsupp}_\scat(u)$ if and only if for every $f\in \Sm(X)$ with $f(p_0)=1$ we have $fu\notin \dSmz(X)$.
Similar to the classical wavefront-set and singular support, we have that $\pr_1(\WFsc(u))=\mathrm{singsupp}_\scat(u)$.
Thus, in particular, if $a$ is rapidly decaying near $\Cp$, then $I_\varphi(a)\in \dSmz(X)$.
\end{rem}

We refer the reader to \cite{CoSc,Schulz} for the details of this analysis of the wavefront-sets. The proof is carried out as in the classical setting: first, a characterization of $\WFsc$ in terms of cut-offs and the Fourier transform is achieved, and then one estimates $\Fu I_\varphi(a)$ in coordinates.

Proposition \ref{prop:WFosci} gives another insight why we consider 
$\Lambda_\varphi$ as the true object of interest associated with a phase function, not $L_\varphi$. In fact, considering \eqref{eq:oscidef} once more, we see that we may modify phase function and amplitude in the integral by any real valued function $\psi\in\Sm(X\times Y)$, writing

$$e^{i\varphi} a=e^{i(\varphi+\psi)} \left(e^{-i\psi}a\right).$$

Then $e^{-i\psi}a\in \rho_X^{-m_e}\rho_Y^{-m_\psi}\Sm(X\times Y)$, and hence it is still an amplitude, and $\varphi+\psi$ is a new local
phase function. Now, while in general $L_\varphi\neq L_{\varphi+\psi}$, we have $\Lambda_\varphi= \Lambda_{\varphi+\psi}$, by Lemma \ref{lem:phaseplussmooth}.
This underlines that only $\Lambda_\varphi$ and not $L_\varphi$ can be associated with $I_\varphi(a)$ in an intrinsic way.
Nevertheless, it is often convenient to have $L_\varphi$ available during the proofs.

\subsection{Definition of Lagrangian distributions}
The class of oscillatory integrals associated with a Lagrangian is -- as in the classical theory -- not a good distribution space, since in general it is not possible to find a single global phase function to parametrize $\Lambda$. Instead, we introduce the following class of Lagrangian distributions. Note that, by our previous findings, we may always reduce an oscillatory integral on $X\times Y$ into a finite sum of oscillatory integrals over $X\times\BB^s$ for $s=\dim(Y)$.

\begin{defn}[$\sct$-Lagrangian distributions]\label{def:Lagdist}
Let $X$ be a mwb, $\Lambda\subset \partial\scOverT^*X$ a $\sct$-Lagrangian. Then, $I^{m_e,m_\psi}(X,\Lambda)$, $(m_e,m_\psi)\in \RR^2$,
denotes the space of distributions that can be written as a finite sum of (local) oscillatory integrals as in \eqref{eq:osciformdef}, whose phase functions are clean and locally parametrize $\Lambda$, plus an element of $\dSmz(X)$. More precisely, $u\in I^{m_e,m_\psi}(X,\Lambda)$ if, modulo a remainder in $\dSmz(X)$,
\begin{equation}
\label{eq:Lagdistdef}
u=\sum_{j=1}^N \int_{Y_j} e^{i\varphi_j}a_j,
\end{equation}
where for $j=1,\dots,N$:
\begin{enumerate}
\item[1.)] $Y_j$ is a mwb of dimension $s_j$;
\item[2.)] $\varphi_j\in \rho_{Y_j}^{-1}\rho_X^{-1}\Sm(X\times  Y_j)$ is a local clean phase function with excess $e_j$, defined on an open neighbourhood of the support of $a_j$,
which locally parametrizes $\Lambda$;
\item[3.)] $a_j\in \rho_{Y_j}^{-m_{\psi,j}}\rho_X^{-m_{e,j}}\Sm\big(X\times Y_j, \scOmega^{1/2}(X) \times \scOmega^1(Y)\big)$ with 
\[
    (m_{\psi,j},m_{e,j})=\left(m_\psi+\frac{d}{4}-\frac{s_j}{2}-\frac{e_j}{2},m_e-\frac{d}{4}+\frac{s_j}{2}-\frac{e_j}{2}\right).
\]
\end{enumerate}
We also set
\begin{align*}
	 I^{-\infty,-\infty}(X,\Lambda) &= \bigcap_{(m_\psi,m_e)\in\RR^2} I^{m_\psi,m_e}(X,\Lambda),
	 \\
	  I(X,\Lambda) =I^{+\infty,+\infty}(X,\Lambda)&= \bigcup_{(m_\psi,m_e)\in\RR^2} I^{m_\psi,m_e}(X,\Lambda).
\end{align*}
\end{defn}
\begin{rem}
	The reason for the choice of the $a_j$ in the scattering amplitude densities spaces of order $(m_{e,j}, m_{\psi,j})$
	will be explained in Section \ref{ssec:order}.
\end{rem}
The next result follows from Proposition \ref{prop:WFosci}.
\begin{prop}
Let $\Lambda\subset \partial\, \scOverT^*X$ be a $\sct$-Lagrangian, and $u\in I(X,\Lambda)$. Then $\WFsc(u)\subseteq \Lambda.$
\end{prop}
As in the classical case, the class of Lagrangian distributions contains the globally regular functions (cf. Treves~\cite[Chapter VIII.3.2]{Treves}):
\begin{lem}\label{lem:smoothpart}
Let $\Lambda\subset \partial\, \scOverT^*X$ be a $\sct$-Lagrangian. Then
\begin{equation}
    \dSmz(X,\scOmega^{1/2}(X)) = I^{-\infty,-\infty}(X,\Lambda).
\end{equation}
\end{lem}
\begin{proof}
We first prove the inclusion ``$\supseteq$''. Choose a finite covering of $\scOverT^*X$ with open sets $\{X_j\}_{j=1}^N$ such that there exists a clean phase function $\varphi_j$ on each $X_j$
parametrizing $\Lambda \cap \scOverT^*X_j$, $j=1,\dots,N$. Let $\{g_j\}_{j=1}^N$ be a smooth partition of unity subordinate to such covering. We view $X_j$ as a subset
of $X \times \BB^d$, $j=1,\dots,N$.

Let $\chi \in \dSmz(\BB^d, \scOmega^1(\BB^d))$ such that $\int \chi = 1$. For any $f \in \dSmz(X,\scOmega^{1/2}(X))$ we set 
\begin{align*}
    a_j = e^{-i\varphi_j}g_j \cdot (f \otimes \chi),\quad
    f_j = \int_{\BB^d} e^{i\varphi_j} a_j, \quad j=1,\dots,N.
\end{align*}
We see that
\[a_j \in \dSmz(X\times \BB^d, \scOmega^{1/2}(X) \times \scOmega^1(\BB^d)), \quad j=1,\dots,N,\]
and, summing up,
\begin{align*}
\sum_{j=1}^N f_j(x) &= \int_{\BB^d} \left(\sum_{j=1}^N g_j(x,y)\right) \cdot (f(x) \otimes \chi(y))= f(x).
\end{align*}
The inclusion ``$\subseteq$'' is achieved by differentiation under the integral sign.
\end{proof}
\subsection{Examples}
We have the following examples of (scattering) Lagrangian distributions.
\begin{enumerate}
\item Standard Lagrangian distributions of compact support, \cite{HormanderFIO,Hormander4}, in particular Lagrangian distributions on compact manifolds $X$ without boundary, are scattering Lagrangian distributions, using the identification
\begin{align*}
\textrm{Fiber-conic sets in }T^*X\setminus\{0\}\longleftrightarrow \textrm{Sets in }S^*X
\stackrel{\textrm{rescaling}}{\longleftrightarrow} \textrm{Sets in }\Wp.
\end{align*}
\item Legendrian distributions of \cite{MZ}. Here, the distributions are smooth functions whose singularities at the boundary are of Legendrian type, meaning in $\We$.
\item Conormal distributions, meaning the distributions where the Lagrangian, see Section \ref{sec:conorm}, is $\partial\big(\scOverN^*X'\big)$ for a ($k$-dimensional) $p$-submanifold $X'\subset Y$. These distributions correspond, under compactification of base and fiber, to the oscillatory integrals given in local (pre-compactified) Euclidean coordinates by
$$u(x',x'')=\int e^{ix'\xi} a(x,\xi)\,\dd \xi, \qquad a(x,\xi)\in\SG_\cl^{m_e,m_\psi}(\RRd\times\RR^{d-k}).$$
A prototypical example is given by (derivatives of) $\delta_{0}(x')\otimes 1$. These arise as (simple or multiple) layers when solving partial differential equations along infinite boundaries or Cauchy surfaces.
\item Examples of scattering Lagrangian distributions which are of none of the previous types arise in the parametrix construction to hyperbolic equations on unbounded spaces, for example the two-point function for the Klein-Gordon equation. For a discussion of this example consider \cite{CoSc2}.
\end{enumerate}

\begin{rem}
Note that, at this stage, the kernels of pseudo-differential operators on $X\times X$ are \emph{not} scattering conormal distributions associated with the diagonal $\Delta\subset X\times X$ when $X$ is a manifold with boundary. In fact, in this case $X\times X$ is a manifold with corners. Furthermore $\Delta\subset X\times X$ does not hit the corner $\partial X\times\partial X$ in a clean way, that is,
$\Delta\subset X\times X$ is not a $p$-submanifold. Similarly, the phase function associated to the $\SG$-phase $(x-y)\xi\in\SG^{1,1}_\cl(\RR^{2d}\times\RRd)$ is not clean.

However, the formulation of the theory developed in this paper admits a natural extension to manifolds with corners. The geometric obstruction of $\Delta\subset X\times X$ -- or more generally the graphs of (scattering) canonical transformations -- not being a $p$-submanifold can be overcome by lifting the analysis to a blow-up space, see \cite{MZ,Melroseb}. We postpone this theory of compositions of canonical relations and calculus of scattering Fourier integral operators to a subsequent paper.
\end{rem}

\subsection{Transformations of oscillatory integrals}
In Section \ref{sec:exchphase} we have seen several procedures that allow
to switch from one phase function to others that parametrize the same Lagrangian. We will now exploit these to transform oscillatory integrals into ``standard form''. In the sequel, we will always assume, by 
a partition of unity, that the support of the amplitude is suitably small.

\subsubsection{Transformation behavior and equivalent phase functions}
\label{sec:moves}
Now we reconsider \eqref{eq:oscilocdef}, to express the transformation 
behavior of the oscillatory integrals under fiber-preserving
diffeomorphisms. With the chosen notation and a local
phase function $\varphi_1$, we have
\begin{equation}\label{eq:oscintsimpl}
	I_{\varphi_1}(a)= \int_{Y_1} e^{i\varphi_1}a=\int_{Y_2} e^{iF^*\varphi_1}F^*a=I_{F^*\varphi_1}(F^*a)
\end{equation}
for any diffeomorphism $F:X\times Y_2\rightarrow X\times Y_1$ of the form $F=\mathrm{id}\times g$.
Assume that $\varphi_2$ is equivalent to $\varphi_1$ by $F$,
see Definition \ref{def:phequiv}. After the transformation, 
we rewrite \eqref{eq:oscintsimpl} as 
\begin{equation}
	\int_{Y_2} e^{i\varphi_2}e^{i(F^*\varphi_1-\varphi_2)}F^*a.
\end{equation}
Now, since $F^*\varphi_1-\varphi_2$ is smooth up to the boundary,
the same holds for $e^{i(F^*\varphi_1-\varphi_2)}$ and this factor can be seen as part of the amplitude.
Therefore, we may write
\begin{equation}\label{eq:oscintequiv}
	I_{\varphi_1}(a)=I_{\varphi_2}
	\big((F^*a)\,\exp(i(F^*\varphi_1-\varphi_2))\big).
\end{equation}
In particular, we can express $I_\varphi(a)$, 
near any boundary point of the domain of definition,
using the principal part of $\varphi$ introduced in Definition \ref{def:princpart}, namely
\begin{equation}\label{eq:oscintstd}
	I_{\varphi_p}(\widetilde{a}), \text{ with } 
	\widetilde{a}=a\,\exp\big(i(\varphi-\varphi_p)\big).
\end{equation}
By Lemma \ref{lem:phpequiv}, $\varphi - \varphi_p \in \Sm$ and thus $\widetilde{a} \in \rho_X^{-m_e}\rho_Y^{-m_\psi} \Sm(B)$.
In the following constructions, we always assume that $\varphi$ is replaced by its principal part, cf. Remark \ref{rem:strictness}.

\subsubsection{Reduction of the fiber}\label{ssbs:redfbr}

We will now analyze the change of boundary behavior under a reduction of fiber variables near $p_0\in\supp(a)\cap\Cp$. Hence, we assume that
$$\rho_Y^{-1}\rho_X^{-1}\,^\scat H_Y\varphi\text{ has rank }r>0\text{ at }p_0\in\Cp.$$
We assume, as explained above, that the oscillatory integral
is in the form \eqref{eq:oscintstd}, namely, $\varphi$ is replaced
by its principal phase part. We observe that, at the boundary point $p_0$,
$$\rk(\rho_Y^{-1}\rho_X^{-1}\,^\scat H_Y\varphi)=\rk(\rho_Y^{-1}\rho_X^{-1}\,^\scat H_Y\sigma(\varphi_p)).$$
By Proposition \ref{prop:fiberred}, we can define a local phase function 
$\varphi_{\red}$ parametrizing the same Lagrangian as 
$\varphi$. In particular, after a change of coordinates by a scattering map,
we can assume $(\bx,\by)\in X\times\BB^{s-r}\times(-\eps,\eps)^{r}$,
and $\varphi_{\red}$ is given by
\[
	\varphi_{\red}(\bx,\rho_Y,y^\prime)=\varphi(\bx,\rho_Y,y^\prime,0),
\]
where $\rho_Y=\rho_{\BB^{s-r}}$ is the boundary defining function on $\BB^{s-r}$ and on $\BB^{s-r}\times(-\eps,\eps)^r$.
We introduce
\begin{equation}\label{eq:wtp}
	\widetilde{\varphi}(\bx,\by)=\varphi_{\red}(\bx,\rho_Y,y^\prime)+
	\frac{1}{2}\rho_X^{-1}\rho_Y^{-1}Q(y^{\prime\prime}),
\end{equation}
where $Q$ is a non-degenerate quadratic form with the same signature as
$\partial_{y^{\prime\prime}}\partial_{y^{\prime\prime}}f$ at $p_0$.
Then, by Theorem \ref{thm:equivphase},
$\varphi$ is equivalent to $\widetilde{\varphi}$ by a local
diffeomorphism $F=\mathrm{id}\times g$. 
Note that $\varphi_\red$ is equal to its principal part, because we assumed that $\varphi$ is replaced by $\varphi_p$.
%
%

We may assume that $a$ is supported in an arbitrarily small neighbourhood of the stationary points of $\varphi$. Indeed, we may achieve this for a general amplitude $a$ by applying a cut-off in $y^{\prime\prime}$ and writing $a=\phi a + (1-\phi) a$. The oscillatory integral with amplitude $(1-\phi)a$ produces a term in $\dSmz(X, \Omega^{1/2}(X))$, by Remark \ref{rem:css}.

Therefore, choosing the support of $a$ small enough, we may perform the change of variables by the local diffeomorphism $F$ as in \eqref{eq:oscintequiv}. We write, motivated by Lemma \ref{lem:intdensity} and Example \ref{ex:embdball},
$$a_\red(\bx,\wby)\,\frac{|\dd \wby''|}{\rho_{\wtY}^{r} \cdot [h(\bx,\wby)]^r}=(F^*a)(\bx,\wby),$$
which is assumed supported in some compact subset of $(-\epsilon,\epsilon)^r$. Then $I_\varphi(a)$ is transformed into
$I_{\varphi_{\red}}(b)$
where
\begin{equation}
b(\bx,\rho_{Y},y^\prime)=\rho_{Y}^{-r}\int_{(-\eps,\eps)^r} e^{\frac{i}{2}\rho_X^{-1}\rho_{Y}^{-1}Q(y^{\prime\prime})} 
\Big(e^{i(F^*\varphi(\bx,\by)-\widetilde{\varphi}(\bx,\by))}\,a_\red(\bx,\by)\Big)\dd y^{\prime\prime}. \label{eq:fiberredsymb}
\end{equation}
%
%
We claim that
$b(\bx,\rho_{Y},y^\prime)$
is again a (density valued) amplitude. First, it is clear that $b$ decays rapidly at $(\bx,\rho_{Y},y^\prime)$ if $a$ decays rapidly at $(\bx,\rho_{Y},y^\prime,0)$. In particular, $b$ is smooth away from $\B$. 

We now we apply the stationary phase lemma \cite[Lem. 7.7.3]{Hormander1} to \eqref{eq:fiberredsymb}, which yields the asymptotic equivalence, as $\rho_Y\rho_X\rightarrow 0$, 
\begin{multline}
\label{eq:princred1}
b(\bx,\rho_Y,y^\prime)= \rho_X^{r/2}\rho_Y^{-r/2}|\det Q|^{-1/2} e^{\frac{i}{4}\pi \mathrm{sgn}(Q)} e^{i(F^*\varphi(\bx,\rho_Y,y^\prime,0)-\widetilde{\varphi}(\bx,\rho_Y,y^\prime,0))} a_\red(\bx,\rho_Y,y^\prime,0)  \\
+\cO\big(\rho_Y^{-m_\psi-\frac{r}{2}+1}\rho_X^{-m_e+\frac{r}{2}+1}\big).
\end{multline}
Similar asymptotics hold for all derivatives of $b$.
We may hence view $b$ as a (density valued) amplitude of the order 
\begin{equation}\label{eq:order-fiber}
    (m_e',m_\psi')=\left(m_e-\frac{r}{2},m_\psi+\frac{r}{2}\right).
\end{equation}
By Remark \ref{rem:strictness} we see that, away from the corner, $F^*\varphi-\widetilde{\varphi}$ vanishes at $\Cp$. Therefore, the principal part of $b$ does not depend on $\varphi$.
Hence, by comparision of principal parts, cf. Lemma \ref{lem:princpart}, \eqref{eq:princred1} reduces to 
\begin{equation}
\label{eq:princred}
b(\bx,\rho_Y,y^\prime)\sim \rho_X^{r/2}\rho_Y^{-r/2}|\det Q|^{-1/2} e^{\frac{i}{4}\pi \mathrm{sgn}(Q)} a_\red(\bx,\rho_Y,y^\prime,0)
\end{equation}
modulo terms of lower order.
\subsubsection{Elimination of excess}\label{subss:elexcess}
Assume now that $\varphi$ is a clean phase function of excess $e>0$. Near some point in $\Cp$, as described in Section \ref{sec:phaseexelim}, we may make the following geometric assumptions after application of some diffeomorphism $F$: We assume that $Y=\BB^{s-e}\times(-\epsilon,\epsilon)^e$ and that the fibers of $\Cp\rightarrow \Lp$ are given by constant $(\bx,\rho_Y,y')$ and arbitrary $y''$.
We proceed as in \cite{Treves} and define 
\begin{equation}\label{eq:wtredex}
	\wt{\varphi}(\rho_{X},x,\rho_{Y},y^\prime):=\varphi(\rho_{X},x,\rho_{Y},y^\prime,0).
\end{equation}
We observe that for any fixed $y^{\prime\prime}$ the phase function $\phi(y^{\prime\prime})$, defined as 
\begin{equation}\label{eq:phiypp}
	[\phi(y^{\prime\prime})](\bx,\rho_{Y},y^\prime)=\varphi(\bx,\rho_{Y},y^\prime,y^{\prime\prime}),
\end{equation}
is equivalent to $\wtp$. Indeed, since $\partial_{y''}\scdY\varphi=0$, the differential $\scHess_Y\phi(y^{\prime\prime})$
has the same signature as $\scHess_{\BB^{s-e}}\wt\varphi$ and both parametrize the same Lagrangian with the same number of phase variables $(s-e)$.
Therefore, Theorem \ref{thm:equivphase} guarantees the existence of a family of diffeomorphisms $G(y^{\prime\prime}):(\bx,\rho_Y,y')\mapsto (\bx,g(\bx,\rho_Y,y',y^{\prime\prime}))$
such that, defining $ \wt{G}\colon (\bx,\by)=(\bx,\rho_Y,y^\prime,y^{\prime\prime})\mapsto(\bx, g(\bx,\rho_Y,y',y^{\prime\prime}), y^{\prime\prime})$,
\begin{equation}\label{eq:diffeoG}
	\wt{G}^*\varphi-\wt{\varphi}
\end{equation}
is smooth everywhere, and vanishes on $\cC_\wtp$ away from the corner by Remark~\ref{rem:strictness}.
Then we may express $I_\varphi(a)$ as $I_{\wt\varphi}(b)$, where
\begin{equation}\label{eq:excesssymb}
	b(\bx,\rho_Y,y^\prime)=\rho_Y^{-e} \int_{(-\eps,\eps)^e} e^{i(\wt{G}^*\varphi-\wt{\varphi})(\bx,\rho_Y,y^\prime,y^{\prime\prime})} (\wt{G}^*a)_{\red}(\bx,\rho_Y,y^\prime,y^{\prime\prime})\,\dd y''
\end{equation}
and
$$ (\wt{G}^*a)_{\red}(\bx,\by)\,\frac{|\dd y''|}{\rho_{\wtY}^{e} \cdot [h(\bx,\by)]^e}=(\wt{G}^*a)(\bx,\by).$$ 
Since $\wt{G}^*\varphi-\wt\varphi$ is smooth, $b$ is again an amplitude of order 
\begin{equation}\label{eq:order-excess}
    (\tilde{m}_e,\tilde{m}_\psi)=\left(m_e,m_\psi+e\right).
\end{equation}
Notice that at points in $\Cp$ away from the corner, $\wt{G}^*\varphi-\wt{\varphi}$ vanishes and hence \eqref{eq:excesssymb} reduces to
\begin{equation}\label{eq:excesssymbbis}
	b(\bx,\rho_Y,y^\prime)=\rho_Y^{-e} \int_{(-\eps,\eps)^e} (\wt{G}^*a)_{\red}(\bx,\rho_Y,y^\prime,y^{\prime\prime})\,\dd y''.
\end{equation}
\subsection{The order of a Lagrangian distribution}\label{ssec:order}
We will now obtain the definition of the order of $I_\varphi(a)$ which is invariant with respect to all the three steps 
described above.
\begin{lem}
    The numbers $\mu_\psi = m_\psi + s/2 + e/2$ and $\mu_e = m_e - s/2 +e/2$ remain constant under reduction of fiber-variables and elimination of excess.
\end{lem}
\begin{proof}
    Consider a Lagrangian distribution $A = I_\varphi(a)$ where $a$ has order $m_\psi, m_e$ and $\dim Y = s$ with excess $e$ and $r$ reduceable fiber variables.
    After the reduction of fiber, we obtain an amplitude $a'$ with order $m_e' = m_e -r/2, m_\psi' = m_\psi + r/2$ (cf. \eqref{eq:order-fiber}), with excess $e' = e$ and number of fiber variables $s' = s - r$.
    The elimination of excess yields an amplitude $a^\#$ with order $m_e^\# = m_e, m_\psi^\# = m_\psi + e$ (cf. \eqref{eq:order-excess}), excess $e^\# = 0$ and $s^\# = s - e$.
    It is now straightforward to check that
    \begin{alignat*}{2}
        m_\psi + s/2 + e/2 &= m_\psi' + s'/2 + e/2 & &= m_\psi^\# + s^\#/2+e^\#/2,\\
        m_e - s/2 + e/2 &= m_e' - s'/2 + e/2 & &= m_e^\# - s^\# / 2+e^\#/2.
    \end{alignat*}
\end{proof}
This shows that the tuple $(\mu_\psi, \mu_e)$ can be used to define the order of a Lagrangian distribution.

We still have the freedom to add arbitrary constants to both orders.
In order to choose these constants, we compare our class of Lagrangian distributions with H\"ormander's Lagrangian distributions and the Legendrian distributions of Melrose--Zworski~\cite{MZ}.
First, consider the Delta-distribution $\delta_0$, which is in the H\"ormander class $I^{d/4}$ and $\mu_\psi = d/2$. Therefore, we choose $m_\psi = \mu_\psi - d/4$ to obtain the same $\psi$-order for $\delta_0$.
Similarly, the constant function is a Legendrian distribution of order $-d/4$ and $\mu_e = 0$, and therefore we choose $m_e = \mu_e + d/4$.
Note that we use the opposite sign convention for the $m_e$-order then in \cite{MZ}.

\section{The principal symbol of a Lagrangian distribution}\label{sec:symb}
We will now define the principal symbol map $j^\Lambda_{m_e,m_\psi}$ on $I^{m_e,m_\psi}(X,\Lambda)$. Similarly to the classical theory,
it takes values in a suitable (density) bundle on $\Lambda$. This is coherent with the notion of principal symbol map $j_{m_e,m_\psi}$ for scattering
operators, see \cite{Melrose1,Melrose2}, as well as of principal part for classical $\SG$ symbols, see \cite{ES, Schulz},
which both provide smooth objects defined on $\Wt=\partial\scOverT^*X\supset\Lambda$. We adapt the construction in \cite{Treves} (see also \cite{Hormander4,HormanderFIO}), starting from the simplest case of local non-degenerate phase functions parametrizing $\Lambda$, up to the general case of local clean functions.

Let $\Lambda\subset\Wt$ be an $\sct$-Lagrangian, which on $B=X\times Y$
is locally parametrized  by a local non-degenerate phase function
$\varphi\in \rho_{Y}^{-1}\rho_X^{-1}\Sm(U)$, $U\subset B$. 
Let $a\in \rho_{Y}^{-m_{\psi}}\rho_X^{-m_{e}}\Sm\big(X\times Y, \scOmega^{1/2}(X)\times \scOmega^{1}(Y)\big)$ be
supported in $U$, and let $I_\varphi(a)$ be a (micro-)local representation of $u\in I^{m_e,m_\psi}(X,\Lambda)$ as a single oscillatory integral.

We now fix a $1$-density $\mu_X$ on $X$. Any choice of $1$ density $\mu_Y$ on $Y$ then trivializes the one-dimensional bundle $\Sm(X\times Y, \scOmega^{1/2}(X)\otimes\scOmega^{1}(Y))$, and any element is given by a multiple of $\rho_X^{-(d+1)/2}\rho_Y^{-s-1}\sqrt{\mu_X}\otimes\mu_Y$.
Any choice of coordinates $(\rho_Y,y)$ in $Y$ allows for us to express $\mu_Y$ locally as $\frac{\partial \mu_Y}{\partial (\rho_Y,y)}\,\dd\rho_Y\dd y$, meaning as having a smooth density factor with respect to the (local) Lebesgue measure. As such, we rewrite the amplitude $a\in\rho_Y^{-m_\psi}\rho_X^{-m_e}\Sm(X\times Y, \scOmega^{1/2}(X)\otimes\scOmega^{1}(Y))$ in any choice of local coordinates as
\begin{align}
	\label{eq:canampl}
	\rho_Y^{m_\psi}\rho_X^{m_e}a(\bx,\by)&=\ap(\bx,\by)\,\rho_X^{-(d+1)/2}\rho_Y^{-s-1}\sqrt{\mu_X}\dd\rho_Y\dd y.
\end{align}
for $\ap\in \Sm(X\times Y)$.

\subsection{Non-degenerate equivalent phase functions}\label{subs:ndg}
As above (cf. \eqref{eq:scdxexpl}), when $U$ is a neighbourhood of a point close to the boundary $\B$, we can there identify $\scd_Y\varphi$ with the map,
\[
	(\bx,\by)\mapsto\Phi(\bx,\by)=\big(-f(\bx,\by)+\rho_Y\partial_{\rho_Y}f(\bx,\by) \quad \partial_yf(\bx,\by)\big) \in\RR^s,
\]
locally well-defined on a neighbourhood of $C_\varphi$ within $U$.

In view of the non-degeneracy of $\varphi$, $\Phi$ has a surjective differential, so that we can consider the pullback of distributions $d_\varphi=\Phi^*\delta$, with $\delta=\delta_0\in\cD^\prime(\RR^s)$
the Dirac distribution, concentrated at the origin, on $\RR^s$ (cf. \cite[Ch. VI]{Hormander1}). More explicitly, choosing functions $(t_1, \dots, t_d)=:t$, which
restrict to a local coordinate system (up to the boundary) on $C_\varphi$, the pull-back $d_\varphi$ can be expressed locally as the density
\[
	d_\varphi=\left| \det\frac{\partial(t,\Phi)}{\partial(\bx, \by )}\right|^{-1}\dd t = \Delta_\varphi(t)\, \dd t.
\]
Consider another local non-degenerate
phase function $\wtp$ parametrizing $\Lambda$, 
defined on an open subset $\wtU\subset X\times \wtY$, such that $\wtp=F^*\varphi$, with a (local, fibered) 
diffeomorphism $F=\mathrm{id}\times g\colon X\times\wtY\to X\times Y$.
Since $F$ is a $\sct$-map, there exists a function $h \in \Sm(X\times Y)$ such that $(F^*\rho_Y)(\bx,\wby)=\rho_{\wtY}\cdot h(\bx,\wby)$.

As above, we identify $\scd_Y\wt\varphi$ with the map $\wtP$ and define $d_\wtp$ and $\Delta_\wtp(\widetilde{t})$ in terms of the functions ${\wt t}_j=F^*t_j$, which are
local coordinates on $C_\wtp$, provided $\wtU$ is small enough.

In the sequel, we show how objects defined in these two choices $(t,\varphi)$ and $(\wt t,\wt \varphi)$ are related. For that, we implicitly assume all objects evaluated at corresponding points $(\bx,\by)\in C_\varphi$ (parametrized by $t$) and $(\bx,\wt \by)=F(\bx,\by)\in C_{\wtp}$ (parametrized by $\wt t$).

\begin{lem}\label{lem:trDelta}
	The functions $\Delta_\wtp(\widetilde{t})$ and $\Delta_\varphi(t)$ are related by
    	\[
        		\Delta_\wtp(\widetilde{t}) = 
          	h(\bx,\by)^{s+1} \left| \det\frac{\partial g(\bx,\wby)}{\partial \wby }\right|^{-2} \, \Delta_\varphi(t(\widetilde{t})).
    	\]
\end{lem}
%
\begin{proof}[Proof of Lemma \ref{lem:trDelta}]
By direct computation, $\wtP$ and $\Phi$ are related by a matrix $M_{\Phi\wtP}$ via
\begin{equation}\label{eq:MPtP}
	\wtP(\bx,\wby)=  \Phi(F(\bx,\wby)) \cdot M_{\Phi\wtP}(\bx,\wby),
\end{equation}
where
\begin{align*}
	M_{\Phi\wtP}(\bx,\wby)&=
	\begin{pmatrix}
		[h(\bx,\wby)]^{-2} \dfrac{\partial\rho_Y}{\partial\rho_{\wtY}}(\bx,\wby) 
		& 
		[h(\bx,\wby)]^{-2} \rho_{\wtY}^{-1}\dfrac{\partial\rho_Y}{\partial \wty}(\bx,\wby)
		\\
		[h(\bx,\wby)]^{-1} \rho_{\wtY} \dfrac{\partial y}{\partial\rho_{\wtY}}(\bx,\wby) \rule{0mm}{9mm}
		&		
		[h(\bx,\wby)]^{-1} \dfrac{\partial y}{\partial \wty}(\bx,\wby)
	\end{pmatrix}
\end{align*}
and 
\[
	|\det M_{\Phi\wtP}(\bx,\wby)|=h(\bx,\wby)^{-s-1}\cdot\left| \det \frac{\partial g(\bx,\wby)}{\partial \wby} \right|.
\]
Differentiating \eqref{eq:MPtP}, we obtain, using that $\wtP(\bx,\by)=\Phi(F(\bx,\wby))=0$ on $C_{\widetilde{\varphi}}$,
\begin{equation}\label{eq:wpJ}
	\begin{aligned}
	\frac{\partial\wtP}{\partial (\bx,\wby)}(\bx,\wby)
	&={{}^t}\!M_{\Phi\wtP}(\bx,\wby) \cdot \frac{\partial(\Phi(F(\bx,\wby)))}{\partial(\bx,\wby)} 
	\\
	&={{}^t}\!M_{\Phi\wtP}(\bx,\wby) \cdot \left[\frac{\partial\Phi}{\partial(\bx,\by)}(F(\bx,\wby))\right]\cdot\frac{\partial F}{\partial(\bx,\wby)}(\bx,\wby).
	\end{aligned}
\end{equation}
Furthermore, we have
\[
	\frac{\partial\wtl}{\partial(\bx,\wby)}(\bx,\wby)=\left[\frac{\partial t}{\partial(\bx,\by)}(F(\bx,\wby))\right]\cdot\frac{\partial F}{\partial(\bx,\wby)}(\bx,\wby).
\]
Summing up, we find
\begin{equation}\label{eq:wtwpJ}
\begin{aligned}
	\frac{\partial(\wtl,\wtP)}{\partial(\bx, \wby)}(\bx, \wby) &= 
	\mathrm{diag}(\mathbbm{1}_{d}, {{}^t}\!M_{\Phi\wtP}(\bx,\wby)) \cdot \left[ \frac{\partial(t,{\Phi})}{\partial(\bx, {\by})}(F(\bx,\wby))\right] \cdot 
	\frac{\partial F}{\partial (\bx,\wby)}(\bx,\wby),
\end{aligned}
\end{equation}
which in turn implies, using $F=\id\times g$,
\begin{equation*}
	\Delta_{\wtp}(\widetilde{t}) = 
	\left| 	\frac{\partial(\wtl,\wtP)}{\partial(\bx, \wby)}(\bx, \wby) \right|^{-1} 
	= [h(\bx,\wby)]^{s+1} \left| \det \frac{\partial g(\bx,\wby)}{\partial \wby} \right|^{-2} \Delta_\varphi(t(\widetilde{t})),
\end{equation*}
as claimed.
\end{proof}

We define
\begin{equation}\label{eq:gamma}
	w_\varphi=(\rho_X^{-m_e}\rho_Y^{-m_\psi-(s+1)/2} \ap)|_{C_\varphi} \cdot \sqrt{|d_\varphi|},
\end{equation}
with $\ap$ given in \eqref{eq:canampl}, which is a half-density on (the interior of) $C_\varphi$.

%
To define $w_\wtp$ accordingly, we check that $I_\varphi(a)$ transforms under the action of $F$ as
\begin{align*}
	\int_Y e^{i\varphi}a&=\int_{\wtY} e^{i(F^*\varphi)(\bx,\wby)} 
	F^*\!\!\left[\rho_X^{-m_e}\rho_{Y}^{-m_\psi}\ap\,\rho_X^{-(d+1)/2}\rho_Y^{-s-1}\sqrt{\mu_X}\otimes \dd\rho_Y \dd y\right](\bx,\wby)
	\\
	&= \int_{\wtY} e^{i\wtp(\bx,\wby)}\rho_X^{-m_e}\rho_{\wtY}^{-m_\psi}\wtap(\bx,\wby)\,
	(\rho_X^{-(d+1)/2}\rho_{\wtY}^{-s-1}\sqrt{\mu_X}\otimes \dd\rho_\wtY \dd\wty),
\end{align*}
where
\begin{equation}\label{eq:trap}
	\wtap(\bx,\wby)=\ap(F(\bx,\wby)) h(\bx,\wby)^{-m_\psi-s-1} \left| \det \frac{\partial g(\bx,\wby)}{\partial \wby} \right|.
\end{equation}
We define, coherently with \eqref{eq:gamma}, $w_\wtp = \rho_X^{-m_e}\rho_{\wtY}^{-m_\psi-(s+1)/2}\wtap\sqrt{|d_{\wtp}|}$.
\begin{lem}\label{lem:wphi}
    The half-densities $w_\wtp$ and $w_\varphi$ are related by
    \begin{align*}
        w_\wtp =
    	F^*w_{\varphi}
    \end{align*}
    in (the interior of) $C_{\wtp}$.
\end{lem}

\begin{proof}
We obtain from \eqref{eq:trap} and Lemma \ref{lem:trDelta} that
\begin{align*}
    \wtap(\bx,\wby)\left|\Delta_\wtp(\widetilde{t})\right|^{1/2} =\ap(F(\bx,\wby)) h(\bx,\wby)^{-m_\psi-(s+1)/2} \left| \Delta_\varphi(t(\widetilde{t}))\right|^{1/2}.
\end{align*}
Then, using the local coordinates $t$ and $\wt t=F^*t$ introduced above, on $C_\wtp$ we find
\begin{align*}
    w_\wtp &= F^*\hspace*{-3pt}\left(\rho_X^{-m_e}\rho_Y^{-m_\psi-(s+1)/2} \ap\right) \left| \Delta_\varphi(t(\widetilde{t}))\right|^{1/2} \sqrt{\left|\dd \widetilde{t}\right|}\\
    &= 
    F^*\hspace*{-3pt}\left(\rho_X^{-m_e}\rho_Y^{-m_\psi-(s+1)/2} \ap 
    \left|\Delta_\varphi(t)\right|^{1/2}
\sqrt{|\dd t|}\right)= 
     F^*w_{\varphi}.
\end{align*}
\end{proof}
%

As a half-density valued amplitude, $w_\varphi$ is  
of order $(m_e,m_\psi-(s+1)/2)$, as shown by the computations above.
In accordance with the definition of the principal part (cf. Definition \ref{def:princpart}), we set
\[
	\wpp=\left.\left(\ap \cdot \sqrt{|d_\varphi|}\right)\right|_{\Cp}.
\]
As seen above, $\wpp$ transforms to 
$\wtpp$ under the pull-back via
$F$. Since $\lambda_\varphi$ is a local diffeomorphism $C_\varphi\to L_\varphi$, we can also consider
\[
	\alpha_\varphi=(\lambda_\varphi)_*(\wpp),
\]
which yields a local half-density on $\Lambda_\varphi$. The fact that, for the two
equivalent phase functions $\varphi$ and $\wtp$, we have 
$\lambda_\wtp=\lambda_{\varphi}\circ F$, together with
the transformation properties of $\wpp$, shows
that
\[
	\alpha_\wtp=\alpha_{\varphi}=\alpha,
\]
that is, $\alpha_\wtp$ and $\alpha_\varphi$ are equivalent local
representations of a half-density $\alpha$ defined on $\Lambda$,
in the local parametrizations $\Lambda_\wtp$ and $\Lp$,
respectively.

We now prove that the same holds true if $\wtp$ is merely a non-degenerate phase function equivalent to $\varphi$ in the sense of Definition \ref{def:phequiv}.
First, if we repeat the construction of $\sqrt{|d_\wtp|}$ described above, all the computations remain valid modulo terms, generated by $\wtP$, 
which contain an extra factor $\rho_X\rho_{\wtY}$. This is due to 
\begin{align*}
	F^*\varphi-\wtp&\in\Sm(\wtU)
	\\
	&\Leftrightarrow
	\rho_X^{-1}\rho_{\wtY}^{-1}
	\widetilde{f}(\bx,\wby)
	=\rho_X^{-1}\rho_{\wtY}^{-1}h(\bx,\wby)^{-1}
	(F^*f)(\bx,\wby)+g(\bx,\wby),
	g\in\Sm(\wtU)
	\\
	&\Leftrightarrow
	\widetilde{f}(\bx,\wby)
	=h(\bx,\wby)^{-1}(F^*f)(\bx,\wby)
	+\rho_X\rho_{\wtY} g(\bx,\wby),
	g\in\Sm(\wtU).
\end{align*}
Then, by rescaling $w_{\wtp}$ through multiplication by $\rho_X^{m_e}\rho_{\wtY}^{m_\psi+(s+1)/2}$ and then restricting $\wpp$ on $\mathcal{C}_\wtp$, 
such additional terms identically vanish. 

Moreover, by Lemma \ref{lem:phpequiv} and Remark \ref{rem:strictness},
we know that, in a neighbourhood $\wtU$ of any point in the interior of
$\mathcal{C}_\wtp^e$ or $\mathcal{C}_\wtp^\psi$, which does not intersect 
$\mathcal{C}_\wtp^{\psi e}$, it can be assumed, after passage to the principal parts, that $\wtp=F^*\varphi$ on $\mathcal{C}_\wtp\cap\partial \wtU$, see Section \ref{sec:moves}. It follows that the factor 
$\exp(i(F^*\varphi-\wtp))$, appearing in $\wtap$ (cf. \eqref{eq:oscintequiv}) also disappears, away from the corner, when 
restricting to the faces $\mathcal{C}_\wtp^e$ or $\mathcal{C}_\wtp^\psi$. 

Finally, we observe that $\wpp$ and $\wtpp$ are obtained 
as restrictions of smooth objects on $X\times Y$ and $X\times \wt Y$ to their respective boundaries. As such, their transformation behavior extends, by continuity, to the corner as well, 
producing smooth objects on $\Cp$ and $\mathcal{C}_\wtp$. 
By push-forward
through $\lambda_\wtp$ and $\lp$, we find again that 
$\alpha_\wtp=\alpha_\varphi=\alpha$ locally on
$\Lambda_\wtp=\Lambda_\varphi=\Lambda$.

\subsection{Non-degenerate phase functions, reduction of the fiber}\label{subs:ndgfbred}
We now consider a $\varphi$ such that reduction of fiber variables, see Section \ref{subs:fbred}, is possible. By the argument in Section \ref{subs:ndg}, we may then write $I_\varphi(a)=I_{\pred}(b)$
with $b$ from \eqref{eq:fiberredsymb}.
We now compare $\alpha_\varphi$ to the analogously defined half-density $\beta_\pred$. We can
replace the phase function $\varphi$ by the equivalent phase function given in \eqref{eq:wtp}, and this does not affect $\alpha_\varphi$. Hence we may assume that $\varphi$ is of the form $\varphi(\bx,\by)=\pred(\bx,\by')+\frac{1}{2}\rho_X^{-1}\rho_Y^{-1} \langle Qy'', y''\rangle$.

As such, we assume, in this splitting of coordinates, $C_\varphi\subset\{(\bx,\by',0)\}$. 
We find:
\begin{lem}
	\label{lem:dtransfnondeg}
	Under the identification $C_{\pred}\times\{0\}=C_\varphi$, we have
	\begin{equation*}
		\sqrt{|d_{\varphi}|}=|\det Q|^{-\frac{1}{2}} \sqrt{|d_{\pred}|}.
	\end{equation*}
\end{lem}
\begin{proof}
We compute
\begin{align*}
	\Phi(\bx,\by)
	&=\big(-\fred(\bx,\byp)+\rho_Y\partial_{\rho_Y}\fred(\bx,\byp)\quad\partial_{y^\prime}\fred(\bx,\byp)\quad 0\big)\\
	&\quad+\big(-\dfrac{1}{2}\langle Qy'', y''\rangle\quad 0 \quad \partial_{\ypp}Q(\ypp)\big)\\
	&=:(\Pred(\bx,\byp)\quad0)
	+\left(\Psi(\ypp)\quad Q\ypp\right)\in\RR^{s-r}\times\RR^{r}.
\end{align*}
Therefore,
\begin{align}
	\nonumber
	\frac{\partial(t,\Phi)}{\partial(\bx,\by)}(\bx,\by)&=
	\begin{pmatrix}
		\dfrac{\partial t}{\partial\bx}(\bx,\by) & \dfrac{\partial t}{\partial\byp}(\bx,\by) & \dfrac{\partial t}{\partial \ypp}(\bx,\by)
	\\
		\rule{0mm}{9mm}\dfrac{\partial\Pred}{\partial\bx}(\bx,\byp) & \dfrac{\partial\Pred}{\partial\byp}(\bx,\byp) & -\dfrac{1}{2}\dfrac{\partial \Psi}{\partial \ypp}(\ypp)
	\\
		\rule{0mm}{6mm}0 & 0 & Q
	\end{pmatrix}.
\end{align}
Consequently,
\begin{align*}
	\sqrt{|d_{\varphi}|}&=	\left|\det\frac{\partial(t,\Phi)}{\partial(\bx,\by)}\right|^{-1/2}_{C_\wtp}
	\sqrt{|dt|}
	\\
	&=\left|\det\frac{\partial (t,\Pred)}{\partial(\bx,\byp)}\right|^{-\frac{1}{2}}_{C_{\pred}}\cdot|\det Q|^{-\frac{1}{2}}
	\sqrt{|dt|}
	\\
	&=|\det Q|^{-\frac{1}{2}}	\sqrt{|d_{\pred}|}.
\end{align*}
\end{proof}
Notice that\footnote{Observe that $\mathfrak{a}_{red}$ is obtained by splitting of the density and weight factors in two steps.} $\mathfrak{a}=\mathfrak{a}_{\red}$. We compute, by \eqref{eq:princred1}, modulo amplitudes of lower order,
\begin{multline}
	\label{eq:orderb}
	b(\bx,\byp)= \rho_X^{-m_e+r/2}\rho_Y^{-m_\psi-r/2} |\det Q|^{-1/2} e^{i\frac{\pi}{4} \mathrm{sgn}(Q)} \ap(\bx,\byp,0)
	 \sqrt{\mu_X} (\rho_Y^{-(s-r+1)/2}|d\byp|).
\end{multline}
We observe that $b$ is an amplitude of order $(m_e-r/2, m_\psi+r/2)$ and find
\begin{align*}
	\bp(\bx,\byp)&=|\det Q|^{-1/2} e^{i\frac{\pi}{4} \mathrm{sgn}(Q)} \ap(\bx,\byp,0) + \cO\big(\rho_X \rho_Y\big),
\end{align*}
which implies, using Lemma \ref{lem:dtransfnondeg}, 
\begin{align*}
	\wredp&=\left.\left( \bp(\bx,\byp) \sqrt{|d_{\pred}|} \right)\right|_{\mathcal{C}_\pred}
	\\
	&=e^{i\frac{\pi}{4} \mathrm{sgn}(Q)}\left.\left( \mathfrak{a}(\bx,\by) \sqrt{|d_\varphi|}\right) 
	\right|_{\mathcal{C}_\wtp}
	\\
	&=e^{i\frac{\pi}{4} \mathrm{sgn}(Q)}\wpp.
\end{align*}
This, in turn, finally gives
\[
	\beta_{\pred}=(\lambda_{\pred})_*(\wredp)=e^{i\frac{\pi}{4} \mathrm{sgn}(Q)}\cdot(\lambda_\varphi)_*(\wpp)=e^{i\frac{\pi}{4} \mathrm{sgn}(Q)}\cdot\alpha_\varphi.
\]
\subsection{Clean phase functions, elimination of the excess}\label{subs:clnphf}
We now proceed with the last reduction step, namely, we consider a clean phase function and eliminate its excess. As in Section \ref{subss:elexcess}, we assume $Y=\BB^{s-e}\times(-\epsilon,\epsilon)^e$ with the fibers of $\Cp\rightarrow \Lp$ given by constant $(\bx,\rho_Y,y')$ and arbitrary 
$y''\in(-\epsilon,\epsilon)^e$.

Switching to the phase function $\wt{\varphi}$ in \eqref{eq:wtredex}, we may write $I_\varphi(a)=I_{\wt\varphi}(b)$ with $b$ defined in \eqref{eq:excesssymb}. We apply the
construction of the previous section, and obtain 
the density $\beta_{\wt{\varphi}}=(\lambda_{\wtp})_*\left(\bp\cdot \sqrt{|d_\wtp|}\right)_{\mathcal{C}_\wtp}$ from the data $(\wtp,b)$.

Alternatively, we may study the parameter dependent family of oscillatory integrals 
$I_{\phi(y^{\prime\prime})}(a(y^{\prime\prime}))$ with phase functions $\phi(y^{\prime\prime})$ defined in \eqref{eq:phiypp}
%
and amplitudes
\[
	a(y^{\prime\prime})\colon(\bx,\rho_Y,y^\prime)\mapsto \rho_Y^{-e}\,a(\bx,\rho_Y,y^\prime,y^{\prime\prime})=\rho_Y^{-e}\,a(\bx,\by),
\]
with corresponding principal parts $\ap(y^{\prime\prime})$. Since $\phi(y^{\prime\prime})$ is non-degenerate, we can define the 
parameter dependent family of half-densities on $\Lambda$
$$
	\alpha_\phi(y^{\prime\prime})=(\lambda_{\phi(y^{\prime\prime})})_*\left(\ap(y^{\prime\prime}) \cdot \sqrt{|d_{\phi(y^{\prime\prime})}|}\right)_{\mathcal{C}_{\phi(y^{\prime\prime})}},
$$ 
and finally set
\begin{equation}\label{eq:prsymbexc}
	\gamma_\wtp=\int_{(-\varepsilon,\varepsilon)^e}\alpha_\phi(y^{\prime\prime})\,dy^{\prime\prime}.
\end{equation}
\begin{prop}
    The half-densities on $\Lambda_{\wtp}=\Lambda_{\varphi}=\Lambda$ given by $\gamma_\wtp$ and $\beta_\wtp$ coincide.
\end{prop}
\begin{proof}
We consider the smooth family of diffeomorphisms $G(y^{\prime\prime})=\mathrm{id}\times g(y^{\prime\prime})$, depending on the parameter $y^{\prime\prime}$, 
involved in $\wt{G}$ from \eqref{eq:diffeoG}. Assuming the amplitudes $a(y^{\prime\prime})$ supported away from the corner points, we can suppose, as above, 
$G(y^{\prime\prime})^*\phi(y^{\prime\prime}) - \wtp=0.$ 
We now compute, using Lemma \ref{lem:CLFstarff} and the expression \eqref{eq:excesssymb}, together with the transformation properties of $\wpp$,
\begin{align*}
	\left(\bp_\wtp\cdot \sqrt{|d_\wtp|}\right)(\bx,\rho_Y,y^\prime)|_{\cC_\wtp} &= 
	\bp_\wtp(\bx,\rho_Y,y^\prime)|_{\cC_\wtp}\left| \det\frac{\partial(\widetilde{t},\wtP)}{\partial(\bx, \by^\prime )}\right|^{-\frac{1}{2}}_{\cC_\wtp}\sqrt{|d\widetilde{t}|}
	\\
	(\eqref{eq:trap} \Rightarrow) \qquad &=\int_{(-\varepsilon,\varepsilon)^e}\hspace*{-12pt}
	\ap(G(\bx,\by))|_{\cC_\wtp}\, 
	\left|\det\frac{\partial g}{\partial\by^\prime}(\bx,\by)\right|_{\cC_\wtp} \hspace*{-2pt} [h(\bx,\by)]_{\cC_\wtp}^{-m_\psi-s-1}\times
	\\
	&\phantom{=\int_{(-\varepsilon,\varepsilon)^e}}
	\times\left| \det\frac{\partial(\widetilde{t},\wtP)}{\partial(\bx, \by^\prime )}\right|^{-\frac{1}{2}}_{\cC_\wtp}\sqrt{|d\widetilde{t}|}\,d y^{\prime\prime}
	\\
	(\text{Lemma }\ref{lem:trDelta} \Rightarrow) \qquad &=\int_{(-\varepsilon,\varepsilon)^e}\hspace*{-12pt}
	G(y^{\prime\prime})^*\hspace*{-3pt}\left[\ap(\bx,\by)|_{\cC_{\phi(y^{\prime\prime})}}\hspace*{-3pt}
	\left| \det\frac{\partial(t,\Phi(y^{\prime\prime}))}{\partial(\bx, \by^\prime )}\right|^{-\frac{1}{2}}_{\cC_{\phi(y^{\prime\prime})}}\hspace*{-18pt}\sqrt{|dt|}\right]\hspace*{-2pt}d y^{\prime\prime}
	\\
	(\text{Def. of }d_{\phi(y^{\prime\prime})}\Rightarrow) \qquad &=\int_{(-\varepsilon,\varepsilon)^e}\hspace*{-12pt}
	G(y^{\prime\prime})^*\hspace*{-3pt}\left[\left(\ap(y^{\prime\prime})\cdot  \sqrt{|d_{\phi(y^{\prime\prime})}|}\right)(\bx,\rho_Y,y^\prime)\right]_{\cC_{\phi(y^{\prime\prime})}}\hspace*{-3pt}d y^{\prime\prime}.
\end{align*}
Applying $(\lambda_\wtp)_*$ to the left-hand side, we obtain $\beta_{\wtp}$.
To apply $(\lambda_\wtp)_*$ to the right-hand side, we first recall that $\wtp$ and $\phi(y^{\prime\prime})$ are equivalent by $G(y^{\prime\prime})$. Using again Lemma \ref{lem:CLFstarff} 
(see also Lemma \ref{lem:arrange}), this implies 
\begin{equation}\label{eq:wtpphequiv}
	\lambda_\wtp=\lambda_{\phi(y^{\prime\prime})}\circ G(y^{\prime\prime})\Rightarrow(\lambda_\wtp)_*=(\lambda_{\phi(y^{\prime\prime})})_*\circ G(y^{\prime\prime})_* .
\end{equation}
Since $\lambda_\wtp$ does not depend on $y^{\prime\prime}$, we can take it inside the integral and use \eqref{eq:wtpphequiv}, finally obtaining
\begin{align*}
	\beta_\wtp&= 
	(\lambda_\wtp)_*\left[\int_{(-\varepsilon,\varepsilon)^e}
	G(y^{\prime\prime})^*\hspace*{-3pt}\left[\left(\ap(y^{\prime\prime})\cdot  \sqrt{|d_{\phi(y^{\prime\prime})}|}\right)\right]_{\cC_{\phi(y^{\prime\prime})}}\hspace*{-3pt}d y^{\prime\prime}\right]
	\\
	&=\int_{(-\varepsilon,\varepsilon)^e}(\lambda_{\phi(y^{\prime\prime})})_*\circ G(y^{\prime\prime})_*\circ
	G(y^{\prime\prime})^*\hspace*{-3pt}\left[\left(\ap(y^{\prime\prime})\cdot  \sqrt{|d_{\phi(y^{\prime\prime})}|}\right)\right]_{\cC_{\phi(y^{\prime\prime})}}\hspace*{-3pt}
	d y^{\prime\prime}
	\\
	&=\int_{(-\varepsilon,\varepsilon)^e}(\lambda_{\phi(y^{\prime\prime})})_*
	\hspace*{-3pt}\left[\left(\ap(y^{\prime\prime})\cdot  \sqrt{|d_{\phi(y^{\prime\prime})}|}\right)\right]_{\cC_{\phi(y^{\prime\prime})}}\hspace*{-3pt}
	d y^{\prime\prime}=\int_{(-\varepsilon,\varepsilon)^e}\alpha_\phi(y^{\prime\prime})\,dy^{\prime\prime}=\gamma_\wtp.
\end{align*}
Extension to the corner points as in the previous subsections proves the claim.
\end{proof}

We already showed that the half-density $\alpha$ associated with $I_\varphi(a)$ is invariant under a change of equivalent non-degenerate phase functions. 
Together with the argument above, this also shows that the half-density $\gamma$ associated with
$I_\varphi(a)$ remains the same under the change of equivalent phase functions which are clean with the same excess.

\subsection{Principal symbol and principal symbol map} Let $u\in I^{m_e,m_\psi}(X,\Lambda)$. Consider any local representation of $u$, as introduced in Definition \ref{def:Lagdist},
with  clean phase function $\varphi$ with excess $e$ associated with $\Lambda$ and $a$ some local symbol density.
The arguments in the previous subsections show how to associate with these data a half-density $\gamma$,
defined on $\Lambda$. We also showed that switching to an equivalent phase function, as well as the elimination of the excess, do not change $\gamma$. The reduction of the fiber variables replaces $\gamma$
with $\gamma^\prime$ such that 
\[
	\gamma^\prime=e^{i\frac{\pi}{4}\mathrm{sgn}(Q)}\,\gamma,
\]
with $Q$ from \eqref{eq:wtp}. Let $\widetilde{\gamma}$ be the half-density defined by an integral representation $I_\wtp(\widetilde{a})$, with another
phase function $\widetilde{\varphi}$ associated with $\Lambda$. Then, similarly to \cite{Treves}, in general we have
\begin{equation}\label{eq:diffsigma}
		\widetilde{\gamma}=e^{i(\sigma-\widetilde{\sigma})\frac{\pi}{4}}\,\gamma,
\end{equation}
where $\sigma=\mathrm{sgn}\left(\,\rho_{Y}^{-1}\rho_X^{-1}\,\scHess_{Y}\varphi\right)$, and $\widetilde{\sigma}=\mathrm{sgn}\left(\,\rho_{\wtY}^{-1}\rho_X^{-1}\,\scHess_{\wtY}\wtp\right)$.
Denote by $\widetilde{r}$ the number of fiber variable for $\wtp$, $\widetilde{s}$ the dimension of $\wtY$ and $\widetilde{e}$ the excess of $\wtp$, and define the integer number
\[
	\kappa=\frac{1}{2}(\sigma-\widetilde{\sigma}-s+\widetilde{s}+e-\widetilde{e}).
\]
Then, \eqref{eq:diffsigma} is equivalent to
\begin{equation}\label{eq:coherence}
	i^\kappa e^{i(s-e)\frac{\pi}{4}}\,\gamma = e^{i(\widetilde{s}-\widetilde{e})\frac{\pi}{4}}\,\widetilde{\gamma}.
\end{equation}
We are then led to the following definition of principal symbol map.
\begin{defn}\label{def:prsymb}
	Let $u\in I^{m_e,m_\psi}(X,\Lambda)$. We define $\mathscr{I}(u)=\{(Y_j,\varphi_j)\}$ as the collection of manifolds and associated clean phase functions $(Y_j,\varphi_j)$
	locally parametrizing $\Lambda$, giving rise to local representations of $u$ in the form $I_{\varphi_j}(a_j)$.
	With each pair $(Y,\varphi)\in\mathscr{I}(u)$ we associate the half-density $\gamma$, as described in Subsection \ref{subs:clnphf},
	in such a manner that, for any other element $(\wtY,\wtp)\in\mathscr{I}(u)$, we have
	the coherence relation \eqref{eq:coherence} in $\lambda_\varphi(Y)\cap\lambda_\wtp(\wtY)$. We call the collection of half-densities $\{\gamma_j\}$, each one associated
	with $(Y_j,\varphi_j)\in\mathscr{I}(u)$, the \emph{principal symbol of $u$}, and write $j^\Lambda_{m_e,m_\psi}(u)=\{\gamma_j\}$.
\end{defn}

By an argument completely similar to the one in \cite{Treves}, we can prove the following result.

\begin{thm}
	Let $\Lambda$ be a $\sct$-Lagrangian on $X$. Then, the map
	\begin{equation}\label{eq:prsymbmap}
		j^\Lambda_{m_e,m_\psi}\colon I^{m_e,m_\psi}(X,\Lambda)\ni u\mapsto \{\gamma_j\}
	\end{equation}
	given in Definition \ref{def:prsymb} is surjective.
	Moreover, the null space of the map \eqref{eq:prsymbmap} is $I^{m_e-1,m_\psi-1}(X,\Lambda)$, and thus
	\eqref{eq:prsymbmap} defines a bijection 
	\[
		\text{classes in } I^{m_e,m_\psi}(X,\Lambda) / I^{m_e-1,m_\psi-1}(X,\Lambda) \mapsto \{\gamma_j\}.
	\]
	The image space of $j^\Lambda_{m_e,m_\psi}$ can be seen as $\Sm(\Lambda, M_\Lambda\otimes\Omega^{1/2})$, where $M_\Lambda$ is the Maslov bundle over $\Lambda$.
\end{thm}

%
\appendix

\section{Resolution of Lagrangian singularities near the corner}
\label{sec:blowup}
In this appendix, we show that $\Lambda^{\psi e}$ may be viewed as a Legendre manifold with respect to a (degenerate) contact form, well defined on the blow-up of the corner component $\Wpe$ of $\scOverT^*X$.

We have already stated that the forms 
\begin{align*}
\alpha^\psi:=\rho_\Xi^2\partial_{\rho_\Xi}\lrcorner\,\omega \quad \text{and} \quad 
\alpha^e:=\rho_X^2\partial_{\rho_X}\lrcorner\,\omega.
\end{align*} 
 are well-defined in the interior near the respective boundary face $\Wt^e$ or $\Wt^\psi$ and extend to it. The freedom in choosing the boundary defining function has as a consequence that these forms are merely well-defined up to a multiple by a positive function, however their contact structure at the boundary (which is all we need to characterize $\Lambda^\bullet$ as Legendrian) is independent of the choice of bdfs. Neither form extends to the corner component $\Wpe$. 
Instead of the rescaled $1$-forms, we now consider the non-rescaled forms
\begin{align*}
^\sct\alpha^\psi:=\rho_\Xi\partial_{\rho_\Xi}\lrcorner\,\omega\\
^\sct\alpha^e:=\rho_X\partial_{\rho_X}\lrcorner\,\omega
\end{align*} 
as sections of $\scT^*(\scT^*X^o)$. Then, these extend as \emph{scattering one forms} on $\scOverT^*X$, cf. \cite[(2.11)]{MZ}.
\begin{lem}
\label{lem:scatforms}
The forms $^\sct\alpha^\psi$ and $^\sct\alpha^e$ extend from $\scT^*X^o$ to scattering one-forms on $\scOverT^*X$.
In a particular choice of coordinates (see \cite{MZ} and Remark \ref{rem:compequiv}) they are given by
\begin{align*}
\,^\sct\alpha^e&=\frac{\dd \eta_1}{\rho_X\rho_\Xi}-\frac{\eta_1\dd \rho_\Xi}{\rho_X\rho_\Xi^2}+\eta''\frac{\dd x}{\rho_X\rho_\Xi},\\
\,^\sct\alpha^\psi&=\eta_1\frac{\dd \rho_X}{\rho_\Xi \rho_X^2}+\eta'' \frac{\dd x}{\rho_X\rho_\Xi}.
\end{align*}
Here, $\eta=(\eta_1,\eta'')$ are smooth functions of $(\rho_\Xi,\xi)$, $d-1$ of which may be chosen as coordinates.
\end{lem}
Again, the (scattering) contact structures of these forms, when restricted to the respective boundary faces, do not depend on the choice of bdf, since two choices of bdf only differ by positive factors. These forms $^\sct\alpha^\bullet$ will then vanish on $\Lambda^\bullet$, $\bullet\in\{e,\psi\}$, since one can identify the kernels of $^\sct\alpha^\bullet$ with that of $\alpha^\bullet$ by rescaling there. Furthermore, both $^\sct\alpha^\psi$ as well as $^\sct\alpha^e$ vanish when restricted to $\Lambda^{\psi e}$.
\begin{ex}
On $T^*\RRd$ with canonical coordinates $(x,\xi)$, this corresponds to both the forms 
$$\xi \cdot \dd x\quad \text{and}\quad -x\cdot \dd \xi$$
vanishing on the bi-conic (in $x$ and $\xi$) manifold with base $\Lambda^{\psi e}$, cf. \cite{CoSc2}.
\end{ex}
Hence, $\Lambda^{\psi e}$ is, in some sense, (scattering) isotropic.\footnote{Not with respect to the standard symplectic form, since it does not extend to the boundary, but to a rescaling of it.} We note, however, that the $\Lambda^{\psi e}$ is not Lagrangian with respect to any symplectic form on $\Wt^{\psi e}$, since
$$\dim(\Lambda^{\psi e})=d-2\neq d-1=\frac{\dim(\Wt^{\psi e})}{2}.$$
However, we may now blow-up the corner $\Wpe$ in $\scOverT(X)$ and consider the front face $\beta^{-1}(\Wt^{\psi e})$ in $[\scOverT(X);\Wt^{\psi e}]$, which is a $2d-1$ dimensional manifold, see Figure \ref{fig:Lpblowup}. Here, 
$$\beta: [\scOverT(X);\Wt^{\psi e}]\rightarrow \scOverT(X),$$
is the blow-down map.

%
\begin{figure}[ht!]
\begin{center}
\begin{tikzpicture}[scale=1.2]
  \node (A) at (2.5,3.5) {$\Wp$};
  \node (B) at (7,-0.5) {$\We$};
  \node (C) at (7.2,2.7) {$\beta^{-1}(\Wpe)$}; 
  \node at (4.9,1.65) {$\beta^{-1}(\Lambda^{\psi e})$};
  \node at (8,1.4) {$\partial(\beta^{-1}(\Lambda^{\psi e}))$};
  \draw[->] (7,1.4) .. controls (6,1) and (6,1) .. (5.6,1.25);
  \draw[->] (7,1.4) .. controls (6,1.8) and (5,2.25) .. (4.6,2.25);
  \draw[dotted,opacity=0.5] (1.5,2.25) -- (2.25,1.875);
  \draw[dotted,opacity=0.75] (2.25,1.875) -- (3,1.5) arc (180:270:1) -- (4,-1);
  \draw (3,1.5) -- (6,3);
  \draw (4,0.5) -- (7,2);  
  \draw[dotted,opacity=0.75] (4.5,3.75) -- (5.25,3.375);
  \draw[dotted,opacity=0.75] (5.25,3.375) -- (6,3) arc (180:270:1) -- (7,0.5);
  \fill[opacity=0.025] (6,3) arc (180:270:1) -- (4,0.5) arc (270:180:1) -- (6,3);
  \fill[opacity=0.01] (1.5,2.25) -- (3,1.5) -- (6,3) -- (4.5,3.75);
  \fill[opacity=0.04] (7,2) -- (4,0.5) -- (4,-1) -- (7,0.5);
  \draw[opacity=0.5,->] (A) -- (3.2,2.8);
  \draw[opacity=0.5,->] (B) -- (6.5,0.5);
  \draw[opacity=0.5,->] (C) -- (6.1,2);zzzzzzt
    \draw[thick] (4,2) .. controls (3,2.5) and (4,3) .. (3.5,3.25) node [right] {$\ \Lambda^\psi$};%
  \draw[thick] (4,2) arc (180:270:1);
\draw[thick] (5,1) .. controls (5,0.1) and (5.8,0.3) .. (5.7,-0.15) node [right] {$\Lambda^e$};%
\end{tikzpicture}
\caption{Resolution of $\Lpe$ near the corner}
\label{fig:Lpblowup}
\end{center}
\end{figure}

\begin{prop}
    The lift of the form
    \begin{align*}
        \alpha^{\psi e} = \frac{\rho_X \rho_\Xi}{2} ({ }^\sct\alpha^\psi + { }^\sct\alpha^e)
    \end{align*}
    to the blowup space $$[\scOverT^*X; \Wt^{\psi e}]\xrightarrow{\ \beta\ }{\scOverT^*X}$$
    restricts to a contact 1-form on the front face $\beta^{-1}\Wt^{\psi e}$.
    Moreover, $\beta^{-1}(\Lambda^{\psi e})$ is Legendrian with respect to $\alpha^{\psi e}$.
\end{prop}
\begin{proof}
We note that 
$$\alpha^{\psi e}=\rho_X\rho_\Xi\frac{1}{2}\left(\rho_X\partial_{\rho_X}+\rho_\Xi\partial_{\rho_\Xi}\right)
\lrcorner\,\omega.$$
In the special choice of coordinates of Lemma \ref{lem:scatforms}, we compute
$$\alpha^{\psi e}=\frac{1}{2}\eta_1\left(\frac{\dd \rho_X}{\rho_X}-\frac{\dd \rho_\Xi}{\rho_\Xi}\right)+\frac{1}{2} \dd\eta_1+\eta''\dd x$$
Now, smooth coordinates on the blow up of $\scOverT^*X$ along $\Wpe=\{\rho_X=\rho_\Xi=0\}$ are given by
\begin{equation}
\begin{cases}
\rho=\rho_X\quad \tau=\frac{\rho_\Xi}{\rho_X} \quad (x,\xi) & \rho_X>\rho_X\\
\rho=\rho_\Xi\quad \tau=\frac{\rho_X}{\rho_\Xi} \quad (x,\xi) & \rho_\Xi>\rho_X
\end{cases}
\end{equation}
In any case, $\beta^*\alpha^{\psi e}$ is of the form
$$\alpha^{\psi e}=\pm\frac{1}{2}\eta_1 \frac{\dd \tau}{\tau}+\frac{1}{2} \dd\eta_1+\eta''\dd x$$
Since $\tau=0$ marks the boundary of the front face $\beta^{-1}\Wt^{\psi e}$, $\alpha^{\psi e}$ is a $1$-form on the interior of $\beta^{-1}\Wt^{\psi e}$. 
Finally, $\alpha^{\psi e}$ vanishes on $\beta^{-1}\Lambda^{\psi e}$ since $^\sct\alpha^\psi$ and $^\sct\alpha^e$ vanish on $\Lambda^{\psi e}$.
\end{proof}
\bibliographystyle{amsalpha}

\end{document}